\documentclass[11pt]{amsart}
\usepackage[a4paper,centering,left=3.2cm,right=3.2cm]{geometry}
\usepackage{amsfonts,amscd,amssymb,amsmath,amsthm,mathrsfs,multirow,xcolor,lscape,longtable,pbox,lipsum,afterpage,graphicx,threeparttable,comment}
\usepackage[thinlines]{easytable}
\usepackage{enumerate}
\usepackage{todonotes}
\usepackage[colorlinks,linkcolor=blue,anchorcolor=blue,citecolor=blue,backref=page]{hyperref}
\usepackage{hyperref}

\allowdisplaybreaks
\newtheorem{proposition}{Proposition}[section]
\newtheorem{lemma}[proposition]{Lemma}
\newtheorem{corollary}[proposition]{Corollary}
\newtheorem{theorem}[proposition]{Theorem}
\newtheorem*{mainthm}{Main Theorem}
\newtheorem{definition}[proposition]{Definition}
\theoremstyle{definition}

\newtheorem{remark}[proposition]{Remark}
\newtheorem{conjecture}[proposition]{Conjecture}

\numberwithin{equation}{section}

\usepackage[all,cmtip]{xy}
\usepackage{tikz}
\usetikzlibrary{arrows} 
\usetikzlibrary{matrix}
\usetikzlibrary{calc}

\newcommand{\eqr}[1]{\mbox{(\ref{eq:#1})}}

\DeclareMathOperator{\Cl}{Cl}

\DeclareMathOperator{\mdeg}{mdeg}

\begin{document}
\title[Diameter bounds for untwisted classical groups]{Growth estimates and diameter bounds for untwisted classical groups}
\date{\today}

\author{Jitendra Bajpai}
\address{Department of Mathematics, Kiel University, 24118 Kiel, Germany}
\email{jitendra@math.uni-kiel.de}

\author{Daniele Dona}
\address{Alfr\'ed R\'enyi Institute of Mathematics, 1053 Budapest, Hungary}
\email{dona@renyi.hu}

\author{Harald Andr\'es Helfgott}
\address{Mathematisches Institut, Georg-August-Universit\"at G\"ottingen, 37073 G\"ottingen, Germany; IMJ-PRG\\ Université Paris-Cité\\Boîte Courrier 7012\\8 place Aurélie Nemours\\75205 Paris Cedex 13\\France}
\email{harald.helfgott@gmail.com}


\begin{abstract}
Babai's conjecture states that, for any finite simple non-abelian group $G$, the diameter of $G$ is bounded by $(\log|G|)^{C}$ for some absolute constant $C$. We prove that, for any untwisted classical group $G$ of rank $r$ defined over a field $\mathbb{F}_{q}$ 
with $q$ not too small with respect to $r$,
\begin{equation*}
\mathrm{diam}(G(\mathbb{F}_{q}))\leq(\log|G(\mathbb{F}_{q})|)^{408r^{4}}.
\end{equation*}  
This bound improves on results by Breuillard, Green, and Tao~\cite{BGT11}, Pyber and Szab\'o~\cite{PS16}, and, for $q$ large enough, also by Halasi, Mar\'oti, Pyber, and Qiao~\cite{HMPQ19}. Our approach is in several ways closer to that of preexistent work by Helfgott~\cite{Hel11}, in that we
give dimensional estimates (that is, bounds of the form $|A\cap V(\mathbb{F}_{q})|\ll|A^{C}|^{\dim(V)/\dim(G)}$, where $A$ is any generating set) for varieties $V$ of specific types, and work in the Lie algebra whenever possible. One of our main tools is a new, more efficient form of escape from subvarieties.
\end{abstract} 
\maketitle
\tableofcontents  


\section{Introduction}\label{se:intro}
\subsection{Background and main result}
Let $G$ be a finite group and $A$ be a subset of $G$. Let $\Gamma(G,A)$ denote the {\em Cayley graph} of $G$ with respect to $A$, that is, the graph $\Gamma=(V,E)$ whose vertices are the elements of $G$ and whose set of edges $E$ is $\{(g,ag): g\in G, a\in A\}$. It is easy to see that $\Gamma(G,A)$ is connected if and only if $G=\langle A \rangle$, i.e.\ $A$ generates $G$. 

Recall that the {\em diameter} of a connected graph $\Gamma$ equals $\max_{x,y\in V} \delta(x,y)$, where $\delta(x,y)$ is the length of the shortest path between $x$ and $y$ (and the {\em length} of a path is the number of edges in it).
We denote by $\mathrm{diam}(G,A)$ the diameter of $\Gamma(G,A)$. For $A$ symmetric
(i.e., such that $A^{-1} = \{x^{-1}: x\in A\}$ equals $A$) with $e\in A$, the diameter
of $\Gamma(G,A)$ equals the smallest integer $m$ such that $A^{m}=G$ (where $A^m=\{x_1 x_2 \dotsb x_m: x_1,\dotsc,x_m\in A\}$). Note that, for $A$ symmetric, $\Gamma(G,A)$ is symmetric.

If $G$ is finite, we can define its \textit{diameter} $\mathrm{diam}(G)$ as the maximum of $\mathrm{diam}(G,A)$, where $A$ runs through all symmetric sets of generators of $G$ with $e\in A$. Considering only these sets does not really entail a loss of generality: $\mathrm{diam}(G)$ is never much smaller than the \textit{directed diameter} $\mathrm{diam}^{+}(G)$ defined considering all generating sets $A$ (see~\cite[Cor.~2.3]{Bab06}).

The following is a well-known conjecture by Babai; it first appeared in~\cite{BS88}.
Here and from now on, $|S|$ denotes the number of elements in a set $S$.

\begin{conjecture}\label{conj:Babai}(\cite[Conj.~1.7]{BS92})
Let $G$ be a non-abelian finite simple group. Then
\begin{equation*}
\mathrm{diam}(G)\leq(\log |G|)^{C},
\end{equation*}
where $C$ is an absolute constant. 
\end{conjecture}

Partial results are known for certain classes of finite simple groups. By the Classification of Finite Simple Groups (CFSG), all but finitely many such groups are either groups of Lie type or $\mathrm{Alt}(n)$ for some $n$. For $\mathrm{Alt}(n)$, the first non-trivial diameter bound of $e^{O(\sqrt{n\log n})}$ appeared in~\cite{BS88} together with the conjecture itself. The current record is of the form $e^{O((\log n)^{4}\log\log n)}$, or $e^{(\log\log|G|)^{O(1)}}$~\cite{HS14}.

The first result for groups of Lie type, by the third author~\cite{Hel08}, settled Conjecture~\ref{conj:Babai} for $\mathrm{PSL}_{2}(\mathbb{F}_{p})$, $p$ prime. This result was followed by an analogous one for $\mathrm{PSL}_{3}(\mathbb{F}_{p})$ \cite{Hel11}: several ideas carry over to general groups of Lie type, and we use them in the present article as well. Today, bounds are known for all finite simple groups of Lie type. There are actually two kinds of strongest results to date, depending on the rank and the size of the base field. 

The first kind, due independently to Breuillard, Green, and Tao~\cite{BGT11} and Pyber and Szab\'o~\cite{PS16}, bounds the diameter of all finite simple groups of Lie type by $(\log|G|)^{C(r)}$, for some constant $C(r)$ depending only on the rank $r$. 
Some of the ideas in these papers -- and part of their overall strategy -- echo those in~\cite{Hel08} and \cite{Hel11}:
compare~\cite[Cor.~2.4]{BGT11} and~\cite[Thm.~68]{PS16} with~\cite[Key Prop.]{Hel08},
or the role of {\em dimensional estimates}, {\em escape} or {\em pivoting} (to be discussed soon) in all of these papers. The exponent $C(r)$ in~\cite{BGT11} is in principle computable but would be hard to make explicit, due to the use of ultrafilters. In~\cite{PS16}, $C(r)$ is given by a tower of exponentials of height depending on $r$: see the explicit work in \cite[App.~A]{PS16}; the tower emerges in~\cite[Lemma~79]{PS16}, ultimately affecting the bound in~\cite[Fact~17(e)]{PS16}.

The second kind of result for groups of Lie type, due to Halasi, Mar\'oti, Pyber, and Qiao~\cite{HMPQ19}, bounds the diameter by $q^{O(r (\log r)^{2})}$ for $r>1$, where $q$ is the size of the base field and $r$ is the rank of the group. The techniques used are inspired by proofs of diameter bounds in the permutation group setting: the paper builds on the work of Biswas and Yang~\cite{BY17}, who in turn adapted concepts from~\cite{BS92} to the Lie type case. This bound is tighter than that given by the first kind of results when $q$ is not too large with respect to $r$, where
the exact meaning of ``not too large'' depends on how fast the function $C(r)$ grows.

Our main result is a bound on the diameter of $G(\mathbb{F}_{q})$, with a mild technical assumption on the base field. Throughout the article, $q$ is any prime power.

\begin{mainthm}
Let $G<\mathrm{GL}_{N}$ be an untwisted classical group of type $\mathrm{SL}_{n}$, $\mathrm{SO}_{2n}^{+}$ ($n\geq 4$), $\mathrm{SO}_{2n+1}$ ($n\geq 3$), or $\mathrm{Sp}_{2n}$. Let $r$ be the rank of $G$, and assume that $\mathrm{char}(\mathbb{F}_{q})>N$. Then
\begin{equation*}
\mathrm{diam}(G(\mathbb{F}_{q}))\leq(2r)^{c_{1}r^{4}}(\log|G(\mathbb{F}_{q})|)^{c_{2}r^{2}}\leq(\log|G(\mathbb{F}_{q})|)^{c_{3}r^{4}},
\end{equation*}
where $c_{1},c_{2},c_{3}$ are absolute constants. For $q\geq e^{8r\log(2r)}$, we can take $(c_{1},c_{2},c_{3})=(930,57,408)$.
\end{mainthm}

The constants improve for large rank (see~\S\ref{se:asymp}): for $r\rightarrow\infty$ and $q\geq e^{(2+o(1))r\log r}$,
\begin{align*}
& \mathrm{diam}(G(\mathbb{F}_{q}))\leq\min\left\{r^{c_{1}r^{4}}(\log|G(\mathbb{F}_{q})|)^{c_{2}r^{2}} \, , \, (\log|G(\mathbb{F}_{q})|)^{c_{3}r^{4}}\right\}, \\
& (c_{1},c_{2},c_{3})=\left(\frac{576}{5}+o(1),\frac{48\log 4}{5}+o(1),\frac{96}{5}+o(1)\right).
\end{align*}

\subsection{Overview of the article}
Our strategy is akin to that of~\cite{BGT11} and~\cite{PS16}, but in some respect, it is closer to the older methods of~\cite{Hel11}.

We start with accessory results about the production of generic elements ({\em escape from subvarieties}). Such ideas appear already in~\cite{EMO05}, although they are probably older than that. Non-explicit predecessors of our results include~\cite[Prop.~4.1]{Hel11}, \cite[Lemma~3.11]{BGT11}, and \cite[Lemma~48]{PS16}. In Section~\ref{se:escape}, we make the argument both quantitatively and qualitatively more powerful: we pass through a $G$-invariance result (Proposition~\ref{pr:strong-escape}) that not only yields a classical escape (Corollary~\ref{co:escape}), but also a recipe that we shall use later to get rid of unwanted special points ({\em Ariadne's cookbook}, Corollary~\ref{co:ariadne}).

We pass then to our main ingredient, a {\em dimensional estimate}. The concept originates from the work of Larsen and Pink~\cite{LP11}, after whom it is sometimes named in the literature: given a variety $V\subseteq G$ and a field $K$, a dimensional estimate is a bound of the form $|A\cap V(K)|\leq C_{1}|A^{C_{2}}|^{\frac{\dim(V)}{\dim(G)}}$ for every set $A$ generating $G(K)$. Statements of a form similar to this one appear in~\cite[Thm.~4.1]{BGT11} and~\cite[Thm.~40]{PS16}, which are designed to apply to any variety $V$: the key feature in obtaining a diameter bound is the exponent $\frac{\dim(V)}{\dim(G)}$, while the sizes of $C_{1}$ and $C_{2}$ determine how good the bound is.

Rather than dealing with general $V$, we focus only on the two classes of varieties needed for the estimation of the diameter: conjugacy classes of regular semisimple elements in Section~\ref{se:de-cl(g)}, and non-maximal tori in Section~\ref{se:de-torus}. This represents somewhat a return to~\cite{Hel11}, both in our restriction to these varieties and in the fact that we do much of our work over the Lie algebra, as part of a more general decision to linearize whenever possible. Dimensional estimates for general subvarieties $V$ of connected almost simple linear algebraic groups have been given by the authors in~\cite{BDH24}, together with diameter bounds for classical groups that are based on these estimates.

All in all, by examining group actions and using Ariadne's cookbook, we manage to obtain $C_{1}, C_{2}$ of the form $e^{O(r^{O(1)})}$ for the dimensional estimates in both cases, where $r$ is the rank of $G$. For comparison, even when supplemented by explicit new routines, the inductive procedure of~\cite{BGT11} leads to a doubly exponential dependence of $C_{1}$ on $r$: see~\cite{BDH24}.

From here, in Section~\ref{se:diam} we follow a well-known route to diameter bounds: a {\em pivoting} argument. The idea comes from the study of the sum-product phenomenon over finite fields (dating back to~\cite{BKT04}), and it was introduced to the study of diameters already in~\cite{Hel08}. It appears
in the form we need in~\cite[\S 5]{BGT11} and~\cite[\S 8]{PS16}. Thanks to the pivoting argument and the bounds on $C_{1},C_{2}$, we first describe the growth of a symmetric set $A$ of generators of $G(\mathbb{F}_{q})$ in Theorem~\ref{thm:growth}, and then prove the Main Theorem.

\subsection{Other finite groups of Lie type}

Since~\cite{BGT11} and \cite{PS16} apply to all finite simple groups of Lie type, a natural question is whether our strategy applies to other families of algebraic groups. Babai's conjecture holds for groups of bounded rank, so the only open cases are the twisted classical groups $\mathrm{SU}_{n}$ and $\mathrm{SO}_{2n}^{-}$.

There seems to be no substantial obstacle. Dimensional estimates for these groups exist, even if they are doubly exponential in $r$: see~\cite[Thm.~6.2]{BDH24}. They rely on expressing $\mathrm{SU}_{n}(q)$ and $\mathrm{SO}_{2n}^{-}(q)$ as sets of $\mathbb{F}_{q}$-points of linear algebraic groups defined in a larger ambient space, and we should be able to follow a similar procedure in our situation; we just need to add more casework in Sections~\ref{se:de-cl(g)-prel} and~\ref{se:de-torus-linind}. The tools of Section~\ref{se:diam} regarding Weyl groups and representations can be used for all finite simple groups of Lie type, so the proof leading to the diameter bound should transfer with only some minor additional technical work.


\section{Basic tools and preliminaries}\label{se:prelim}

This section familiarizes the reader with the notations and concepts to be used throughout the article. Here we collect basic properties about varieties, degrees, classical groups, and $\mathbb{F}_{q}$-points.

We start by recalling two elementary and general facts about growth.

\begin{proposition}\label{pr:ruzsa}
Let $G$ be a finite group, and let $A=A^{-1}$ be a set inside $G$. Then
\begin{equation*}
\frac{|A^{k}|}{|A|}\leq\left(\frac{|A^{3}|}{|A|}\right)^{k-2}
\end{equation*}
for any $k\geq 3$.
\end{proposition}
\begin{proof}
This is a classical result due to Ruzsa; see for instance~\cite[(3.3)]{Helfgott-BAMS2015}.
\end{proof}

\begin{proposition}\label{pr:olson}
Let $G$ be a finite group, and let $A$ be a set of generators of $G$ containing $e$. Then either $A^{3}=G$ or $|A^{3}|\geq 2|A|$.
\end{proposition}

\begin{proof}
Use~\cite[Thm.~1]{Ols84} twice with $B=A$.
\end{proof}

\subsection{Varieties}

\subsubsection{Basic nomenclature}\label{sss:verybasic}

We will use classical language, meaning not ``as in Weil'' but ``as in the first chapter of standard, modern textbooks, such as~\cite{Hartshorne}, \cite{Mumford} or \cite{Shafv1}''. Since there are some dialectal differences between textbooks, we review some basic terms.

A {\em variety} $V$ in $m$-dimensional affine space $\mathbb{A}^{m}$ is defined by a set of $s$ equations of the form $P_{i}(x_{1},\ldots,x_{m})=0$ (for $1\leq i\leq s$), where all $P_i$ are polynomials and $s$ is any non-negative integer. Given a field $K$ to which all coefficients of all $P_i$ belong, we define
\begin{equation*}
V(K)=\{(k_{1},\ldots,k_{m})\in K^{m}:  P_{i}(k_{1},\ldots,k_{m})=0\;\;\forall 1\leq i\leq s\}.
\end{equation*}
Two sets of polynomials $\mathcal{P}=\{P_{i}\}_{i\leq s}$ and $\mathcal{Q}=\{Q_{j}\}_{j\leq t}$ with $P_{i},Q_{j}\in K[x_{1},\ldots,x_{m}]$ define the same variety $V$ if $\mathcal{P}$ and $\mathcal{Q}$ generate the same ideal in $\overline{K}[x_{1},\ldots,x_{m}]$, which we call $I(V)$. By the Nullstellensatz, $V=W$ if and only if $V(\overline{K})=W(\overline{K})$. A {\em subvariety} $W\subseteq V$ is a variety with $I(W)\supseteq I(V)$; $W$ is not necessarily defined over the same field as $V$.

The intersection of two varieties defined by sets of polynomials $\mathcal{P}=\{P_{i}\}_{i\leq s}$ and $\mathcal{Q}=\{Q_{j}\}_{j\leq t}$ is the variety defined by $\mathcal{P}\cup\mathcal{Q}$, while the union is defined by $\{P_{i}Q_{j}\}_{i\leq s, j\leq t}$. Our intersection is the {\em set-theoretic} or {\em reduced} one, as opposed to the scheme-theoretic one. We adopt this definition because it suits our needs better: for instance, it allows us to use \eqref{eq:bezout}, which does not hold for intersections in the other sense unless we impose more conditions; see further comments on the issue in \cite[\S 2.1.2]{BDH24}.

We speak of projective varieties analogously, viz., as defined within projective space $\mathbb{P}^m$ by homogeneous polynomial equations in $m+1$ variables.

In the {\em Zariski topology}, the closed sets are the sets $V(\overline{K})$ for all varieties $V$; since $V(\overline{K})$ determines $V$, we may speak of $V$ itself as being the closed set. Any topological term we use refers to the Zariski topology. The {\em Zariski closure} $\overline{S}$ of a set $S\in \mathbb{A}^m(\overline{K})$ is the smallest variety containing $S$.

A variety $V$ is {\em irreducible}\footnote{Many sources reserve the word {\em variety} for
what we call {\em irreducible variety}, and call {\em closed algebraic sets} what we call {\em varieties}.}
if it cannot be written as $V_{1}\cup V_{2}$ with $V_{1}\not\subseteq V_{2}$ and $V_{2}\not\subseteq V_{1}$. Every $V$ can be uniquely decomposed into a finite union of irreducible varieties not contained in each other, called the {\em irreducible components} of $V$. The {\em dimension} $\dim(V)$ of an irreducible variety $V$ is the largest $d$ for which we can write a chain of irreducible proper subvarieties $V_{0}\subsetneq V_{1}\subsetneq\ldots\subsetneq V_{d}=V$. We define the dimension of a non-irreducible variety $V$ to be the largest of the dimensions of its irreducible components. A variety is {\em pure-dimensional} when all its components are of the same dimension.

A {\em constructible set} in $\mathbb{A}^m$ is defined as the result of applying the union, intersection and complement operations a finite number of times (in any order) to varieties in $\mathbb{A}^m$. The image of a constructible set under a morphism is constructible (\cite[Cor I.8.2]{Mumford}; this is a theorem of Chevalley's).

We say that a statement is true for a {\em generic} point of a variety $V$ if the set of points
$x\in V(\overline{K})$ for which the statement is false all lie in a subvariety $W$ of $V$
with dimension lower than $\dim(V)$. (If $V$ is irreducible, then
$\dim(W)<\dim(V)$ if $W$ is a proper subvariety, i.e., if the complement of $W$ is non-empty.)
If we are to speak of, say, a generic hyperplane, we must first define a variety (or an open set thereof) whose points correspond to hyperplanes. This is easy: there is a bijection between 
hyperplanes in $\mathbb{P}^n$ and points of $\mathbb{P}^n$ itself,
in that $(y_0,y_1,\dotsc,y_n)$ defines the hyperplane
$x_0 y_0 + \dotsc + x_n y_n = 0$. More generally,
linear irreducible $d'$-dimensional varieties on $\mathbb{P}^n$ (i.e., $d'$-dimensional planes) correspond to points the {\em Grassmannian} $G(d'+1,n+1)$, an irreducible projective variety; $d'$-dimensional planes in $\mathbb{A}^n$ correspond to points in an open subset of the Grassmannian $G(d'+1,n+1)$.

A \textit{morphism} $f:\mathbb{A}^{m}\rightarrow\mathbb{A}^{m'}$ is given by an $m'$-tuple of polynomials $f_{i}$ on $m$ variables.
A morphism $f:X\rightarrow Y$ for $X\subseteq\mathbb{A}^{m},Y\subseteq\mathbb{A}^{m'}$
is given as the restriction of a morphism 
$g:\mathbb{A}^{m}\rightarrow\mathbb{A}^{m'}$ such
that $g(x)$ lies on $Y$ for every point $x\in X(\overline{K})$. 

Let us prove an auxiliary result we will soon need, using the concepts above.

\begin{lemma}\label{lem:auxy} Let $X\subset \mathbb{A}^m$ be irreducible. Let $f:X\to \mathbb{A}^n$ be a morphism. Assume that $\dim(\overline{f(X)})<\dim(X)$. Then, for $h$ a generic hyperplane in $\mathbb{A}^m$, $\overline{f(h\cap X)} = \overline{f(X)}$. 
\end{lemma}
\begin{proof}
Let $Y = \overline{f(X)}$.
The Grassmannian $G(m,m+1)$ parametrizes hyperplanes in
$\mathbb{A}^m$. Consider the subvariety $W$ of 
$G(m,m+1)\times X\times Y$ given by 
$$W=\{(h,x,y): \text{$x$ lies on $h$}, f(x)=y\}.$$
By Chevalley's theorem, the image of $\pi(W)$ of $W$ under the projection  
$\pi:G(m,m+1)\times X\times Y\to G(m,m+1)\times Y$ is constructible.
Now, for any point $y\in f(X)$, the preimage $f^{-1}(y)$ has
dimension at least $\dim(X)-\dim(\overline{f(X)})\geq 1$
(by \cite[Cor. to Thm.~I.8.2]{Mumford}), and hence intersects a generic hyperplane in $\mathbb{A}^m$. In other words,
the projection of $\pi(W)\cap (G(m,m+1)\times \{y\})$ to the first
coordinate (that is, to
$G(m,m+1)$) has Zariski closure equal to $G(m,m+1)$ for every $y$.
Since $\pi(W)$ is constructible, it follows that
$\pi(W)$ contains an open subset of $G(m,m+1)\times Y$. Hence,
there is an open subset $U$ of $G(m,m+1)$ such that, for every $h\in U$, the projection of $\pi(W)\cap (h\times Y)$ to the second coordinate (that is, to $Y$) has Zariski closure equal to $Y$. A moment's thought shows that
that projection equals $f(h\cap X)$.
\end{proof}

\subsubsection{Degrees}
Let $V\subseteq\mathbb{A}^{n}$ be a pure-dimensional variety with $\dim(V)=d$. By \cite[\S 3.1.2, Thm.]{DS98}, there is a non-empty open subset $U\subseteq G(n-d+1,n+1)$ such that, for every $(n-d)$-dimensional hyperplane $L$ in $\mathbb{A}^n$ corresponding to a point in $U$, the intersection of
$V$ with $L$ consists of a finite, fixed number of points. We call that number the {\em degree} $\deg(V)$ of $V$.

We can extend the definition of degree to general varieties $V$: $\deg(V)$ is the sum of the degrees of the pure-dimensional parts of $V$.
The bound $\deg(V_{1}\cup V_{2})\leq\deg(V_{1})+\deg(V_{2})$ holds for any $V_{1},V_{2}$ directly by definition. 
Since a generic $(n-d)$-dimensional plane does not intersect a variety
of dimension lower than $d$, we see that equality 
$\deg(V_{1}\cup V_{2})=\deg(V_{1})+\deg(V_{2})$
holds when any two irreducible components of $V_{1},V_{2}$ intersect properly. In particular, if $V$ is the union of irreducible components $V_{i}$, then $\deg(V)=\sum_{i}\deg(V_{i})$.

By \textit{B\'ezout's theorem}, for $V_{1},V_{2}$ pure-dimensional,
\begin{equation}\label{eq:bezout}
\deg(V_1\cap V_2)\leq \deg(V_{1})\deg(V_{2}).
\end{equation}
The statement above, coming from the work on intersection theory of Fulton and Macpherson, can be found for instance in \cite[Ex.~8.4.6]{Ful84}, \cite[(2.26)]{zbMATH03880868}, or \cite[\S II.3.2.2, Thm.]{DS98}; even more general and precise statements can be found in Fulton~\cite{Ful84}. By our definition of degree, we deduce easily from~\eqref{eq:bezout} that~\eqref{eq:bezout} itself holds for any $V_1,V_2$ not necessarily pure-dimensional.

Let $V$ be a variety defined by a single polynomial equation $P=0$ with $\deg(P)>0$. Then $\deg(V)\leq\deg(P)$, and equality holds if $P$ has no repeated factors. By B\'ezout, if $V$ is defined by many equations $P_{i}=0$, which means that $V=\bigcap_{i}V_{i}$ with $V_{i}$ defined by the single $P_{i}=0$, then $\deg(V)\leq\prod_{i}\deg(P_{i})$.

Adding coordinates to the ambient space does not change the degree: $\deg(V\times\mathbb{A}^{t})=\deg(V)$. The degree also behaves well under Cartesian products; if $V\subsetneq\mathbb{A}^{s},V'\subsetneq\mathbb{A}^{t}$ then $\deg(V\times V')\leq\deg(V)\deg(V')$ by B\'ezout ($V\times V'=(V\times\mathbb{A}^{t})\cap(\mathbb{A}^{s}\times V')$).

For a morphism $f:\mathbb{A}^{m}\rightarrow\mathbb{A}^{m'}$ given by an $m'$-tuple of polynomials $f_{i}$, we define the {\em maximum degree} of $f$ to be $\mdeg(f):=\max_{i}\deg(f_{i})$. This definition is not in the standard literature, but it is natural to work with this quantity when dealing with explicit bounds in algebraic geometry. 
For a morphism $f:X\rightarrow Y$, we define $\mdeg(f)$ to be the minimum of $\mdeg(g)$ over all $g:\mathbb{A}^{m}\rightarrow\mathbb{A}^{m'}$ with $g|_{X}=f$.

\begin{lemma}\label{le:zarimdeg}
Let $f:V\rightarrow\mathbb{A}^{n}$ be a morphism. Then
\begin{equation*}
\deg(\overline{f(V)})\leq\deg(V)\mdeg(f)^{\dim(\overline{f(V)})}.
\end{equation*}
\end{lemma}

\begin{proof}
We may assume without loss of generality that $V$ is pure-dimensional. We thank user ``Angelo'' on MathOverflow (question 63451) for the following proof. 

Applying Lemma~\ref{lem:auxy} repeatedly,
we reduce the dimension of $V$ by cutting it with generic hyperplanes until $\dim(V)=\dim(\overline{f(V)})=d$,
without changing $\overline{f(V)}$.
This operation does not increase $\deg(V)$, and in addition now a generic fibre of $f$ is finite and nonempty. Then, given a generic linear space $L$ of codimension $d$, we have $\deg(\overline{f(V)})=|L\cap\overline{f(V)}|\leq|f^{-1}(L)|$, where $f^{-1}(L)$ is finite (of cardinality bounded by its degree); $f^{-1}(L)$ is also the intersection of $V$ and of $d$ varieties each defined by one linear combination of the polynomials defining $f$, so by B\'ezout $\deg(f^{-1}(L))\leq\deg(V)\mdeg(f)^{d}$.
\end{proof}

In particular, since any projection map $\pi$ has maximum degree $1$, we have $\deg(\overline{\pi(V)})\leq \deg(V)$.

\subsection{Untwisted classical groups}\label{se:chev}

An \textit{untwisted classical group} $G$ is for us one of the following: $\mathrm{SL}_{n},\mathrm{SO}_{2n}^{+},\mathrm{SO}_{2n+1},\mathrm{Sp}_{2n}$. The same groups are called ``classical Chevalley'' in~\cite{Hel11}. Following the conventional setup of CFSG, as in (for instance)~\cite[\S 1.2]{Wil09}, we shall exclude the small cases $\mathrm{SO}_{2n}^{+}$ with $n\leq 3$, $\mathrm{SO}_{2n+1}$ with $n\leq 2$, and $\mathrm{Sp}_{2}$, since either their subquotients fail to yield families of finite simple groups or they introduce repetitions in the list of such families. We shall also exclude the case of even characteristic for $\mathrm{SO}_{2n}^{+}$ and $\mathrm{SO}_{2n+1}$: the first because $\mathrm{SO}_{2n}^{+}(\mathbb{F}_{q})$ is not connected for $q$ even, adding unnecessary complications (and $\mathrm{char}(\mathbb{F}_{q})=2$ is excluded by the hypotheses of the main theorem anyway); the second because $\mathrm{SO}_{2n+1}(\mathbb{F}_{q})\simeq\mathrm{Sp}_{2n}(\mathbb{F}_{q})$ (see for example~\cite[\S 3.4.7]{Wil09}).

Each of the untwisted classical groups $G$ can be defined as a subvariety of the corresponding space of matrices. Note that the representation of $\mathrm{SL}_{n}$ that we choose to use is not the most common one as a set inside $\mathrm{Mat}_{n}$ (although we switch to using the common one in the final section, since some computations improve such as in Proposition~\ref{pr:nonrs}).

\begin{definition}\label{de:ambient}
An untwisted classical group $G=G_{n}$ (of parameter $n$) over a field $K$ of characteristic $p\geq 0$ is one of the following varieties:
\begin{align*}
\mathrm{SL}_{n} & =\left\{\begin{pmatrix} x_{1} & 0 \\ 0 & x_{2} \end{pmatrix}\in\mathbb{A}^{4n^{2}}:\det(x_{1})=1,x_{1}x_{2}^{\top}=\mathrm{Id}_{n}\right\}, & & n\geq 2, \\
\mathrm{SO}_{2n}^{+} & =\{x\in\mathbb{A}^{4n^{2}}:\det(x)=1,x^{\top}M_{1}x=M_{1}\} & & \text{$p\neq 2$, $n\geq 4$,} \\
\mathrm{SO}_{2n+1} & =\{x\in\mathbb{A}^{(2n+1)^{2}}:\det(x)=1,x^{\top}M_{2}x=M_{2}\} & & \text{$p\neq 2$, $n\geq 3$,} \\
\mathrm{Sp}_{2n} & =\{x\in\mathbb{A}^{4n^{2}}:x^{\top}M_{3}x=M_{3}\} & & n\geq 2,
\end{align*}
with
\begin{align*}
M_{1} & =\begin{pmatrix} 0 & \mathrm{Id}_{n} \\ \mathrm{Id}_{n} & 0 \end{pmatrix}, &
M_{2} & =\begin{pmatrix} 0 & \mathrm{Id}_{n} & 0 \\ \mathrm{Id}_{n} & 0 & 0 \\ 0 & 0 & 1 \end{pmatrix}, &
M_{3} & =\begin{pmatrix} 0 & \mathrm{Id}_{n} \\ -\mathrm{Id}_{n} & 0 \end{pmatrix}.
\end{align*}
\end{definition}

The choice of $M_{1},M_{2},M_{3}$ above is standard: see \cite[Prop.~2.5.3(i)]{KL90}, \cite[Prop.~2.5.3(iii)]{KL90}, and \cite[Prop.~2.4.1]{KL90} respectively.

The inversion map $x\mapsto x^{-1}$ is a polynomial map whenever $\det(x)=1$, so Definition~\ref{de:ambient} indeed yields varieties, and as a matter of fact \textit{algebraic groups} since multiplication and inversion are also morphisms. Furthermore, in the representations above, the morphism $^{-1}:G\rightarrow G$ has maximum degree $1$: this happens in $\mathrm{SO}_{2n}^{+}$, $\mathrm{SO}_{2n+1}$ and $\mathrm{Sp}_{2n}$ because the inverse is related to the transpose, and in $\mathrm{SL}_{n}$ because of our choice of embedding. As a matter of fact, we have defined $\mathrm{SL}_{n}$ using this representation exactly because we wanted the inversion map to have maximum degree $1$.

Each $G$ in Definition~\ref{de:ambient} is connected \cite[Ex.~2.42-2.43-2.44]{Mil17}, hence irreducible \cite[Cor.~1.35]{Mil17}. Moreover, they are all \textit{semisimple} \cite[Ex.~19.19]{Mil17}. Throughout the rest of the paper, when dealing with an untwisted classical group $G$ we fix the following notation:
\begin{itemize}
\item $n\geq 2$ is the parameter appearing in the subscripts of $\mathrm{SL}_{n},\mathrm{SO}_{2n}^{+},\mathrm{SO}_{2n+1},\mathrm{Sp}_{2n}$;
\item $r$ is the \textit{rank} of $G$, i.e.\ the dimension of a maximal torus in $G$;
\item $N$ is the dimension of the affine space on which $G$ acts, i.e.\ $G<\mathrm{GL}_{N}$, so that in particular we have $G\subsetneq\mathbb{A}^{N^{2}}$ as a variety;
\item $\ell=\dim(G)/r$.
\end{itemize}

Each $G$ in Definition~\ref{de:ambient} is a variety, so it has a degree. We get $\deg(\mathrm{SL}_{n})\leq 2^{n^{2}}n$ from the equations defining it; we can improve the bound for $n=2$ to $\deg(\mathrm{SL}_{2})=2$, since the entries of $x_{2}$ are linear in the entries of $x_{1}$. For the other untwisted classical groups, precise computations of $\deg(G)$ are given in~\cite[Thm.~1.1 and Cor.~3.4]{BBBKR}; we give here some more manageable upper bounds. 
A simple way to do so is to use directly Definition~\ref{de:ambient}, but we can in fact do a bit better. By \cite[Prop.~4.2 and Rem.~4.3]{BBBKR}, we have
\begin{align*}
\deg(\mathrm{SO}_{N}^{+}),\deg(\mathrm{SO}_{N}) & =2^{N-1}P(N), & \deg(\mathrm{Sp}_{N}) & =P(N+1),
\end{align*}
where $P(k)$ is the number of ways to construct $\lfloor k/2\rfloor$ non-intersecting paths $\omega_{1},\ldots,\omega_{\lfloor k/2\rfloor}$ on the grid graph of vertices $\mathbb{Z}^{2}$, so that $\omega_{i}$ starts at $(2i-k,0)$, ends at $(0,k-2i)$, and is made of unitary steps that increase one of the coordinates. Dropping the ``non-intersecting'' condition, we can bound $P(k)$ by $\prod_{i=1}^{\lfloor k/2\rfloor}\binom{2(k-2i)}{k-2i}$, and then each binomial coefficient trivially by $2^{2(k-2i)}$. Therefore, we obtain
\begin{align*}
\deg(\mathrm{SO}_{2n}^{+}) & \leq 2^{2n-1}\prod_{i=1}^{n}2^{2(2n-2i)}=2^{2n^{2}-1}, \\
\deg(\mathrm{SO}_{2n+1}) & \leq 2^{2n}\prod_{i=1}^{n}2^{2(2n-2i+1)}=2^{2n^{2}+2n}, \\
\deg(\mathrm{Sp}_{2n}) & \leq\prod_{i=1}^{n}2^{2(2n-2i+1)}=2^{2n^{2}}.
\end{align*}

Table~\ref{ta:basicg} sums up values of relevant quantities for untwisted classical groups.
We choose to express our intermediate bounds mostly in terms of the rank $r$, so the entries below follow the same convention for simplicity. We write $\min(r)$ for the minimum possible rank for $G$ according to the definition.

\begin{table}[ht!]
\centering
\caption{Properties of untwisted classical groups of rank $r$.}
\begin{tabular}{|l|l|l|l|l|l|l|l|}
\hline
$G$ & \!\!\!\! & $\min(r)$ & $n$ & $N$ & $\dim(G)$ & $\deg(G)$ & $\ell$ \\ \hline\hline
$\mathrm{SL}_{n}$ & \!\!\!\! & $1$ & $r+1$ & $2r+2$ & $r^{2}+2r$ & \begin{tabular}{l}$\leq 2^{r^{2}+4r}$ \\ ($2$ for $r=1$) \end{tabular} & $r+2$ \\ \hline
$\mathrm{SO}_{2n}^{+}$ & \!\!\!\! & $4$ & $r$ & $2r$ & $2r^{2}-r$ & $\leq 2^{2r^{2}-1}$ & $2r-1$ \\ \hline
$\mathrm{SO}_{2n+1}$ & \!\!\!\! & $3$ & $r$ & $2r+1$ & $2r^{2}+r$ & $\leq 2^{2r^{2}+2r}$ & $2r+1$ \\ \hline
$\mathrm{Sp}_{2n}$ & \!\!\!\! & $2$ & $r$ & $2r$ & $2r^{2}+r$ & $\leq 2^{2r^{2}}$ & $2r+1$ \\ \hline
\end{tabular}
\label{ta:basicg}
\end{table}

Furthermore, for all $G$ the multiplication map $\cdot:G\times G\rightarrow G$ has maximum degree $2$ and the inversion map $^{-1}:G\rightarrow G$ has maximum degree $1$. 

We also list the order of the finite groups $G(\mathbb{F}_{q})$ for all $G$ of rank $r$ and all $q$ with $\mathrm{char}(\mathbb{F}_{q})>2$. See~\cite[\S 3.3.1, \S 3.5, \S 3.7.2]{Wil09}.
\begin{equation}\label{eq:numelgr}
\begin{aligned}
|\mathrm{SL}_{r+1}(\mathbb{F}_{q})| & =\frac{1}{q-1}\prod_{i=0}^{r}(q^{r+1}-q^{i}), & |\mathrm{SO}_{2r}^{+}(\mathbb{F}_{q})| & =q^{r(r-1)}(q^{r}-1)\prod_{i=1}^{r-1}(q^{2i}-1), \\
|\mathrm{SO}_{2r+1}(\mathbb{F}_{q})| & =q^{r^{2}}\prod_{i=1}^{r}(q^{2i}-1), & |\mathrm{Sp}_{2r}(\mathbb{F}_{q})| & =q^{r^{2}}\prod_{i=1}^{r}(q^{2i}-1).
\end{aligned}
\end{equation}

We often resort to working with {\em Lie algebras} $\mathfrak{g}$ instead of the groups $G$ themselves. For $G$ as in Definition~\ref{de:ambient}, the corresponding $\mathfrak{g}$ are the following varieties:
\begin{align*}
\mathfrak{sl}_{n} & =\left\{\begin{pmatrix} x_{11} & 0 \\ 0 & x_{22} \end{pmatrix}\in\mathbb{A}^{4n^{2}}:x_{11}=-x_{22}^{\top},\mathrm{tr}(x_{11})=0\right\}, \\
\mathfrak{so}_{2n}^{+} & =\{x\in\mathbb{A}^{4n^{2}}:x+M_{1}x^{\top}M_{1}^{-1}=0\} \\
 & =\left\{\begin{pmatrix} x_{11} & x_{12} \\ x_{21} & x_{22} \end{pmatrix}\in\mathbb{A}^{4n^{2}}:x_{11}=-x_{22}^{\top},x_{12}=-x_{12}^{\top},x_{21}=-x_{21}^{\top}\right\}, \\
\mathfrak{so}_{2n+1} & =\{x\in\mathbb{A}^{(2n+1)^{2}}:x+M_{2}x^{\top}M_{2}^{-1}=0\}, \\
 & =\left\{\begin{pmatrix} x_{11} & x_{12} & y_{1} \\ x_{21} & x_{22} & y_{2} \\ -y_{2}^{\top} & -y_{1}^{\top} & 0 \end{pmatrix}\in\mathbb{A}^{(2n+1)^{2}}:x_{11}=-x_{22}^{\top},x_{12}=-x_{12}^{\top},x_{21}=-x_{21}^{\top}\right\}, \\
\mathfrak{sp}_{2n} & =\{x\in\mathbb{A}^{4n^{2}}:x+M_{3}x^{\top}M_{3}^{-1}=0\} \\
 & =\left\{\begin{pmatrix} x_{11} & x_{12} \\ x_{21} & x_{22} \end{pmatrix}\in\mathbb{A}^{4n^{2}}:x_{11}=-x_{22}^{\top},x_{12}=x_{12}^{\top},x_{21}=x_{21}^{\top}\right\},
\end{align*}
where $M_{1},M_{2},M_{3}$ are as in Definition~\ref{de:ambient}, where the $x_{ij}$ are $n\times n$ matrices, and where the $y_{i}$ are $n\times 1$ matrices. The usual definition of $\mathfrak{so}_{2n}^{+},\mathfrak{so}_{2n+1}$ is presented in a slightly different way, with the identity matrix instead of $M_{1},M_{2}$, but the resulting Lie algebras are isomorphic: see \cite[App.~A.2--A.4]{Kir08}. All the $\mathfrak{g}$ above are linear varieties, with $\dim(\mathfrak{g})=\dim(G)$ and $|\mathfrak{g}(\mathbb{F}_{q})|=q^{\dim(G)}$ for $\mathrm{char}(\mathbb{F}_{q})>2$.

\subsection{Conjugacy classes}

Let $G$ be an untwisted classical group over a field $K$. For an element $g\in G(K)$, we define two important objects: the {\em conjugacy class} $\Cl(g)=\{xgx^{-1}:x\in G(\overline{K})\}$ and the {\em centralizer} $C(g)=\{x\in G:gx=xg\}$. By definition, $C(g)$ is a variety. We can also ask ourselves whether there is a variety $V$ such that $V(\overline{K})=\Cl(g)$: if so, we allow ourselves to use $\Cl(g)$ as a shorthand for $V$ itself, and make sense of the statement ``$\Cl(g)$ is a variety''. If $g\in G(K)$, the intersection $\Cl(g)\cap G(K)$ contains the $G(K)$-conjugacy class of $g$, meaning the set $\{xgx^{-1}:x\in G(K)\}$; the latter may however be smaller, which may result in additional technicalities (as in Lemma~\ref{le:escwecan}).

Recall that an element $g\in G(K)$ is {\em semisimple} if it is diagonalizable over $\overline{K}$ \cite[\S 9.b]{Mil17}, and {\em regular} if $C(g)$ is of minimal dimension among all the centralizers of the elements of $G$ \cite[\S 1.6]{Hum95a}. In general $\Cl(g)$ may not be a variety, but it is one if $g$ is regular semisimple.

\begin{proposition}\label{pr:clvar}
Let $G<\mathrm{GL}_{N}$ be an untwisted classical group of rank $r$ over a field $K$. Let $g\in G(K)$ be regular semisimple.

Then $\Cl(g)$ is a variety over $K$, with $\dim(\Cl(g))=\dim(G)-r=\left(1-\frac{1}{\ell}\right)\dim(G)$ and $\deg(\Cl(g))\leq r!\deg(G)\leq r! 2^{2 r^2 + 2 r}$.
\end{proposition}

\begin{proof}
Any untwisted classical group is connected semisimple, and for such groups $\Cl(g)$ is Zariski closed (and therefore a variety) if and only if $g$ is semisimple~\cite[\S 1.7]{Hum95a}. Since $g$ is regular, the dimension of the centralizer $C(g)$ is $r$ \cite[\S 2.3]{Hum95a}, which is $\frac{\dim(G)}{\ell}$ by definition. We have $\dim(\Cl(g))=\dim(G)-\dim(C(g))$ \cite[\S 1.5]{Hum95a}, so we get $\dim(\Cl(g))=\dim(G)-r=\left(1-\frac{1}{\ell}\right)\dim(G)$.

Finally, we study the degree of $\Cl(g)$.
Elements of $\mathrm{GL}_{N}$ are conjugate if and only if they have the same Jordan canonical form, which in the semisimple case means that they have the same eigenvalues. Thus, any two semisimple $g,g'\in G(K)$ are conjugate in $\mathrm{GL}_{N}$ if and only if the characteristic polynomials $p_{g}(t)=\det(t\mathrm{Id}_{N}-g)$ and $p_{g'}(t)=\det(t\mathrm{Id}_{N}-g')$ are identical.

Let $G=\mathrm{SL}_{n}<\mathrm{GL}_{2n}$ as given in Definition~\ref{de:ambient}, and denote by $\pi:\mathrm{Mat}_{2n}\rightarrow\mathrm{Mat}_{n}$ the restriction to the upper left corner. As we just said, for $g,g'\in G(K)$ semisimple, $\pi(g),\pi(g')$ are conjugate in $\mathrm{GL}_{n}$ if and only if $p_{\pi(g)}(t)=p_{\pi(g')}(t)$. The coefficient of $t^{i}$ in $p_{\pi(g')}(t)$ is of degree $n-i$ in the entries of $\pi(g')$ (and of $g'$), and if $g'\in G(K)$ then the coefficient of $t^{0}$ must be $(-1)^{n}\det(\pi(g'))=(-1)^{n}$. Thus, there is a variety $W$ of degree $\leq\prod_{i=1}^{n-1}(n-i)=(n-1)!$ such that $\pi(g')$ is conjugate to $\pi(g)$ in $\mathrm{GL}_{n}$ if and only if $\pi(g')\in W$: by the definitions of $W$ and $\pi$, we can think of $W$ as a variety inside $\mathrm{Mat}_{2n}$, and have $g'$ conjugate to $g$ via some element of the form $\begin{pmatrix} h & 0 \\ 0 & (h^{-1})^\top \end{pmatrix}$ (with $h\in\mathrm{GL}_{n}$) if and only if $g'\in W$. Lastly, if we have such an element, we can replace $h$ by $\det(h)^{-1/n}h$: we conclude that $g'$ is conjugate to $g$ via an element of $G$ if and only if $g'\in W$. Hence, $\Cl(g)$ has degree bounded by
\begin{equation*}
\deg(W\cap G)\leq\deg(W)\deg(G)\leq(n-1)!\deg(G)=r!\deg(G).
\end{equation*}
Since $g$ is defined over $K$, by the construction of $W$ it is clear that $\Cl(g)$ is defined over $K$.

Now let $G=\mathrm{SO}_{N}^{+},\mathrm{SO}_{N},\mathrm{Sp}_{N}$. Define the set
\begin{equation}\label{eq:heartsic}
\widetilde{\Cl}(g)=\{xgx^{-1}:x\in\mathrm{GL}_{N}(\overline{K})\}\cap G(\overline{K}).
\end{equation}
Since the characteristic polynomials of elements of $\mathrm{SO}_{N}^{+},\mathrm{SO}_{N},\mathrm{Sp}_{N}$ are symmetric, we only need to consider the coefficients of $t^{i}$ with $i\geq\frac{N}{2}$, as the other coefficients are equal to these ones up to sign. Thus, $\deg(\widetilde{\Cl}(g))\leq\deg(G)\prod_{i=\lceil N/2\rceil}^{N-1}(N-i)=\lfloor\frac{N}{2}\rfloor!\deg(G)$. It remains to prove that $\widetilde{\Cl}(g)=\Cl(g)$, or in other words that if $g,g'$ are conjugate by an element of $\mathrm{GL}_{N}(\overline{K})$ then they are conjugate by an element of $G(\overline{K})$. We resort to an elementary result of Freudenthal~\cite{Fre52}: for any algebraic group $G<\mathrm{GL}_{N}$ defined by $g^{\top}Eg=E$ for some $E$ symmetric or anti-symmetric, $g,g'\in G(\overline{K})$ are conjugate by an element of $\mathrm{GL}_{N}(\overline{K})$ if and only if they are conjugate by an element of $G(\overline{K})$. We can take $E=M_{1},M_{2},M_{3}$ for $G=\mathrm{O}_{N}^{+},\mathrm{O}_{N},\mathrm{Sp}_{N}$ respectively, and then pass from $\mathrm{O}_{N}^{+},\mathrm{O}_{N}$ to $\mathrm{SO}_{N}^{+},\mathrm{SO}_{N}$ by replacing $h$ by $\det(h)^{-1/N}h$, as we did for $\mathrm{SL}_{N}$. Hence $\Cl(g)$ has degree bounded by
\begin{equation*}
\deg(\widetilde{\Cl}(g))\leq\left\lfloor\frac{N}{2}\right\rfloor!\deg(G)=r!\deg(G).
\end{equation*}
Again, by construction $\Cl(g)$ is defined over $K$.

In all cases, we may bound $\deg(G)$ using Table~\ref{ta:basicg}.
\end{proof}

Alternatively, $\deg(\Cl(g))$ can be bounded using $\Cl(g)=\overline{\Cl(g)}$ and Lemma~\ref{le:zarimdeg}, with
$f(h) = h g h^{-1}$. Proposition~\ref{pr:clvar} is tighter, but the alternative method may be used for other semisimple groups $G$ for which~\cite{Fre52} does not apply.

\subsection{Tori}

A {\em torus} $T$ of an untwisted classical group $G$ is an algebraic subgroup of $G$ isomorphic to $(\mathrm{GL}_{1})^{m}$ for some $m\leq r$ (where $\mathrm{GL}_{1}$ is the same as the multiplicative algebraic group). A {\em maximal torus} of $G$ is a torus for which $m=r$.

In Section~\ref{se:de-torus} we shall discuss tori at length. For our initial construction, we will start by focusing on somewhat special tori, which we call {\em canonical}.

\begin{definition}\label{de:canontorus}
Let $G=G_{n}$ be an untwisted classical group. Write $\Delta(a_{i})$ for the diagonal matrix of entries $a_{i}$. The canonical maximal torus $T_{\max}$ and its Lie algebra $\mathfrak{t}_{\max}$ are given as follows.
\begin{enumerate}
\item If $G=\mathrm{SL}_{n}$, then
\begin{align*}
T_{\max} & =\left\{\begin{pmatrix} x_{11} & 0 \\ 0 & x_{22} \end{pmatrix}\in\mathbb{A}^{4n^{2}}:x_{11}=x_{22}^{-1}=\Delta(a_{i}),\prod_{i}a_{i}=1\right\}, \\
\mathfrak{t}_{\max} & =\left\{\begin{pmatrix} x_{11} & 0 \\ 0 & x_{22} \end{pmatrix}\in\mathbb{A}^{4n^{2}}:x_{11}=-x_{22}=\Delta(a_{i}),\sum_{i}a_{i}=0\right\}.
\end{align*}
\item If $G=\mathrm{SO}_{2n}^{+}$ or $G=\mathrm{Sp}_{2n}$, then
\begin{align*}
T_{\max} & =\left\{\begin{pmatrix} x_{11} & 0 \\ 0 & x_{22} \end{pmatrix}\in\mathbb{A}^{4n^{2}}:x_{11}=x_{22}^{-1}=\Delta(a_{i})\right\}, \\
\mathfrak{t}_{\max} & =\left\{\begin{pmatrix} x_{11} & 0 \\ 0 & x_{22} \end{pmatrix}\in\mathbb{A}^{4n^{2}}:x_{11}=-x_{22}=\Delta(a_{i})\right\}.
\end{align*}
\item If $G=\mathrm{SO}_{2n+1}$, then $T_{\max}$ (respectively $\mathfrak{t}_{\max}$) is as in the $\mathrm{SO}_{2n}^{+}$ case with one more $1$ (respectively one more $0$) in the last entry of the diagonal.
\end{enumerate}
A non-maximal torus $T$ of $G$ is canonical if its Lie algebra $\mathfrak{t}$ is the linear subspace of $\mathfrak{t}_{\max}$ given by $\sum_{i}\eta_{i}a_{i}=0$, for some $\eta_{i}\in\mathbb{Z}$ such that $\eta_{n}\neq 0$.
\end{definition}

In each $G$, the canonical torus $T_{\max}$ is contained in a variety $L$ of degree $1$ and dimension at most $N$. Specifically we can choose the variety of diagonal matrices:
\begin{equation}\label{eq:canoncontain}
L=\{(x_{ij})_{i,j}\in\mathrm{Mat}_{N}\,:\,\text{$x_{ij}=0$ when $i\neq j$}\}.
\end{equation}
Similarly, in each case $\mathfrak{t}_{\max}$ is the diagonal Cartan subalgebra of $\mathfrak{g}$.

\subsection{Existence of regular semisimple elements in $G(\mathbb{F}_{q})$}

For an untwisted classical group $G$ defined over $\mathbb{F}_{q}$, the set of regular semisimple elements is open and dense in $G$, and in fact in every maximal torus $T$; this holds more generally for any connected semisimple $G$ over any field \cite[\S\S 2.3--2.5]{Hum95a}. What is more, there exists a regular semisimple element over $\mathbb{F}_{q}$ itself.

\begin{proposition}\label{pr:rsq}
Let $G$ be an untwisted classical group defined over $\mathbb{F}_{q}$. Then $G(\mathbb{F}_{q})$ contains a regular semisimple element.
\end{proposition}

\begin{proof}
By \cite[Prop.~7.1.4 and Rem.~7.1.5]{DOR10}, the result holds for any $G$ reductive and $q\neq 2$, and if $G$ is an untwisted classical group there are only finitely many exceptions to check for $q=2$.

Alternatively, $G$ has at least a regular semisimple class invariant under the Frobenius automorphism $x\mapsto x^{q}$ (in fact, there is an odd number of such classes for any $G$ simply connected and semisimple, see \cite[Cor.~3.5]{Leh92}). Then, every such class has a regular semisimple element defined over $\mathbb{F}_{q}$ (true for every connected $G$, see \cite[2.7(a)]{SS70}).
\end{proof}

\subsection{$\mathbb{F}_{q}$-points on varieties}

For technical reasons, we are going to need the existence of an $\mathbb{F}_{q}$-point outside a proper subvariety. 
Knowing that we are dealing with a proper subvariety
is not always enough: we may have $V(\mathbb{F}_{q})=W(\mathbb{F}_{q})$ even if $V\subsetneq W$.

We need the following simple and explicit upper bound.
\begin{proposition}\label{pr:lwnotq}\cite[Prop.~2.3]{LR15}
Let $V\subseteq\mathbb{A}^{m}$ be a variety defined over $\overline{\mathbb{F}_{q}}$, of dimension $d$ and degree $D$. Then $|V(\overline{\mathbb{F}_{q}})\cap\mathbb{A}^{m}(\mathbb{F}_{q})|\leq Dq^{d}$.
\end{proposition}

We now apply this proposition. The constraints in the Main Theorem on the size of $q$ with respect to $r$ originate here. 

\begin{corollary}\label{co:langweil}
Let $G<\mathrm{GL}_{N}$ be an untwisted classical group of rank $r$, with $\ell=\frac{\dim(G)}{r}$. Let $V=\{x\in\mathbb{A}^{N^{2}m}:F(x)=0\}$ for some $m\geq 1$ and some polynomial $F$ defined over $\overline{\mathbb{F}_{q}}$, where $\mathrm{char}(\mathbb{F}_{q})>2$. Assume that $V\cap G^{m}\subsetneq G^{m}$.
\begin{enumerate}[(a)]
\item\label{co:langweilariadne} If $q\geq e^{8r\log(2r)}$, $m=\ell$, and $\deg(F)\leq 4(\ell-1)\dim(G)$, then $G^{\ell}(\mathbb{F}_{q})\setminus V(\overline{\mathbb{F}_{q}})\neq\emptyset$.
\item\label{co:langweilstick} If $r\geq 2$, $q\geq e^{3r\log(2r)}$, $m=\ell$, and $\deg(F)\leq 2\ell$, then $G^{\ell}(\mathbb{F}_{q})\setminus V(\overline{\mathbb{F}_{q}})\neq\emptyset$.
\item\label{co:langweilvgv} If $q\geq 15r^{2}$, $m=1$, and $\deg(F)\leq 2$, then $G(\mathbb{F}_{q})\setminus V(\overline{\mathbb{F}_{q}})\neq\emptyset$.
\end{enumerate}
\end{corollary}

\begin{proof}
We follow two routes, according to whether $G=\mathrm{SL}_{n}$ or $G\neq\mathrm{SL}_{n}$. 

First, let $G=\mathrm{SL}_{n}$: we may deal with this case directly via Proposition~\ref{pr:lwnotq}, but since $\deg(G)$ is large we follow another route. Denote by $\pi$ the projection to the upper left corners of the $m$ copies of $G$: in other words, if $x_{ijk}$ is the $(i,j)$-th entry of the matrix in the $k$-th copy of $G$, $\pi$ is the map forgetting every $x_{ijk}$ with either $i>n$ or $j>n$. Let $\tilde{G}=\pi(G)$, i.e.\ the more usual representation of $\mathrm{SL}_{n}$. Then, let $\tilde{F}$ be the polynomial defined taking $F$, setting $x_{ijk}=0$ inside it for all variables having either $i\leq n<j$ or $j\leq n<i$, and replacing all $x_{ijk}$ having $i,j>n$ with the corresponding $(n-1)\times(n-1)$ minor from the upper left corner: clearly for any $x\in G$ we have $F(x)=0$ if and only if $\tilde{F}(\pi(x))=0$. Thus, we can work inside $\tilde{G}$ and retrieve the result for $G$.

Call $d=\dim(G)=\dim(\tilde{G})$, and note that we have $\deg(\tilde{G})\leq n$. By \eqref{eq:numelgr}, $|G^{m}(\mathbb{F}_{q})|=|\tilde{G}^{m}(\mathbb{F}_{q})|>(q-1)^{dm}>q^{dm}\left(1-\frac{dm}{q}\right)$. If $\tilde{V}=\{x\in\mathbb{A}^{n^{2}m}:\tilde{F}(x)=0\}$, we also have $\dim(\tilde{V}\cap\tilde{G}^{m})\leq dm-1$ by the irreducibility of $G$, and $\deg(\tilde{V}\cap \tilde{G}^{m})\leq\deg(\tilde{F})n^{m}$ by B\'ezout. Finally, we have $\deg(\tilde{F})\leq(n-1)\deg(F)$ by construction. Thus, using Proposition~\ref{pr:lwnotq} we obtain
\begin{equation*}
|G^{m}(\mathbb{F}_{q})\setminus(V(\overline{\mathbb{F}_{q}})\cap\mathbb{A}^{N^{2}m}(\mathbb{F}_{q}))|>q^{dm}\left(1-\frac{dm+(n-1)\deg(F)n^{m}}{q}\right).
\end{equation*}
By Table~\ref{ta:basicg} and our choices of $r,q,m,F$ in the various cases, the bound above is non-negative.

Now let $G\neq\mathrm{SL}_{n}$. In this case we shall obtain a better bound by working in the Lie algebra $\mathfrak{g}$, which is a linear variety. If $G<\mathrm{GL}_{N}$, let $\lambda$ be the birational map $\lambda(x)=(\mathrm{Id}_{N}-x)(\mathrm{Id}_{N}+x)^{-1}$: $\lambda$ is well-defined in $\mathfrak{g}\setminus Z$, where $Z=\{x\in\mathbb{A}^{N^{2}}:\det(\mathrm{Id}_{N}+x)=0\}$, and by \cite{LPR06} we know that $\lambda:\mathfrak{g}\setminus Z\rightarrow G\setminus Z$ is an isomorphism between dense open sets of $\mathfrak{g}$ and $G$ (more than that, $\lambda$ is a {\em Cayley map} for $G$, as shown by Weil and referenced in \cite[Ex.~1.16]{LPR06}; it is known by~\cite[Thm.~1.31]{LPR06} that $\mathrm{SL}_{n}$ does not have a Cayley map for $n>3$, whence the difference of treatments here).

Let
\begin{equation*}
X=\{(x,y)\in\mathbb{A}^{N^{2}+1}\,:\,x\in\mathfrak{g},\,\det(\mathrm{Id}_{N}+x)y=1\},
\end{equation*}
and let $\pi:\mathbb{A}^{N^{2}+1}\rightarrow\mathbb{A}^{N^{2}}$ be the projection $\pi(x,y)=x$: we have $\deg(X)\leq N+1$ since $\mathfrak{g}$ is linear, and $X$ is isomorphic to $\mathfrak{g}\setminus Z$ via $\pi$, whose inverse is the rational map $x\mapsto(x,\det(\mathrm{Id}_{N}+x)^{-1})$. Recall that the {\em adjugate} of $x\in\mathrm{Mat}_{N}$ is the matrix $\mathrm{adj}(x)$ whose $(i,j)$-th entry is $(-1)^{i+j}M_{ji}$, where $M_{ji}$ is the $(j,i)$-th minor of $x$: in particular, $\mathrm{adj}$ is a polynomial map with $\mdeg(\mathrm{adj})\leq N-1$, and $\mathrm{adj}(x)=\det(x)x^{-1}$ whenever $x$ is invertible. Let $f:\mathbb{A}^{N^{2}+1}\rightarrow\mathbb{A}^{N^{2}}$ be defined by $f(x,y)=y(\mathrm{Id}_{N}-x)\mathrm{adj}(\mathrm{Id}_{N}+x)$: $f$ is a polynomial map with $\mdeg(f)\leq N+1$ and, since $f|_{X}=(\lambda\circ\pi)|_{X}$, $X$ is isomorphic to $G\setminus Z$ via $f$. Thus, $f$ is a birational morphism to $G$ irreducible, which means that $X$ itself is irreducible. Write $f^{m}$ for the map made of $m$ copies of $f$. Let $\tilde{V}=(f^{m})^{-1}(V)$, namely
\begin{equation*}
\tilde{V}=\{(x_{1},\ldots,x_{m},y_{1},\ldots,y_{m})\in\mathbb{A}^{(N^{2}+1)m}:F(f(x_{1},y_{1}),\ldots,f(x_{m},y_{m}))=0\},
\end{equation*}
which is of degree $\leq\mdeg(f)\deg(F)$. Since $X$ is irreducible, $X^{m}\cap\tilde{V}$ is a proper subvariety of $X^{m}$ of positive codimension and of degree $\leq(N+1)^{m+1}\deg(F)$.

It suffices to find an element in $X^{m}(\mathbb{F}_{q})\setminus(\tilde{V}(\overline{\mathbb{F}_{q}})\cap\mathbb{A}^{(N^{2}+1)m}(\mathbb{F}_{q}))$: if there is one such $x$, then $f(x)$ is an element of $G^{m}(\mathbb{F}_{q})\setminus(V(\overline{\mathbb{F}_{q}})\cap\mathbb{A}^{N^{2}m}(\mathbb{F}_{q}))$ because $f$ is defined over $\mathbb{F}_{q}$. The projection $\pi$ induces a bijection from $X(\mathbb{F}_{q})$ to $\mathfrak{g}(\mathbb{F}_{q})\setminus Z(\mathbb{F}_{q})$, since it preserves $\mathbb{F}_{q}$-points: hence, calling $d=\dim(G)$,
\begin{equation*}
|X(\mathbb{F}_{q})|=|\mathfrak{g}(\mathbb{F}_{q})|-|(\mathfrak{g}\cap Z)(\mathbb{F}_{q})|\geq q^{d}-q^{d-1}N,
\end{equation*}
using $|\mathfrak{g}(\mathbb{F}_{q})|=q^{d}$, B\'ezout, and Proposition~\ref{pr:lwnotq}. Thus, applying Proposition~\ref{pr:lwnotq} again and noting that $\dim(X)=d$, we obtain
\begin{align*}
|X^{m}(\mathbb{F}_{q})\setminus(\tilde{V}(\overline{\mathbb{F}_{q}})\cap\mathbb{A}^{(N^{2}+1)m}(\mathbb{F}_{q}))| & \geq q^{dm}\left(1-\frac{N}{q}\right)^{m}-q^{dm-1}(N+1)^{m+1}\deg(F) \\
 & >q^{dm}\left(1-\frac{Nm+(N+1)^{m+1}\deg(F)}{q}\right).
\end{align*}
By Table~\ref{ta:basicg} and our choices of $r,q,m,F$, the bound above is non-negative.
\end{proof}

Computations like the ones above can be performed whenever we want to show that there is some $\mathbb{F}_{q}$-point inside $W\setminus V$ for some given $V,W$ with $\dim(V)<\dim(W)$. Since in our case $W=G$, by Section~\ref{se:chev} we know exactly how many $\mathbb{F}_{q}$-points it has: for general $W$, one would need also a lower bound in addition to Proposition~\ref{pr:lwnotq}. To achieve this purpose, we can use a variant of the Lang-Weil bounds~\cite{LW54} -- say, by combining the explicit version contained in~\cite[Thm.~4.1]{GL02} with the definability result of~\cite[Prop.~3]{Hei83}.


\section{Escape from subvarieties}\label{se:escape}

In this section, we give results that allow us to find generic elements of $G(K)$ in $k$ steps, where $K$ is any field and $k$ is appropriately bounded. Such a result appears in a similar form already in~\cite{EMO05}; the idea itself, however, is older than that. In particular, a qualitative proof of Corollary~\ref{co:escape} can be found in~\cite{Hel11}; our main goal is to stress the quantitative details in order to optimize $k$.

\subsection{Escape via $G$-invariance}

Let us start with an intermediate result. The following lemma lets us keep under control a weighted version of the degree of variety when taking intersections.

\begin{lemma}\label{le:escinters}
Let $\{Z_{i}\}_{i\in I}$ be a (not necessarily finite) collection of subvarieties of an affine space $\mathbb{A}^{m}$, all of dimension $\leq d$ and degree $\leq D$. Then, if $\{W_{j}\}_{j}$ is the collection of irreducible components of $Z_{I}=\bigcap_{i\in I}Z_{i}$, we have
\begin{equation}\label{eq:escinters}
\sum_{j}\frac{\deg(W_{j})}{D^{d-\dim(W_{j})}}\leq D.
\end{equation}
\end{lemma}

\begin{proof}
Notice that, since $\mathbb{A}^{m}$ is Noetherian under the Zariski topology, any $Z_{I}$ is equal to some $Z_{I'}$ with $I'\subseteq I$ finite. We can assume then that $|I|<\infty$, and we proceed by induction on $|I|$; the case $|I|=1$ is trivial, since $\deg(Z_{i})\leq D$ is the sum of the degrees of its components.

Assume then that we have proved the statement for all $|I|=k$, and consider $I'=I\cup\{i\}$; for simplicity, let us call $X=Z_{I}$ and $Y=Z_{i}$, so that $Z_{I'}=X\cap Y$. Call $X_{1}$ the union of all the irreducible components of $X$ that appear in $X\cap Y$, and $X_{2}$ the union of all the ones that do not: hence, every component $W$ of $X\cap Y$ is either a component of $X_{1}$ or a component of $W'\cap Y$, where in turn $W'$ is a component of $X_{2}$.

If $W_{j}$ is in $X_{2}$, we have $W_{j}=W_{j}\cap Y$, and its contribution to~\eqref{eq:escinters} for $X\cap Y$ is unchanged. If $W_{j}$ is in $X_{1}$, the components of $W_{j}\cap Y$ are of dimension at most $\dim(W_{j})-1$, and the sum of their degrees is at most $D\deg(W_{j})$; thus, their contribution to~\eqref{eq:escinters} for $X\cap Y$ does not increase. Therefore, \eqref{eq:escinters} is proved for $X\cap Y$ as well, concluding the inductive step.
\end{proof}

Thanks to the lemma above, starting with a variety $V$ we are going to be able to find a $G$-invariant intersection of varieties $g\cdot V$, where $\cdot: G\times\mathbb{A}^{N}\rightarrow\mathbb{A}^{N}$ is some linear action. Note that $\bigcap_{g\in G}g\cdot V$ is trivially $G$-invariant: what we want is to find a small set of $g$ that can play the same role and such that these elements $g$ are quickly obtainable from a generating set.

\begin{proposition}\label{pr:strong-escape}
Let $K$ be a field, and let $G$ be a group acting linearly on $\mathbb{A}^{N}$ over $K$ (equivalently, a group with a homomorphism $\phi:G\rightarrow\mathrm{GL}_{N}(K)$); let $A\subseteq G$ be a set with $e\in A$ and $\langle A\rangle=G$. Consider a proper subvariety $V\subsetneq\mathbb{A}^{N}$ with $d=\dim(V)$ and $D=\deg(V)$.

Then, there are elements $e=a_{0},a_{1},\ldots,a_{m}\in A^{k}$ such that $\bigcap_{i=0}^{m}a_i\cdot V$ is $G$-invariant, with 
\begin{equation}\label{eq:mk}
m, k=\sum_{d'=0}^{d}D^{d-d'+1}\leq\begin{cases}
\left(1+\frac{1}{D-1}\right)D^{d+1} & \text{if } D\geq 2, \\
d+1 & \text{if } D=1. \\
\end{cases}
\end{equation}
\end{proposition}

With $\langle A\rangle=G$ we mean that $A$ generates $G$ {\em as a group}: in other words, every element $g\in G$ is the product of finitely many elements of $A\cup A^{-1}$, not necessarily of $A$. Nevertheless, the elements $a_{i}$ in the statement sit inside a power of $A$ itself. The condition $e\in A$ is only needed to simplify notation, as it implies that $A^{k_{1}}\supseteq A^{k_{2}}$ whenever $k_{1}\geq k_{2}$.

\begin{proof}
Our aim is to construct a variety $Z=a_{0}\cdot V\cap a_{1}\cdot V\cap\ldots\cap a_{m}\cdot V$, with $a_{i}\in A^{k}$, such that $Z$ is $G$-invariant.

We proceed in steps. At the $j$-th step, we assume to be working with a variety $X=X_{1}\cup X_{2}$ with the following properties: $X=\bigcap_{i=0}^{m_{j-1}}a_{i}\cdot V$ for $a_{i}\in A^{k_{j-1}}$, $X_{1},X_{2}$ are unions of distinct irreducible components of $X$, $\dim(X_{1})\leq d+1-j$, and $X_{2}$ is $G$-invariant; this is trivially the case at the first step, using $X=X_{1}=V$ and $m_{0}=k_{0}=0$. At the end of the $j$-th step, we want to have found a variety $X'$ with the same properties for $j$ instead of $j-1$, with $\dim(X'_{1})\leq\dim(X_{1})-1$ and $m_{j},k_{j}$ bounded in terms of $m_{j-1},k_{j-1}$: since the dimension of $X_{1}$ goes down, we must have $X'_{1}=\emptyset$ at the $(d+1)$-th step at the latest, and $Z$ is the final $X'=X'_{2}$.

Call $d_{1}=\dim(X_{1})$, let $\tilde{X}_{1},\tilde{X}_{2},\ldots,\tilde{X}_{s}$ be the $d_{1}$-dimensional irreducible components of $X_{1}$, and let $S\subseteq\{1,\ldots,s\}$ be the maximal set of indices such that $\bigcup_{s'\in S}\tilde{X}_{s'}$ is $G$-invariant; if there are no indices outside $S$ we skip to the end of the step, otherwise $\tilde{X}=\overline{X_{1}\setminus\bigcup_{s'\in S}\tilde{X}_{s'}}$ has some $d_{1}$-dimensional component. Since $A$ generates $G$ as a group, there must be some $a\in A\cup A^{-1}$ such that $\tilde{X}\cap a\cdot \tilde{X}$ has at least one $d_{1}$-dimensional component fewer than $\tilde{X}$, or else we would have had $S=\{1,\ldots,s\}$. If there is no such $a$ inside $A$ then there is no such $a$ inside $A\cup A^{-1}$ either, thus we may assume more strongly that $a\in A$. Then, observe that $X=\tilde{X}\cup\left(\bigcup_{s'\in S}\tilde{X}_{s'}\right)\cup X_{2}$, and the last two terms are $G$-invariant: if $Y=\left(\bigcup_{s'\in S}\tilde{X}_{s'}\right)\cup X_{2}$ then $Y=a\cdot Y$, and $\tilde{X}\cap Y=\tilde{X}\cap a\cdot Y$ does not have any $d_{1}$-dimensional component, or else it would contradict the definition of $\tilde{X}$. Therefore, $\tilde{X}\cap a\cdot \tilde{X}$ and $\tilde{X}\cap a\cdot X$ have the same number of $d_{1}$-dimensional components. From the above, since $X=\bigcap_{i=0}^{m_{j-1}}a_{i}\cdot V$, for some $i$ we must have that $\tilde{X}\cap aa_{i}\cdot V$ has at least one $d_{1}$-dimensional component fewer than $\tilde{X}$. Note on the other hand that, since $\left(\bigcup_{s'\in S}\tilde{X}_{s'}\right)\cup X_{2}$ is $G$-invariant and contained in $a_{i}\cdot V$, it is also contained in $aa_{i}\cdot V$. Hence, we can write $X\cap aa_{i}\cdot V=(\tilde{X}\cap aa_{i}\cdot V)\cup\left(\bigcup_{s'\in S}\tilde{X}_{s'}\right)\cup X_{2}$, with $\tilde{X}\cap aa_{i}\cdot V$ having one $d_{1}$-dimensional component fewer than $\tilde{X}$.

If $\tilde{X}\cap aa_{i}\cdot V$ has still some $d_{1}$-dimensional components, we repeat the argument with a new $a'\in A$ such that $(\tilde{X}\cap aa_{i}\cdot V)\cap a'(\tilde{X}\cap aa_{i}\cdot V)$ has one $d_{1}$-dimensional component fewer than $\tilde{X}\cap aa_{i}\cdot V$: this $a'$ must exist, or else the indices of the $d_{1}$-dimensional components of $\tilde{X}\cap aa_{i}\cdot V$ would have been in $S$. We iterate until there are no more $d_{1}$-dimensional components outside $\left(\bigcup_{s'\in S}\tilde{X}_{s'}\right)\cup X_{2}$. Therefore, if $\overline{a_{1}}\cdot V,\ldots,\overline{a_{l}}\cdot V$ are the translates of $V$ we used throughout the procedure, we can write $X'=X\cap\bigcap_{i'=1}^{l}\overline{a_{i'}}\cdot V$ as $X'_{1}\cup X'_{2}$, with $X'_{2}=\left(\bigcup_{s'\in S}\tilde{X}_{s'}\right)\cup X_{2}$ being $G$-invariant and with $X'_{1}$ defined as the union of all components of $\tilde{X}\cap\bigcap_{i'=1}^{l}\overline{a_{i'}}\cdot V$ not contained in $X'_{2}$. Both $X'_{1},X'_{2}$ are unions of irreducible components of $X'$, distinct among the two sets by definition, and $\dim(X'_{1})\leq d_{1}-1\leq d-j$ by construction. Thus, we rename $\overline{a_{i'}}$ as $a_{m_{j-1}+i'}$, and we observe that $a_{m_{j-1}+i'}\in A^{k_{j-1}+i'}$: the $j$-th step is concluded, with $m_{j}=m_{j-1}+l$ and $k_{j}=k_{j-1}+l$.

It remains to bound the final $m$ and $k$. Write $X^{(d')}$ for the union of the $d'$-dimensional components of $X$. In the process above, $l\leq s-|S|\leq s\leq\deg(X_{1}^{(d_{1})})\leq\deg(X^{(d_{1})})$. At any step, $X$ is the intersection of varieties of the form $a\cdot V$, which are all of dimension $d$ and degree $D$ like $V$ itself: then, by Lemma~\ref{le:escinters}, we have in particular $\deg(X^{(d_{1})})\leq D^{d-d_{1}+1}$. Since $d_{1}$ decreases at every step, in the worst case we have
\begin{equation*}
m, k\leq\sum_{d_{1}=0}^{d}D^{d-d_{1}+1},
\end{equation*}
and we are done.
\end{proof}

Proposition~\ref{pr:strong-escape} has as an immediate corollary the following escape result.

\begin{corollary}\label{co:escape}
Let $K$, $G$ and $A$ be as in Proposition~\ref{pr:strong-escape} (we may drop the assumption $\langle A\rangle=G$). Consider a proper subvariety $V\subsetneq\mathbb{A}^{N}$ and a point $x\in\mathbb{A}^{N}(K)$ with $\langle A\rangle x\not\subseteq V(K)$.

Then, if $d=\dim(V)$ and $D=\deg(V)$, there is an element $g\in A^{k}$ such that $gx\not\in V(K)$ with $k=(d+1)D^{d+1}$; more precisely, we can choose $k$ as defined in~\eqr{mk}.
\end{corollary}

\begin{proof}
Without loss of generality we may redefine $G$ to be $\langle A\rangle$ itself. By Proposition~\ref{pr:strong-escape}, there is a variety $Z=\bigcap_{i=0}^{m}g_{i}V$ that is $G$-invariant under the left multiplication action, and such that $g_{i}\in A^{k}$ with $k$ given in \eqref{eq:mk}. If $x\in Z(K)$ then $Gx\subseteq Z(K)\subseteq V(K)$, so $x$ cannot be contained in $Z(K)$; therefore we must have $x\not\in g_{i}V(K)$ for some $i$, and the result follows taking $g=g_{i}^{-1}$.
\end{proof}

Let us add here another useful application of Lemma~\ref{le:escinters}.

\begin{remark}\label{re:degesc}
If $Z$ is any intersection of varieties of dimension $\leq d$ and degree $\leq D$, then $\deg(Z)\leq D^{d+1}$. In fact, given the constraint of~\eqref{eq:escinters}, we maximize the degree of $Z$ when $\dim(Z)=0$, and the conclusion follows.
\end{remark}

\subsection{Ariadne's cookbook}

Proposition~\ref{pr:strong-escape} can also be exploited in a more subtle way. In a classical escape lemma, given a point $x$ whose orbit is not contained in a variety $V$, we manage to make $x$ escape quickly from $V$. For technical reasons however, we shall need a stronger statement: not only does a short recipe exist for $x$ to escape, but also there are few recipes to choose from, regardless of the point $x$ with which we start.

A mythical example may help. Theseus is going to be placed in the labyrinth of the Minotaur, from which an escape route exists. However, he does not know exactly where inside the labyrinth he will be starting from. Thus, Ariadne compiles for him a cookbook of escape recipes: no matter where he finds himself, at least one of the recipes is going to lead him outside.

This is the aim of Corollary~\ref{co:ariadne} below. Let $V$ be a variety from which it is possible to escape. Then there is a cookbook $\mathcal{F}$ of recipes $f$ such that: \eqref{co:ariadne-few} the cookbook is short; \eqref{co:ariadne-short} each recipe is short; \eqref{co:ariadne-escape} for any point of $V$, at least one recipe works. This is the meaning of the three properties in the statement.

For our future convenience, we actually work here with $s$-tuples of matrices. Thus for any finite set $S$, to avoid any ambiguity with $S^{s}$, we write $S^{\times s}$ for the set $S\times\ldots\times S$ ($s$ times).

\begin{corollary}\label{co:ariadne}
Let $K$, $G$ and $A$ be as in Proposition~\ref{pr:strong-escape}, where $G<\mathrm{GL}_{N}(K)$ acts on $\mathrm{Mat}_{N}\simeq\mathbb{A}^{N^{2}}$. Consider a proper subvariety $V\subsetneq\mathbb{A}^{N^{2}s}$ for some integer $s$.

Then, there are integers $m,k$ bounded as in~\eqref{eq:mk}, with $D=\deg(V)$ and $d=\min\{\dim(V),N^{2}-1\}$,
and there is a set $\mathcal{F}$ of invertible functions $f:\mathbb{A}^{N^{2}s}(K)\rightarrow\mathbb{A}^{N^{2}s}(K)$ with the following properties:
\begin{enumerate}[(1)]
\item\label{co:ariadne-few} $|\mathcal{F}|\leq(m+1)^{s}$;
\item\label{co:ariadne-short} for any $\vec{z}\in\mathbb{A}^{N^{2}s}(K)$ and any $f\in\mathcal{F}$ we have $f(\vec{z})\in(A^{k})^{\times s}\cdot\vec{z}$;
\item\label{co:ariadne-escape} for any $\vec{z}\in\mathbb{A}^{N^{2}s}(K)$, if $G^{s}\cdot\vec{z}\not\subseteq V(K)$ then there is some $f\in\mathcal{F}$ such that $f(\vec{z})\notin V(K)$.
\end{enumerate}
If $G=\mathrm{SL}_{N/2}(K)$ with $\mathrm{SL}_{N/2}<\mathrm{GL}_{N}$ as in Definition~\ref{de:ambient}, then inside the bound for $m,k$ we may use $D=\deg(V\cap G)$ and $d=\min\{\dim(V\cap G),\frac{1}{4}N^{2}-1\}$.
\end{corollary}

\begin{proof}
Take $\vec{z}=(z_{1},\ldots,z_{s})\in\mathbb{A}^{N^{2}s}(K)$. Our strategy is to construct the functions $f$ pointwise and componentwise, building for each $\vec{z}$ appropriate varieties in each of the $s$ copies of $\mathbb{A}^{N^{2}}$ from which to escape, and use Proposition~\ref{pr:strong-escape} to obtain the desired bounds on $m,k$ (we comment on the case $G=\mathrm{SL}_{N/2}(K)$ at the end).

We begin by building the varieties, which we do assuming that we have $G^{s}\cdot\vec{z}\not\subseteq V(K)$ as in the hypothesis of~\eqref{co:ariadne-escape}. Fix arbitrarily an ordering of the elements in each of the orbits of the action of $G^{s}$ inside $\mathbb{A}^{N^{2}s}(K)$, and let $\vec{x}$ be the first element of $(G^{s}\cdot\vec{z})\setminus V(K)$ according to this ordering. In the first step, we define the variety
\begin{equation*}
V_{1}=V_{1}(\vec{x})=\{v\in\mathbb{A}^{N^{2}}:(v,x_{2},\ldots,x_{s})\in V\}.
\end{equation*}
Applying Proposition~\ref{pr:strong-escape} to $V_{1}$, we obtain a set $X_{1}$ of $m_{1}+1$ elements $g_{1}\in A^{k_{1}}$ such that $\bigcap_{g_{1}\in X_{1}}g_{1}\cdot V_{1}$ is $G$-invariant, for some $m_{1},k_{1}$ depending on $\dim(V_{1}),\deg(V_{1})$. For each $g_{1}\in X_{1}$, we then define the variety
\begin{equation*}
V_{2}=V_{2}(g_{1},z_{1},\vec{x})=\{v\in\mathbb{A}^{N^{2}}:(g_{1}^{-1}\cdot z_{1},v,x_{3},\ldots,x_{s})\in V\},
\end{equation*}
and again applying Proposition~\ref{pr:strong-escape} to $V_{2}$ we obtain $X_{2}$ such that $\bigcap_{g_{2}\in X_{2}}g_{2}\cdot V_{2}$ is $G$-invariant, with analogous $m_{2},k_{2}$. Proceeding similarly, at the $j$-th step we obtain for each choice of $g_{1},\ldots,g_{j-1}$ a variety
\begin{align*}
V_{j} & =V_{j}(g_{1},z_{1},\ldots,g_{j-1},z_{j-1},\vec{x}) \\
 & =\{v\in\mathbb{A}^{N^{2}}:(g_{1}^{-1}\cdot z_{1},\ldots,g_{j-1}^{-1}\cdot z_{j-1},v,x_{j+1},\ldots,x_{s})\in V\}
\end{align*}
and some $X_{j},m_{j},k_{j}$ corresponding to $V_{j}$. The varieties $V_{j}$ and the sets $X_{j}$ may depend on the choices of $\vec{x},\vec{z},g_{1},\ldots,g_{j-1}$, but we do have a uniform bound for all $m_{j},k_{j}$. In fact, if we denote by $\pi_{j}$ the projection to the $j$-th copy of $\mathbb{A}^{N^{2}}$, we have by definition
\begin{align*}
V_{j} & =\pi_{j}(V\cap L), & L & =\{(g_{1}^{-1}\cdot z_{1},\ldots,g_{j-1}^{-1}\cdot z_{j-1})\}\times\mathbb{A}^{N^{2}}\times\{(x_{j+1},\ldots,x_{s})\},
\end{align*}
implying that $\deg(V_{j})\leq\deg(V)$ by B\'ezout and Lemma~\ref{le:zarimdeg}. Moreover either $V_{j}\subsetneq\mathbb{A}^{N^{2}}$, which yields $\dim(V_{j})\leq\min\{\dim(V),N^{2}-1\}$, or $V_{j}=\mathbb{A}^{N^{2}}$, for which we can trivially choose $X_{j}=\{e\}$. Therefore, all possible values of $m_{j},k_{j}$ are bounded by $m,k$ as in our statement.

Now we build $\mathcal{F}$. In order to have~\eqref{co:ariadne-few}, we define one function $f_{\vec{h}}$ for each tuple $\vec{h}\in\{1,\ldots,m+1\}^{\times s}$ (some functions may be equal to each other, whence the inequality). For each $\vec{z}$ with $G^{s}\cdot\vec{z}\not\subseteq V(K)$ we perform the construction described above: for convenience, we can redefine each $X_{j}$ to be an ordered $(m+1)$-tuple, whose entries we fill with the values in the original set and with additional copies of $e$ (if the set had fewer than $m+1$ elements). Then, we set
\begin{equation*}
f_{h_{1},\ldots,h_{s}}(\vec{z})=(g_{1,h_{1}}^{-1}\cdot z_{1},g_{2,h_{2}}^{-1}\cdot z_{2},\ldots,g_{s,h_{s}}^{-1}\cdot z_{s}),
\end{equation*}
where for each $1\leq j\leq s$ the element $g_{j,h_{j}}$ is the $h_{j}$-th entry of $X_{j}$ (we recall that $X_{j}$ depends on $\vec{z}$ and on the previous choices $g_{1,h_{1}},\ldots,g_{j-1,h_{j-1}}$). For the points $\vec{z}$ for which $G^{s}\cdot\vec{z}\subseteq V(K)$, we may set $f_{\vec{h}}(\vec{z})=\vec{z}$ for all $\vec{h}$.

By construction we have already~\eqref{co:ariadne-few} and~\eqref{co:ariadne-short}. Now we prove that $f_{\vec{h}}$ is invertible, namely that from $\vec{h}$ and $\vec{y}$ we can retrieve uniquely $\vec{z}$ such that $f_{\vec{h}}(\vec{z})=\vec{y}$. The element $\vec{z}$ must sit in the same orbit of $G^{s}$ as $\vec{y}$; if this orbit is contained in $V(K)$ then $\vec{z}=\vec{y}$ and we are done. Otherwise, from the orbit of $\vec{y}$ we know $\vec{x}$, thanks to the ordering of the orbit: this lets us obtain $X_{1}$, which depends only on $\vec{x}$, and then $g_{1,h_{1}}$, which depends on $X_{1}$ and $h_{1}$; this way we reconstruct $z_{1}$. Then $\vec{x},z_{1},g_{1,h_{1}}$ give us $X_{2}$, which by knowing $h_{2}$ in turn gives us $g_{2,h_{2}}$ and $z_{2}$, and so on.

It remains to show \eqref{co:ariadne-escape}: since $G^{s}\cdot\vec{z}\not\subseteq V(K)$, we have our fixed element $\vec{x}\in(G^{s}\cdot\vec{z})\setminus V(K)$. We prove by induction that for all $1\leq j\leq s$ there is some $h_{j}$ for which $g_{j,h_{j}}^{-1}\cdot z_{j}\notin V_{j}$. First of all, $G\cdot z_{1}\not\subseteq V_{1}(K)$ by definition of $\vec{x}$ and $V_{1}$: thus $z_{1}$ cannot be in the $G$-invariant intersection of all $g_{1,h}\cdot V_{1}$, which means that there is some $h_{1}$ with $g_{1,h_{1}}^{-1}\cdot z_{1}\notin V_{1}$. Suppose now that $g_{j,h_{j}}^{-1}\cdot z_{j}\notin V_{j}$: then by definition $x_{j+1}\notin V_{j+1}$, which implies that $G\cdot z_{j+1}\not\subseteq V_{j+1}(K)$. As before, $z_{j+1}$ cannot be in the intersection of all $g_{j+1,h}\cdot V_{j+1}$, giving $g_{j+1,h_{j+1}}^{-1}\cdot z_{j+1}\notin V_{j+1}$ for some $h_{j+1}$. The induction works, and at the last step we obtain $f_{\vec{h}}(\vec{z})\notin V$ for our tuple of indices $\vec{h}$, proving \eqref{co:ariadne-escape}.

If $G=\mathrm{SL}_{N/2}(K)$ as in Definition~\ref{de:ambient} we can restrict all our constructions to $G$, in particular passing from $V$ to $V\cap G$. In each copy of $\mathbb{A}^{N^{2}}$, this change results in a restriction to the upper left corner in terms of dimension, so that $\dim(V_{j})\leq\frac{1}{4}N^{2}-1$ when $V_{j}$ is nontrivial: the alternative bound on $m,k$ follows.
\end{proof}

\subsection{Escape \`a la Shitov}

Aside from the classical route taken in Corollary~\ref{co:escape}, one can also rely on a different kind of result for varieties of matrices: in brief, for a variety $V\subseteq\mathrm{Mat}_{N}$ and a set of elements $A\subseteq\mathrm{GL}_{N}(K)$, the idea is to reduce ourselves to $A', V'$ for $V'$ linear inside some larger $\mathrm{Mat}_{N'}$, and then use a theorem by Shitov \cite{Shi19} as an escape argument for linear matrix varieties.

\begin{proposition}\label{pr:shitov}
Let $K$ be a field, and let $A\subseteq\mathrm{GL}_{N}(K)$ be a set with $e\in A$. Let $V\subseteq\mathrm{Mat}_{N}$ be a subvariety defined by a finite set of polynomial equations $P_{l}(g_{ij})=0$ of degree $\leq D$, and suppose that $V(K)$ does not contain $\langle A\rangle$.

Then there exists an element $g\in A^{k}\setminus V(K)$ with $k<11D(N+1)^{D}\log N$.
\end{proposition}

\begin{proof}
Without loss of generality, we assume that $V$ is defined by one polynomial $P$: in fact, there must be one among the $P_{l}$ such that $\langle A\rangle$ is not contained in its set of zeros. Unlike in Corollary~\ref{co:escape}, this assumption does not weaken our final bound, since $\dim(V)$ is irrelevant in our passages.

First of all, we reduce to $P$ homogeneous of degree $D$: consider the embedding
\begin{align*}
\iota & :\mathrm{Mat}_{N}(K)\rightarrow\mathrm{Mat}_{N+1}(K), & \iota(M) & =\begin{pmatrix} 1 & 0 \\ 0 & M \end{pmatrix},
\end{align*}
which preserves matrix multiplication, and notice that the escape of $A$ from $V$ takes as many steps as the escape of $\iota(A)$ from the variety $V'$ defined by the homogenization of $P$ through the variable $g_{00}$ corresponding to the top left coordinate.

Then we reduce ourselves to the case of a linear variety. To do so, consider the linearization map $\rho_{D}:\mathrm{Mat}_{N}\rightarrow\mathrm{Mat}_{N^{D}}$ defined as follows: if the rows and columns of $\mathrm{Mat}_{N^{D}}$ are indexed by tuples $\vec{i}=(i_{1},\ldots,i_{D}),\vec{j}=(j_{1},\ldots,j_{D})$ with $1\leq i_{k},j_{k}\leq N$, the matrix $M=(m_{ij})_{i,j}\in\mathrm{Mat}_{N}$ is sent to $\rho_{D}(M)$ whose $(\vec{i},\vec{j})$-th entry is $m_{i_{1}j_{1}}\cdot\ldots\cdot m_{i_{D}j_{D}}$. Then, escaping from $V$ in $\mathrm{Mat}_{N}$ is the same as escaping from $V'$ in $\mathrm{Mat}_{N^{D}}$ defined by a linear polynomial $P'$ obtained from $P$ by replacing each monomial $\prod_{k\leq D}g_{i_{k}j_{k}}$ with the corresponding variable $g_{\vec{i}\vec{j}}$.

Thus, applying $\iota$ and $\rho_{D}$ we have reduced ourselves to escaping from a linear variety $V'$ of $\mathrm{Mat}_{N'}$ with $N'=(N+1)^{D}$. Finally, we can reinterpret~\cite[Thm.~3]{Shi19} as follows: if $\langle A\rangle\not\subseteq V'(K)$ with $V'$ linear in $\mathrm{Mat}_{N'}$, then there must be an element of $A^{2N'\log_{2}N'+4N'}$ outside $V'(K)$; therefore we escape from this $V'$ (and the original $V$) in a number of steps bounded by
\begin{equation*}
\frac{2}{\log 2}(N+1)^{D}\log(N+1)^{D}+4(N+1)^{D}<11D(N+1)^{D}\log N. \qedhere
\end{equation*}
\end{proof}

By comparison, under the same conditions as in the statement above, Corollary~\ref{co:escape} would give $k\leq N^{2}D^{N^{2}}$, and using Corollary~\ref{co:escape} after linearization would give $k\leq N^{2D}$. 


\section{Estimates for conjugacy classes}\label{se:de-cl(g)}

In this section, we establish an upper bound for $|A^{t}\cap V(K)|$ where the variety $V=\Cl(g)$ is the conjugacy class of a regular semisimple element $g\in G(K)$. This is one of our key tools to prove the main result about the diameter bound of $G$. Recall that by Proposition~\ref{pr:clvar} $\Cl(g)$ is indeed a variety, defined over $K$ and whose dimension and degree are known. Since $\Cl(g)$ is also the image of an irreducible variety (i.e., $G$ itself) via the morphism $x\mapsto xgx^{-1}$, $\Cl(g)$ is irreducible as well.

Here are the key ideas of the proof. Recall that we set $\ell=\frac{\dim(G)}{r}$ throughout the paper. First, we build a map $f:V^{\ell}\rightarrow G^{\ell-1}$ with two properties: $f(A^{t})\subseteq A^{k(t)}$ for some controlled exponent $k(t)$, and all the fibres through $f$ outside of an exceptional subvariety $E$ of $V^{\ell}$ have dimension $0$. If we had $E=\emptyset$, we would have $|A^{t}\cap V(K)|^{\ell}\leq C|A^{k(t)}|^{\ell-1}$, where $C$ is a bound on the number of elements of $A^{t}$ on a single fibre, and we would be done. We still have to deal with $E$. In a usual dimensional estimate, where $V$ is a general variety, we can set up an induction process on the dimension and treat $E$ by inductive hypothesis: this is what has been done in \cite{BGT11} and \cite{PS16}, and also what we do in \cite{BDH24}. However, we can do better: since our $V$ is the orbit of an action, namely conjugation, we can escape from $E$ using the techniques of Section~\ref{se:escape} up to increasing $C$ and $k(t)$ slightly (depending on $\deg(E)$).

The plan above unfolds as follows. In Proposition~\ref{pr:clg} we define a suitable map $f$, proving that its generic fibre is $0$-dimensional. In Lemma~\ref{le:de} we describe the exceptional $E$ collecting all large fibres: our $E$ is larger than strictly necessary, but we gain in the process because its degree is small. The main result of the section is Theorem~\ref{th:cl(g)-bound}, in which we put the previous facts together and obtain the estimate.

\subsection{The map and the large fibres}\label{se:de-cl(g)-prel}

We begin by producing the map $f$.

\begin{proposition}\label{pr:clg}
Let $G<\mathrm{GL}_{N}$ be an untwisted classical group of rank $r$ defined over a field $K=\mathbb{F}_{q}$ with $\mathrm{char}(\mathbb{F}_{q})>N$. Let $V=\Cl(g)$ for some $g\in G$ regular semisimple. Then, for $\ell=\frac{\dim(G)}{r}$, the map $f:V^{\ell}\rightarrow G^{\ell-1}$ defined by
\begin{equation}\label{eq:f}
f(v_{1},v_{2},\ldots,v_{\ell})=(v_{1}^{-1}v_{2},v_{2}^{-1}v_{3},\ldots,v_{\ell-1}^{-1}v_{\ell})
\end{equation}
is dominant.
\end{proposition}

\begin{proof}
By~\cite[\S I.8]{Mumford}, every fibre $V^{\ell}\cap f^{-1}(x)$ has dimension at least as large as $\dim(V^{\ell})-\dim(\overline{f(V^{\ell})})$, and strictly larger fibres form a proper subvariety of $V^{\ell}$. Therefore, since $\dim(V^{\ell})=\ell(\dim(G)-\dim(C(g)))=(\ell-1)\dim(G) = \dim(G^{\ell-1})$
 and $G^{\ell-1}$ is irreducible, $f$ is not dominant if and only if every fibre is of dimension $\geq 1$.

Fix an element $\vec{x}=(x_{1},\ldots,x_{\ell})\in V^{\ell}$, and write $\mathfrak{g}$ for the Lie algebra of $G$. For each $i$, define the maps
\begin{align*}
\varphi_{i} & :G\rightarrow G, & \varphi_{i}(y) & =y^{-1}x_{i}yx_{i}^{-1}, \\
\phi_{i} & :\mathfrak{g}\rightarrow\mathfrak{g}, & \phi_{i}(y) & =-y+\mathrm{Ad}_{x_{i}}(y).
\end{align*}
Then
\begin{equation*}
T_{x_{i}}V=T_{x_{i}}\Cl(x_{i})\simeq T_{e}\Cl(x_{i})x_{i}^{-1}=T_{e}(\varphi_{i}(G))=\phi_{i}(\mathfrak{g}).
\end{equation*}
If the fibre $Z\subseteq V^{\ell}$ of $f$ at $\vec{x}$ has dimension $\geq 1$, take a nonzero vector $\vec{t}=(t_{1},\ldots,t_{\ell})$ in the corresponding subspace $\mathfrak{z}\subseteq\prod_{i}\phi_{i}(\mathfrak{g})$. By the definition of $f$,
\[Df(\vec{t}) = (x_1^{-1} (t_2 - t_1) x_2, x_2^{-1} (t_3 - t_2) x_3,\dotsc),\]
and so  $t_{i}=t_{j}$ for all $i,j$. Hence, it would follow that $\bigcap_{i}\phi_{i}(\mathfrak{g})$ is nontrivial.

Write $\mathfrak{t}$ for $T_{e}C(x_{i})$. We would like to show that $\mathfrak{t}$ is orthogonal to the space
$\phi_i(\mathfrak{g})$ with respect to the Killing form $\langle \cdot,\cdot\rangle$. For $t\in \mathfrak{t}$ and $\mathrm{Ad}_{x_i}(y) - y
\in \phi_i(\mathfrak{g})$, using the fact that the Killing form is invariant under the adjoint action $\mathrm{Ad}_{x_i}$, we see that
\[\begin{aligned}\langle t, \mathrm{Ad}_{x_i}(y) - y\rangle &= \langle t, \mathrm{Ad}_{x_i}(y)\rangle - 
\langle t, y\rangle = \langle \mathrm{Ad}_{x_i^{-1}}(t), y\rangle - \langle t,y\rangle\\
&= \langle t,y\rangle - \langle t,y \rangle = 0,\end{aligned}\]
since $\mathrm{Ad}_{x_i^{-1}}(t) = t$.
Thus, the spaces $\phi_{i}(\mathfrak{g})$ and $T_{e}C(x_{i})$ are orthogonal with respect to the Killing form.
By our conditions on $\mathrm{char}(K)$, the Killing form is non-degenerate; for details see~\cite[p.~47]{Sel67}. 
Hence, if $\bigcap_{i}\phi_{i}(\mathfrak{g})$ is nontrivial, the tangent spaces $T_{e}C(x_{i})$ span a proper subspace of $\mathfrak{g}$.

By~\cite[Lemma~5.2]{Hel11} and~\cite[Prop.~4.13]{Hel11}, for a classical group over a field $K$ of characteristic
$>\max\{r,2\}$, and $\mathfrak{t}$ the Lie algebra of a maximal torus of $G$,  the spaces $\mathfrak{t},\mathrm{Ad}_{g_{1}}(\mathfrak{t}),\ldots,\mathrm{Ad}_{g_{\ell-1}}(\mathfrak{t})$ span all of $\mathfrak{g}$
for a generic tuple $(g_{1},\ldots,g_{\ell-1})\in G^{\ell-1}$. Writing $x=x_{\ell}$ and $x_i = g_i x g_i^{-1}$ for $1\leq i\leq \ell-1$, we see that
$T_{e}C(x_{i})  = g_i T_e C(x) g_i^{-1} = \mathrm{Ad}_{g_i}(\mathfrak{t})$, where 
$\mathfrak{t} = T_e C(x)$.
We conclude that the fibre of $f$ at $\vec{x}$ is $0$-dimensional for $g_1,\ldots g_{\ell-1}$ generic.
Therefore, $f$ is dominant.
\end{proof}

For any map $g:X\rightarrow Y$, it is possible to define a derivative map $Dg|_{x=x_{0}}$ at some point $x_{0}\in X$ between the tangent spaces $TX|_{x=x_{0}},TY|_{y=g(x_{0})}$. Call $Dg|_{x=x_{0}}$ {\em non-singular} if it is non-singular as a linear transformation, i.e.\ if the only $v$ such that $Dg|_{x=x_{0}}(v)=0$ is the zero vector. We can use the singularity of the derivative of our map $f$ to produce the exceptional variety $E$.

\begin{lemma}\label{le:de}
Let $G,V,f,r,\ell,N$ be as in Proposition~\ref{pr:clg}. For any point $\vec{v}\in V^{\ell}$, denote by $X_{\vec{v}}$ the connected component of the fibre $f^{-1}(f(\vec{v}))$ in which $\vec{v}$ is contained, and let
\begin{equation*}
E'=\{\vec{v}\in V^{\ell}:\dim(X_{\vec{v}})>0\}.
\end{equation*}
Then there exists a variety of the form $E=\{\vec{x}\in\mathbb{A}^{N^{2}\ell}:F(\vec{x})=0\}$ such that $E\cap V^{\ell}\subsetneq V^{\ell}$, $E\supseteq E'$, and $\deg(E)\leq\deg(F)\leq 2(\ell-1)\dim(G)\leq 4r^{2}(2r+1)$.
\end{lemma}

\begin{proof}
Since by Proposition~\ref{pr:clg} we have $\dim(V^{\ell})=\dim(\overline{f(V^{\ell})})=\dim(G^{\ell-1})$, for generic $\vec{v}\in V^{\ell}$ the derivative map $Df|_{\vec{x}=\vec{v}}$ is non-singular, whereas it is singular for any $\vec{v}\in E'$. It is then enough to define $E$ so that it contains all points with singular derivative.

Fix a point $\vec{g}_{0}=(g_{0,1},\ldots,g_{0,\ell})\in V^{\ell}$ where the derivative is non-singular. Every element of $V$ is conjugate to all the $g_{0,i}$, so for any $\vec{v}\in V^{\ell}$ we can write $\vec{v}=(h_{i}g_{0,i}h_{i}^{-1})_{i=1}^{\ell-1}$ for some $h_{i}\in G$. Let $h_{i}=1+\Delta_{i}$, where $\Delta_{i}\in\mathrm{Mat}_{N}$: then $h_{i}^{-1}=1-\Delta_{i}+O(\Delta_{i}^{2})$, and we can rewrite
\begin{equation*}
f(v_{1},\ldots,v_{\ell})=f(g_{0,1},\ldots,g_{0,\ell})+([\Delta_{i},g_{0,i}^{-1}]g_{0,i+1}+g_{0,i}^{-1}[\Delta_{i+1},g_{0,i+1}])_{i=1}^{\ell-1}+\ldots,
\end{equation*}
where $[x,y]=xy-yx$, and where the final ellipsis hides terms of higher order in the $\Delta_{i}$. Since the derivative is non-singular at $\vec{g}_{0}$, for each $i$ there is a set of $\dim(V)=\left(1-\frac{1}{\ell}\right)\dim(G)$ values of $\Delta_{i}$ (say $w_{i,1},\ldots,w_{i,\dim(V)}$) with the following property. Consider the elements
\begin{equation*}
\vec{w}_{i,j}(\vec{g})=(0,\ldots,0,g_{i-1}^{-1}[w_{i,j},g_{i}],[w_{i,j},g_{i}^{-1}]g_{i+1},0,\ldots,0)
\end{equation*}
as vectors in $\mathbb{A}^{(\ell-1)N^2}$ for all $1\leq i\leq\ell$ and all $1\leq j\leq\dim(V)$, where in the notation above the two nonzero coordinates are the $(i-1)$-th and the $i$-th (for $i=1,\ell$ there is only one nonzero coordinate); then, the set of vectors $\{\vec{w}_{i,j}(\vec{g}_{0})\}_{i=1,j=1}^{\ell,\dim(V)}$ is linearly independent.

Now let $M$ be the $(\ell-1)\dim(G)\times(\ell-1)N^2$ matrix whose rows are the $\vec{w}_{i,j}(\vec{g})$, where the entries of the $w_{i,j}$ are considered constants. Since for one choice of $\vec{g}$ the rows are linearly independent, there is a $(\ell-1)\dim(G)\times(\ell-1)\dim(G)$ minor $F(\vec{g})$ of $M$ that is not identically zero for $\vec{g}$ varying in $V^{\ell}$; on the other hand, a singular derivative in some point implies in particular that all the minors of $M$ are zero there.
Therefore, we can define $E$ by the equation $F(\vec{g})=0$: $E\cap V^{\ell}$ is proper in $V^{\ell}$ because $F(\vec{g}_{0})\neq 0$, and it contains $E'$.

Finally, every entry of $M$ is of degree $2$ in the entries of the $g_{i}$, because $g_{i}^{-1}$ is of degree $1$; hence,
\begin{equation*}
\deg(E)\leq\deg(F)\leq 2(\ell-1)\dim(G)\leq 4r^{2}(2r+1),
\end{equation*}
and we are done.
\end{proof}

\subsection{Dimensional estimate}

Now we put together Sections~\ref{se:escape} and~\ref{se:de-cl(g)-prel} to obtain the estimate we want.

We start by proving that escaping will indeed be possible during the proof of the main result. Recall that we write $S^{\times k}$ for the set $S\times\ldots\times S$ ($k$ times), so as to distinguish it from $S^{k}$. Also, we have $\langle A^{\times k}\rangle=\langle A\rangle^{\times k}=G(K)^{\times k}=G^{k}(K)$.

\begin{lemma}\label{le:escwecan}
Let $G<\mathrm{GL}_{N}$ be an untwisted classical group of rank $r$ defined over $K=\mathbb{F}_{q}$.  Assume that $q\geq e^{8r\log(2r)}$ and $\mathrm{char}(\mathbb{F}_{q})>2$. Let $g\in G(K)$ be regular semisimple, and let $V=\Cl(g)$.

Then, for any variety $E=\{\vec{x}\in\mathbb{A}^{N^{2}\ell}:F(\vec{x})=0\}$ defined over $\overline{K}$ with $E\cap V^{\ell}\subsetneq V^{\ell}$ and $\deg(F)\leq 2(\ell-1)\dim(G)$, we have $E(\overline{K})\cap V^{\ell}(K)\subsetneq V^{\ell}(K)$.
\end{lemma}

\begin{proof}
By Proposition~\ref{pr:clvar}, $V$ is a variety defined over $K$, and by definition the set $V(K)$ is closed under the action of conjugation by $G(K)$. Divide $V(K)$ into the orbits of this action, and take any set $S$ of elements of $G(\overline{K})$ such that for each orbit $O$ there is exactly one $s\in S$ with $sgs^{-1}\in O$.

For each $\vec{s}=(s_{1},\ldots,s_{\ell})\in S^{\times\ell}$, let $\varphi_{\vec{s}}:G^{\ell}\rightarrow V^{\ell}$ be the map
\begin{equation*}
\varphi_{\vec{s}}(h_{1},\ldots,h_{\ell})=(h_{1}s_{1}gs_{1}^{-1}h_{1}^{-1},\ldots,h_{\ell}s_{\ell}gs_{\ell}^{-1}h_{\ell}^{-1}).
\end{equation*}
By construction, $V^{\ell}(K)$ is the disjoint union of the $\varphi_{\vec{s}}(G(K))$ for all $\vec{s}\in S$. The variety $\tilde{E}_{\vec{s}}=\varphi_{\vec{s}}^{-1}(E)$ is defined by the equation $F(\varphi_{\vec{s}}(\vec{x}))=0$: observe that $F\circ\varphi_{\vec{s}}$ is a polynomial of degree $\leq 4(\ell-1)\dim(G)$ because the inverse map has degree $1$. Since $E\cap V^{\ell}\subsetneq V^{\ell}$, by definition we also have $\tilde{E}_{\vec{s}}\cap G^{\ell}\subsetneq G^{\ell}$: therefore, we can apply Corollary~\ref{co:langweil}\eqref{co:langweilariadne} to $\tilde{E}_{\vec{s}}$ and obtain that $\tilde{E}_{\vec{s}}(\overline{K})\cap G^{\ell}(K)\subsetneq G^{\ell}(K)$. In particular, by the definition of $\tilde{E}_{\vec{s}}$ we obtain as well $E(\overline{K})\cap\varphi_{\vec{s}}(G(K))\subsetneq\varphi_{\vec{s}}(G(K))$. Since the union is disjoint, this implies that $E(\overline{K})\cap V^{\ell}(K)\subsetneq V^{\ell}(K)$.
\end{proof}

\begin{theorem}\label{th:cl(g)-bound}
Let $G<\mathrm{GL}_{N}$ be an untwisted classical group of rank $r$ defined over $K=\mathbb{F}_{q}$.  Assume that $q\geq e^{8r\log(2r)}$ and $\mathrm{char}(\mathbb{F}_{q})>N$. Let $A$ be a symmetric set of generators of $G(K)$. Let $g\in G(K)$ be regular semisimple, and let $V=\Cl(g)$.

Then, for any $t\geq 1$, we have $|A^{t}\cap V(K)|\leq C_{1}|A^{C_{2}}|^{\frac{\dim(V)}{\dim(G)}}$ with $C_{1}\leq(2r)^{18r^{2}}$, and $C_{2}\leq(2r)^{17r^{2}}+2t$.
\end{theorem}

\begin{proof}
Consider the map $f$ defined in \eqref{eq:f}, for which $\mdeg(f)\leq 2$, and let $E$ be the variety defined in Lemma~\ref{le:de} by a polynomial $F$. We know that $E\cap V^{\ell}$ is a proper subvariety of $V^{\ell}$ containing all the connected components of dimension $\geq 1$ of the fibres of $V^{\ell}$ through $f$, and we have also $\deg(E)\leq\deg(F)\leq 2(\ell-1)\dim(G)$.

Following the outline at the beginning of the section, our proof is articulated into two steps. First, we bound the number of points in the $0$-dimensional part of a fibre: see \eqref{eq:fibresize}. Second, we use the escape procedure as given in Corollary~\ref{co:ariadne} to end up outside $E$: this results in the construction of an escape map $w_{1}$ as in \eqref{eq:mapw2}. Combining the two steps, we find how many points of $(A^{t}\cap V(K))^{\times\ell}$ at most are sent to a unique point of some $(A^{C_{2}})^{\times(\ell-1)}$ via $f$, or rather via the composition $w_{2}=f\circ w_{1}$ defined in \eqref{eq:mapw2}. This procedure leads to \eqref{eq:combinedbound}, and then to the final estimate.

We start with the fibre bound. Let $\vec{x}\notin E$: then $\vec{x}$ sits in the union $S_{\vec{x}}$ of all $0$-dimensional components of the fibre $f^{-1}(f(\vec{x}))$. By definition, $S_{\vec{x}}$ is made of finitely many points: we want to estimate $|S_{\vec{x}}|$ independently from $\vec{x}$; we may do so by bounding the degree of the fibre directly, but there is a more convenient route. Given $\vec{y}=(y_{1},\ldots,y_{\ell-1})\in G^{\ell-1}$, define the variety
\begin{equation*}
W(\vec{y})=\{v_{1}\in G:p_{v_{1}}=p_{v_{1}y_{1}}=p_{v_{1}y_{1}y_{2}}=\ldots=p_{v_{1}y_{1}\ldots y_{\ell-1}}=p_{g}\},
\end{equation*}
where $p_{h}$ denotes the characteristic polynomial of $h$, and define also the map
\begin{align*}
F_{\vec{y}} & :G\rightarrow G^{\ell}, & F_{\vec{y}}(v_{1}) & =(v_{1},v_{1}y_{1},v_{1}y_{1}y_{2},\ldots,v_{1}y_{1}\ldots y_{\ell-1}).
\end{align*}
The map $F_{\vec{y}}$ is a morphism, and its inverse is just the projection to the first component, thus $F_{\vec{y}}(W(\vec{y}))$ is isomorphic to $W(\vec{y})$. On one hand, $F_{\vec{y}}(W(\vec{y}))$ is equal to $V^{\ell}\cap f^{-1}(\vec{y})$ by construction. Isomorphisms preserve the number of isolated points, i.e.\ of $0$-dimensional components, even when they do not preserve overall degrees: therefore
\begin{align*}
|S_{\vec{x}}| & =|\{\text{isolated points of $f^{-1}(f(\vec{x}))$}\}| \\
 & =|\{\text{isolated points of $W(f(\vec{x}))$}\}|\leq\deg(W(f(\vec{x}))).
\end{align*}
Then, arguing as in the proof of Proposition~\ref{pr:clvar}, we obtain that $\deg(W(f(\vec{x})))\leq(r!)^{\ell}\deg(G)$, which gives
\begin{equation}\label{eq:fibresize}
|S_{\vec{x}}|\leq(r!)^{\ell}\deg(G)\leq r^{4r^{2}}
\end{equation}
for $r>1$, and $|S_{\vec{x}}|\leq 2$ for $r=1$.

Now apply Corollary~\ref{co:ariadne} for $(G,K,A,V,s)=(G(K),K,A,E,\ell)$ and the action of $G(K)$ by conjugation. Thus, we know that there is a set of $(m+1)^{\ell}$ invertible functions, say $w_{\theta}$ for an index $1\leq\theta\leq(m+1)^{\ell}$, such that $w_{\theta}((A^{t}\cap V(K))^{\times\ell})\subseteq(A^{2k+t}\cap V(K))^{\times\ell}$ for all $\theta$, and also such that for every $\vec{x}\in V^{\ell}(K)$ there is some $\theta_{\vec{x}}$ with $w_{\theta_{\vec{x}}}(\vec{x})\notin E(K)$ (because $E\cap V^{\ell}\subsetneq V^{\ell}$ implies $V^{\ell}(K)\not\subseteq E(K)$ by Lemma~\ref{le:escwecan}). Construct the functions $w_{1}:V^{\ell}(K)\rightarrow V^{\ell}(K)$ and $w_{2}:V^{\ell}(K)\rightarrow G^{\ell-1}(K)$ given by
\begin{align}\label{eq:mapw2}
w_{1}(\vec{x}) & =w_{\theta_{\vec{x}}}(\vec{x}), & w_{2}(\vec{x}) & =f(w_{1}(\vec{x})).
\end{align}
We have
\begin{equation}\label{eq:containments}
w_{2}((A^{t}\cap V(K))^{\times\ell})\subseteq f((A^{2k+t}\cap V(K))^{\times\ell}\setminus E(K))\subseteq(A^{4k+2t})^{\times(\ell-1)},
\end{equation}
so let us examine preimages via $w_{2}$. On one hand, for any fixed $\vec{z}\in(A^{4k+2t})^{\times(\ell-1)}$ there can only be at most $\max\{2,r^{4r^{2}}\}$ elements $\vec{y}\in(A^{2k+t}\cap V(K))^{\times\ell}\setminus E(K)$ whose image via $f$ is $\vec{z}$: this is because each $\vec{y}$ must sit inside $S_{\vec{x}}$ (by definition of $E$), and the bound follows from \eqref{eq:fibresize}. On the other hand, for any fixed $\vec{y}$ there can only be at most $(m+1)^{\ell}$ elements $\vec{x}\in(A^{t}\cap V(K))^{\times\ell})$ whose preimage via $w_{1}$ is $\vec{y}$, because each $w_{\theta}$ is invertible. Combining these facts with \eqref{eq:containments}, we conclude that
\begin{equation*}
\left|(A^{t}\cap V(K))^{\times\ell}\right|\leq(m+1)^{\ell}\max\{2,r^{4r^{2}}\}\left|(A^{4k+2t})^{\times(\ell-1)}\right|,
\end{equation*}
which implies
\begin{equation}\label{eq:combinedbound}
|A^{t}\cap V(K)|\leq(m+1)\max\{2,r^{4r^{2}}\}^{1/\ell}|A^{4k+2t}|^{1-\frac{1}{\ell}}.
\end{equation}

It remains only to estimate the values of $m$ and $k$, which according to Corollary~\ref{co:ariadne} are given by Proposition~\ref{pr:strong-escape}. For $G\neq\mathrm{SL}_{n}$, by Proposition~\ref{pr:strong-escape}, Table~\ref{ta:basicg}, and the bound $\deg(E)\leq\deg(F)\leq 2(\ell-1)\dim(G)$, we obtain
\begin{equation}\label{eq:clmk}
m,k\leq\left(1+\frac{1}{\deg(F)-1}\right)\deg(F)^{N^{2}}<\frac{1}{4}(2r)^{17r^{2}},
\end{equation}
which gives us the desired $C_{2}$. For $G=\mathrm{SL}_{n}$ we may use the alternative bounds in Corollary~\ref{co:ariadne}: since $E$ is given by one equation $F(\vec{x})=0$, restricting to $G$ replaces the entries of $\vec{x}$ with polynomials of degree $n-1$ in the entries of the upper left corners of $\vec{x}$, thus yielding $\deg(E\cap G)\leq(n-1)\deg(F)$. The bound then becomes
\begin{equation}\label{eq:clmk2}
m,k\leq\left(1+\frac{1}{(n-1)\deg(F)-1}\right)((n-1)\deg(F))^{\frac{1}{4}N^{2}}<\frac{1}{4}(2r)^{17r^{2}},
\end{equation}
again giving the desired $C_{2}$. Then, inside~\eqref{eq:combinedbound} we can bound $(2r)^{17r^{2}}\max\{2,r^{4r^{2}}\}^{1/\ell}\leq(2r)^{18r^{2}}$ case by case, which gives us the desired $C_{1}$.
\end{proof}


\section{Estimates for non-maximal tori}\label{se:de-torus}

Our aim in this section is to prove an upper bound for $|A^{t}\cap T(K)|$, where $T$ is a torus of $G$ whose elements are not regular semisimple. This will help us conclude that not only can we find when necessary some $g\in A^{t}\cap T_{\max}(K)$, for $T_{\max}$ a maximal torus and for $t$ small, but that we can also assume $g$ to be regular semisimple. For the rest of the section we suppose that $r\geq 2$, since all tori are maximal by definition when the rank is $1$, and append the case $r=1$ only at the very end (Corollary~\ref{co:torusnonrs}).

The exponent in the dimensional estimate we obtain here is not as tight as in the case of $\Cl(g)$: instead of aiming for $\frac{\dim(T)}{\dim(G)}$, we content ourselves with $\frac{1}{\ell+1}$. As long as it is bounded away from $\frac{1}{\ell}=\frac{\dim(T_{\max})}{\dim(G)}$, we are fine.

Our process follows in many parts~\cite[\S\S 4-5]{Hel11}. The key ideas of the proof are essentially as in Section~\ref{se:de-cl(g)}. We find a suitable map from $T^{\ell+1}$ to $G$, given in \eqref{eq:suitmaptorus}, proving first in Section~\ref{se:de-torus-linind} that its generic fibre will be $0$-dimensional. Then we find a variety $E$ containing all large fibre components, and we bound its degree: this is the role played by Proposition~\ref{pr:ybound}, where in the end $E=\tilde{Y}|_{\vec{x}=\vec{g}}$ for some suitable $\vec{g}\notin\tilde{X}$. Finally, from this construction we reach our estimate in Theorem~\ref{th:torusbound}: the result already appears for $G=\mathrm{SL}_{n}$ as~\cite[Cor.~5.14]{Hel11}, though without explicit constants.

\subsection{Linear independence}\label{se:de-torus-linind}

Fix any non-maximal torus $T$, and recall that $\ell=\frac{\dim(G)}{r}$. We want to find a ``good'' $\vec{g}=(g_{1},\ldots,g_{\ell})\in G^{\ell}$, in the sense that the Lie algebra $\mathfrak{t}$ of $T$ and its adjoint spaces $\mathrm{Ad}_{g_{i}}(\mathfrak{t})$ are all \textit{linearly independent}; in other words, the only linear combination $v+v_{1}+\ldots+v_{\ell}$ equal to $0$ for $v\in\mathfrak{t}$ and $v_{i}\in\mathrm{Ad}_{g_{i}}(\mathfrak{t})$ is the trivial one. Here we establish that there is at least one such $\vec{g}$.

Recall the definition of canonical torus from Definition~\ref{de:canontorus}.
For canonical non-maximal tori, we can explicitly construct a linearly independent combination of spaces.

\begin{proposition}\label{pr:torusoneg}
Let $G=G_{n}<\mathrm{GL}_{N}$ be an untwisted classical group of rank $r\geq 2$ over $\mathbb{F}_{q}$, with $\ell=\frac{\dim(G)}{r}$, and let $T$ be a canonical non-maximal torus of $G$. Assume that $\mathrm{char}(\mathbb{F}_{q})\nmid 2N$.

Then, there is a tuple $\vec{g}=(g_{1},\ldots,g_{\ell})\in\mathfrak{g}^{\ell}(\mathbb{F}_{q})$ such that the spaces $\mathfrak{t},[g_{1},\mathfrak{t}],\ldots,$ $[g_{\ell},\mathfrak{t}]$ are linearly independent and of dimension $\dim(\mathfrak{t})$.
\end{proposition}

\begin{proof}
We prove the statement case by case. For $G_{n}=\mathrm{SL}_{n}$ (and $\mathrm{char}(\mathbb{F}_{q})\nmid N$), this is \cite[Lemma~5.12]{Hel11}: the result is proved for the more common representation of $\mathrm{SL}_{n}$, but it extends naturally to $\mathrm{SL}_{n}$ as in Definition~\ref{de:ambient} by constructing the spaces in the upper left corner.

Let $G_{n}=\mathrm{SO}_{2n}^{+}$, for which $\ell=2n-1$. Its Lie algebra $\mathfrak{so}_{2n}^{+}$ is given in \S\ref{se:chev}. We may embed a copy of $\mathfrak{so}_{2n-2}^{+}$ inside $\mathfrak{so}_{2n}^{+}$ as
\begin{align*}
\iota & :\mathfrak{so}_{2n-2}^{+}\rightarrow\mathfrak{so}_{2n}^{+}, & \iota\begin{pmatrix} x_{11} & x_{12} \\ x_{21} & x_{22} \end{pmatrix} & =\left(
\begin{array}{cc|cc}
x_{11} & \vec{0} & x_{12} & \vec{0} \\
\vec{0} & 0 & \vec{0} & 0 \\
\hline
x_{21} & \vec{0} & x_{22} & \vec{0} \\
\vec{0} & 0 & \vec{0} & 0 \\
\end{array}
\right).
\end{align*}
Let $\mathfrak{t}'$ be the Lie algebra of the canonical maximal torus of $\mathfrak{so}_{2n-2}^{+}$, and recall that it is the same as the diagonal Cartan subalgebra of $\mathfrak{so}_{2n-2}^{+}$. We apply \cite[Lemma~5.2]{Hel11} to $\mathfrak{so}_{2n-2}^{+}$; the result holds since $\mathrm{char}(\mathbb{F}_{q})\neq 2$, and it can be used whenever $2n-2\geq 2$ (by our definition $n\geq 4$, so we are fine). Therefore, there are $2n-4=\ell-3$ elements $g_{1},\ldots,g_{\ell-3}\in\mathfrak{so}_{2n-2}^{+}$ such that $\mathfrak{t}',[g_{1},\mathfrak{t}'],\ldots,[g_{\ell-3},\mathfrak{t}']$ are linearly independent inside $\mathfrak{so}_{2n-2}^{+}$ and of dimension $\dim(\mathfrak{t}')$. Since $\sum_{i}\eta_{i}a_{i}=0$ with $\eta_{n}\neq 0$, $a_{n}$ is linearly dependent on $a_{1},\ldots,a_{n-1}$ inside $\mathfrak{t}$: this means that the $a_{1},\ldots,a_{n-1}$ are linearly independent in $\mathfrak{t}$, and that $\iota([g_{i},\mathfrak{t}'])=[\iota(g_{i}),\mathfrak{t}]$. This implies that the spaces $\mathfrak{t},[\iota(g_{1}),\mathfrak{t}],\ldots,[\iota(g_{\ell-3}),\mathfrak{t}]$ are also linearly independent and of dimension $\dim(\mathfrak{t}')=\dim(\mathfrak{t})$.

We need $3$ more elements. Call $v$ the $(n-1)\times 1$ matrix having all entries $1$, and call $e_{i}$ the $(n-1)\times 1$ matrix having $1$ at the $i$-th entry and $0$ elsewhere. Then, since $n\geq 4$ we can take
\begin{align*}
h_{1} & =\left(
\begin{array}{cc|cc}
\vec{0} & v & \vec{0} & \vec{0} \\
\vec{0} & 0 & \vec{0} & 0 \\
\hline
\vec{0} & e_{2} & \vec{0} & \vec{0} \\
-e_{2}^{\top} & 0 & -v^{\top} & 0 \\
\end{array}
\right), &
h_{2} & =\left(
\begin{array}{cc|cc}
\vec{0} & \vec{0} & \vec{0} & \vec{0} \\
-v^{\top} & 0 & \vec{0} & 0 \\
\hline
\vec{0} & e_{3} & \vec{0} & v \\
-e_{3}^{\top} & 0 & \vec{0} & 0 \\
\end{array}
\right), \\
h_{3} & =\left(
\begin{array}{cc|cc}
\vec{0} & e_{1} & \vec{0} & v \\
-e_{1}^{\top} & 0 & -v^{\top} & 0 \\
\hline
\vec{0} & \vec{0} & \vec{0} & \vec{0} \\
\vec{0} & 0 & \vec{0} & 0 \\
\end{array}
\right). & &
\end{align*}
The spaces $[h_{j},\mathfrak{t}]$ have matrices whose nonzero entries are concentrated in the $n$-th and $2n$-th rows and columns, so their span is orthogonal to $\iota(\mathfrak{so}_{2n-2}^{+})$. We only need to verify that they are linearly independent among each other and of dimension $\dim(\mathfrak{t})$. Given $t_{j}\in\mathfrak{t}$ defined by entries $a_{j,i}$, the sum $x=[h_{1},t_{1}]+[h_{2},t_{2}]+[h_{3},t_{3}]$ has entries
\begin{align*}
x_{2n,n+i} & =a_{1,i}-a_{1,n} \ \ \ (1\leq i<n), & x_{n+2,n} & =a_{1,2}+a_{1,n}, \\
x_{n+i,2n} & =a_{2,i}-a_{2,n} \ \ \ (1\leq i<n), & x_{n+3,n} & =a_{2,3}+a_{2,n}, \\
x_{n,n+i} & =a_{3,i}+a_{3,n} \ \ \ (1\leq i<n), & x_{1,n} & =-a_{1,1}-a_{3,1}+a_{1,n}+a_{3,n}.
\end{align*}
Since $\mathrm{char}(\mathbb{F}_{q})\neq 2$ we can reconstruct all the entries of the $t_{j}$ from $x$, and we are done. This completes the case $G_{n}=\mathrm{SO}_{2n}^{+}$.

Let $G_{n}=\mathrm{SO}_{2n+1}$, for which $\ell=2n+1$. Its Lie algebra $\mathfrak{so}_{2n+1}$ is again given in \S\ref{se:chev}. We may embed a copy of $\mathfrak{so}_{2n}^{+}$ inside $\mathfrak{so}_{2n+1}$ as
\begin{align*}
\iota & :\mathfrak{so}_{2n}^{+}\rightarrow\mathfrak{so}_{2n+1}, & \iota\begin{pmatrix} x_{11} & x_{12} \\ x_{21} & x_{22} \end{pmatrix} & =\left(
\begin{array}{cc|c}
x_{11} & x_{12} & \vec{0} \\
x_{21} & x_{22} & \vec{0} \\
\hline
\vec{0} & \vec{0} & 0
\end{array}
\right).
\end{align*}
The images via $\iota$ of the $\ell-2$ elements of $\mathfrak{so}_{2n}^{+}$ we found before can be used for $\mathfrak{so}_{2n+1}$ too, so it is sufficient to find $2$ more elements. Call $v$ the $n\times 1$ matrix having all entries $1$. Then we can take
\begin{align*}
h_{1} & =\left(
\begin{array}{cc|c}
\vec{0} & \vec{0} & v \\
\vec{0} & \vec{0} & \vec{0} \\
\hline
\vec{0} & -v^{\top} & 0
\end{array}
\right), &
h_{2} & =\left(
\begin{array}{cc|c}
\vec{0} & \vec{0} & \vec{0} \\
\vec{0} & \vec{0} & v \\
\hline
-v^{\top} & \vec{0} & 0
\end{array}
\right).
\end{align*}
The matrices $[h_{1},t],[h_{2},t]$ have nonzero entries only in the last row and column. Moreover, $a_{j}$ appears as $x_{2n+1,n+j}$ in $[h_{1},t]$, and as $x_{n+j,2n+1}$ in $[h_{2},t]$. This completes the case $G_{n}=\mathrm{SO}_{2n+1}$.

Finally, let $G_{n}=\mathrm{Sp}_{2n}$, for which $\ell=2n+1$. Its Lie algebra $\mathfrak{sp}_{2n}$ is the set of matrices $\begin{pmatrix} A & B \\ C & -A^{\top} \end{pmatrix}$ with $B^{\top}=B$ and $C^{\top}=C$. We may embed a copy of $\mathfrak{sp}_{2n-2}$ inside $\mathfrak{sp}_{2n}$ as
\begin{align*}
\iota & :\mathfrak{sp}_{2n-2}\rightarrow\mathfrak{sp}_{2n}, & \iota\begin{pmatrix} A & B \\ C & -A^{\top} \end{pmatrix} & =\left(
\begin{array}{cc|cc}
A & \vec{0} & B & \vec{0} \\
\vec{0} & 0 & \vec{0} & 0 \\
\hline
C & \vec{0} & -A^{\top} & \vec{0} \\
\vec{0} & 0 & \vec{0} & 0
\end{array}
\right).
\end{align*}
As in the case of $\mathrm{SO}_{2n}^{+}$, by \cite[Lemma~5.2]{Hel11} we can take $\ell-3$ elements from $\mathfrak{sp}_{2n-2}$ whose images via $\iota$ are suitable elements for our $\mathfrak{t}$. We need $3$ more.

Since $\eta_{n}\neq 0$, we must have either $\eta_{n}+\sum_{i\neq n}\eta_{i}\neq 0$ or $\eta_{n}-\sum_{i\neq n}\eta_{i}\neq 0$. Assume first the former inequality. Then, for $i=1,2,3$ we define $h_{i}=\begin{pmatrix} A_{i} & B_{i} \\ 0 & -A_{i}^{\top} \end{pmatrix}$ with
\begin{align*}
A_{1}=A_{2}^{\top} & =\begin{pmatrix} \vec{0} & \vec{1} \\ \vec{0} & 0 \end{pmatrix}, & B_{3} & =\begin{pmatrix} \vec{0} & \vec{1} \\ \vec{1} & 1 \end{pmatrix}, & A_{3}=B_{1}=B_{2} & =0.
\end{align*}
The set of coordinates with nonzero entries of each $[h_{i},t]$ is disjoint from the others and from those of the already considered elements, so we only need to prove that $\dim([h_{i},\mathfrak{t}])=\dim(\mathfrak{t})$. For $[h_{3},t]$, we have $x_{j,2n}=-a_{j}-a_{n}$ for $j\geq n$: thus, since $\mathrm{char}(\mathbb{F}_{q})\neq 2$ we get $a_{n}$ from $x_{n,2n}$ and then all other $a_{j}$. For $[h_{1},t]$, we have $x_{j,n}=a_{n}-a_{j}$ for $j<n$, and analogous entries are found in $[h_{2},t]$. Define the map $\varphi_{+}:\mathbb{A}^{n}\rightarrow\mathbb{A}^{n}$ as $\varphi_{+}(\vec{y})=(y_{n}-y_{1},\ldots,y_{n}-y_{n-1})$, and note that $\dim(\varphi_{+}(\mathbb{A}^{n}))=\dim(\mathbb{A}^{n-1})=n-1=\dim(\mathfrak{t})$: as long as $\varphi_{+}(V)=\mathbb{A}^{n-1}$ for $V=\{\vec{y}\in\mathbb{A}^{n}:\sum_{i}\eta_{i}y_{i}=0\}$, we are settled. For $\vec{s}\in\mathbb{A}^{n-1}$, an element of the fibre $\varphi_{+}^{-1}(\vec{s})\cap V$ has to have $y_{j}=y_{n}+s_{j}$ for all $j<n$, and it must satisfy $\sum_{i}\eta_{i}y_{i}=\eta_{n}y_{n}+\sum_{i\neq n}\eta_{i}(y_{n}+s_{i})=0$ too; but there is always a $y_{n}$ satisfying such an equation, since $\eta_{n}+\sum_{i\neq n}\eta_{i}\neq 0$ by hypothesis, so $\varphi_{+}(V)=\mathbb{A}^{n-1}$ and we are done.

Assume now instead that $\eta_{n}+\sum_{i\neq n}\eta_{i}=0$, which implies $\eta_{n}-\sum_{i\neq n}\eta_{i}\neq 0$. In this case we define $h_{i}=\begin{pmatrix} A_{i} & B_{i} \\ C_{i} & -A_{i}^{\top} \end{pmatrix}$ with
\begin{align*}
A_{1}=A_{2}^{\top} & =\begin{pmatrix} \vec{0} & \vec{1} \\ \vec{0} & 0 \end{pmatrix}, & B_{3} & =\begin{pmatrix} \vec{0} & \vec{1} \\ \vec{1} & 1 \end{pmatrix}, & B_{1}=C_{2} & =\begin{pmatrix} \vec{0} & \vec{0} \\ \vec{0} & 1 \end{pmatrix}, & A_{3}=B_{2}=C_{1}=C_{3} & =0.
\end{align*}
For $[h_{1},t]$, we have $x_{j,n}=a_{n}-a_{j}$ for $j<n$ and $x_{n,2n}=-2a_{n}$. For $[h_{2},t]$, we have $x_{n,j}=a_{j}-a_{n}$ for $j<n$ and $x_{2n,n}=2a_{n}$. For $[h_{3},t]$, we have $x_{j,2n}=-a_{n}-a_{j}$ for $j<n$. Defining $\varphi_{-}:\mathbb{A}^{n}\rightarrow\mathbb{A}^{n}$ as $\varphi_{-}(\vec{y})=(y_{n}+y_{1},\ldots,y_{n}+y_{n-1})$, we repeat the passages as in the previous case and find that $\varphi_{-}(V)=\mathbb{A}^{n-1}$ using $\eta_{n}-\sum_{i\neq n}\eta_{i}\neq 0$. This completes the case $G_{n}=\mathrm{Sp}_{2n}$.
\end{proof}

From the result above, we get an analogous statement for arbitrary non-maximal tori.

\begin{corollary}\label{co:torusoneg}
Let $G=G_{n}<\mathrm{GL}_{N}$ be an untwisted classical group of rank $r\geq 2$ over $\mathbb{F}_{q}$, with $\ell=\frac{\dim(G)}{r}$, and let $T$ be a non-maximal torus of $G$, with Lie algebra $\mathfrak{t}$. Assume that $q\geq e^{3r\log(2r)}$ and $\mathrm{char}(\mathbb{F}_{q})>N$.

Then there is a tuple $\vec{g}=(g_{1},\ldots,g_{\ell})\in G^{\ell}(\mathbb{F}_{q})$ such that the spaces $\mathfrak{t},\mathrm{Ad}_{g_{1}}(\mathfrak{t}),\ldots,$ $\mathrm{Ad}_{g_{\ell}}(\mathfrak{t})$ are linearly independent and of dimension $\dim(\mathfrak{t})$.
\end{corollary}

\begin{proof}
Let $T_{\max}$ be a maximal torus of $G$ containing $T$. By \cite[Lemma~4.14]{Hel11}, there is a non-trivial character $\alpha:T_{\max}\rightarrow\mathbb{A}^{1}$ such that $T\subseteq\ker(\alpha)$: $\alpha$ is a map of the form $\alpha(x_{1},\ldots,x_{\ell})=\prod_{i}x_{i}^{\eta_{i}}$ for some $\eta_{i}\in\mathbb{Z}$ not all $0$, where the $x_{i}$ are the free parameters defining $T_{\max}$. Without loss of generality, we may assume that $T=\ker(\alpha)$, as linear independence for $\ker(\alpha)$ and its adjoints implies linear independence for their subspaces; up to conjugation, we may also assume that $T_{\max}$ is canonical as in Definition~\ref{de:canontorus}, since all maximal tori are conjugate in $G$. Then, the Lie algebra $\mathfrak{t}$ lies inside the kernel $\mathfrak{k}$ of the corresponding character $\bar{\alpha}:\mathfrak{t}_{\max}\rightarrow\mathbb{A}^{1}$ given by $\bar{\alpha}(a)=\sum_{i}\eta_{i}a_{i}$; up to conjugation again, we can reorder the $a_{i}$ so as to have $\eta_{n}\neq 0$.

We may have $\bar{\alpha}$ identically $0$ on $\mathbb{F}_{q}$, so we avoid this case by reasoning as follows. Let $p=\mathrm{char}(\mathbb{F}_{q})$, and let $p^{e}$ be the highest power of $p$ dividing all the $\eta_{i}$: then $\alpha(x)=\alpha'(x)^{p^{e}}$, where $\alpha'(x)=\prod_{i}x_{i}^{\eta'_{i}}$ is a non-trivial character with $p\nmid\mathrm{gcd}(\eta'_{1},\ldots,\eta'_{n})$. The new character $\bar{\alpha}'$ has the same kernel $\mathfrak{k}$, because $y\mapsto y^{p^{e}}$ is a power of the Frobenius automorphism of $\mathbb{F}_{q}$ and thus $\alpha(x)=1$ if and only if $\alpha'(x)=1$. Therefore we may work with $\bar{\alpha}'$ instead of $\bar{\alpha}$, and $\sum_{i}\eta'_{i}a_{i}$ is not identically $0$ anymore.

By the description above, $T$ is canonical. Since in particular $p\nmid 2N$, by Proposition~\ref{pr:torusoneg} there is some $\vec{g}'\in\mathfrak{g}^{\ell}(\mathbb{F}_{q})$ for which the spaces $\mathfrak{k},[g'_{1},\mathfrak{k}],\ldots,[g'_{\ell},\mathfrak{k}]$ are linearly independent and of dimension $\dim(\mathfrak{k})$, so the same holds for $\mathfrak{t},[g'_{1},\mathfrak{t}],\ldots,[g'_{\ell},\mathfrak{t}]$. Then, since $p>\dim(\mathfrak{t})$, the previous assertion implies that for every $\vec{g}\in G^{\ell}(\overline{\mathbb{F}_{q}})$ outside a proper subvariety $X\subsetneq G^{\ell}$ the spaces $\mathfrak{t},\mathrm{Ad}_{g_{1}}(\mathfrak{t}),\ldots,\mathrm{Ad}_{g_{\ell}}(\mathfrak{t})$ are linearly independent and of dimension $\dim(\mathfrak{t})$ by~\cite[Prop.~4.13]{Hel11}.

It remains to find some $\vec{g}$ inside $G^{\ell}(\mathbb{F}_{q})$ specifically. We have $p>r$, so we can use~\cite[(4.14)]{Hel11} to define $X$: thus, $X$ is defined by the equation $\theta\wedge\mathrm{Ad}_{g_{1}}(\theta)\wedge\ldots\wedge\mathrm{Ad}_{g_{\ell}}(\theta)=0$, where $\theta$ is the exterior product of a basis of $\mathfrak{t}$. Since $\mathrm{Ad}_{g}(x)$ has degree $2$ in the entries of $g$, and since $x\wedge y$ has degree equal to the sum of the degrees of $x,y$, $X$ is defined by finitely many equations $f_{i}=0$ with $\deg(f_{i})\leq 2\ell$; at least one $f_{i_{0}}$ gives a proper subvariety of $G^{\ell}$, or else $X$ itself would not be proper. Therefore we can actually work with $\tilde{X}=\{f_{i_{0}}=0\}$. By Corollary~\ref{co:langweil}\eqref{co:langweilstick}, there exists a $\vec{g}\in G^{\ell}(\mathbb{F}_{q})\setminus\tilde{X}(\mathbb{F}_{q})$.
\end{proof}

\subsection{Dimensional estimates}

We start with a technical lemma. Recall the definition of non-singular derivative from Section~\ref{se:de-cl(g)}.

\begin{lemma}\label{le:carvefibre}
Let $f:X\rightarrow\mathbb{A}^{m}$ be a morphism. Let $E\subseteq X$ be such that $Df|_{x=x_{0}}$ is non-singular for all $x_{0}\in X\setminus E$. Then, for any $x_{0}\in X\setminus E$, we have
\begin{align*}
\dim(\overline{f^{-1}(f(x_{0}))\setminus E}) & =0, \\
\deg(\overline{f^{-1}(f(x_{0}))\setminus E}) & \leq\deg(X)\mdeg(f)^{m}.
\end{align*}
\end{lemma}

\begin{proof}
Fix some $x_{0}\in X\setminus E$ (if $E=X$ the statement is empty), and assume that $x_{0}$ sits on an irreducible component $Z$ of $\overline{f^{-1}(f(x_{0}))\setminus E}$ of dimension at least $1$. Then, any nonzero vector $v\in TZ|_{x=x_{0}}$ must have $Df|_{x=x_{0}}(v)=0$, since it is tangent to the fibre of $f$ at $x_{0}$, which contradicts non-singularity at $x_{0}$.

It remains to bound the degree. For a $0$-dimensional object we can write
\begin{equation*}
\overline{f^{-1}(f(x_{0}))\setminus E}=f^{-1}(f(x_{0}))\setminus E=f^{-1}(f(x_{0}))\setminus(f^{-1}(f(x_{0}))\cap E),
\end{equation*}
so we have reduced ourselves to a situation in which $Z=Z_{1}\setminus Z_{2}$ with all $Z,Z_{1},Z_{2}$ being varieties and with $Z_{1}\supseteq Z_{2}$: in this case we must have $Z_{1}=Z\sqcup Z_{2}$, and then $\deg(Z)\leq\deg(Z_{1})$. Thus, the degree is bounded by
\begin{equation*}
\deg(f^{-1}(f(x_{0})))=\deg(\{x\in X:f(x)=f(x_{0})\})\leq\deg(X)\mdeg(f)^{m},
\end{equation*}
and we are done.
\end{proof}

Let $\phi:G^{\ell}\times T^{\ell+1}\rightarrow G$ be the map given by
\begin{equation*}
\phi(\vec{x},\vec{y})=\phi((x_{1},x_{2},\ldots,x_{\ell}),(y,y_{1},y_{2},\ldots,y_{\ell}))=y\cdot x_{1}y_{1}x_{1}^{-1}\cdot x_{2}y_{2}x_{2}^{-1}\cdot\ldots\cdot x_{\ell}y_{\ell}x_{\ell}^{-1}.
\end{equation*} 
For fixed $\vec{g}\in G^{\ell}$, we write $\phi_{\vec{g}}:T^{\ell+1}\rightarrow G$ for the map defined by $\phi_{\vec{g}}(\vec{y})=\phi(\vec{g},\vec{y})$. As mentioned before, for all $\vec{g}\in G^{\ell}$ and $\vec{t}\in T^{\ell+1}$ the map $\phi_{\vec{g}}$ defines a derivative
\begin{align*}
(D\phi_{\vec{g}})|_{\vec{y}=\vec{t}} & :TT^{\ell+1}|_{\vec{y}=\vec{t}}\rightarrow TG|_{\vec{z}=\phi_{\vec{g}}(\vec{t})}, & (D\phi_{\vec{g}})|_{\vec{y}=\vec{t}}(v) & =\left.\frac{\partial}{\partial\vec{y}}\phi_{\vec{g}}(\vec{y})\right|_{\vec{y}=\vec{t}}(v).
\end{align*}
Call
\begin{equation}\label{eq:ygt}
Y_{G^{\ell}\times T^{\ell+1}}=\{(\vec{g},\vec{t})\in G^{\ell}\times T^{\ell+1}:(D\phi_{\vec{g}})|_{\vec{y}=\vec{t}}\text{ is singular}\}.
\end{equation}

Let $\psi:G^{\ell}\times T^{\ell+1}\rightarrow G^{\ell}$ be the map defined by $\psi(\vec{x},\vec{y})=\vec{x}'=(x'_{1},x'_{2},\ldots,x'_{\ell})$ where $x'_{\ell}=x_{\ell}$ and $x'_{j}=x'_{j+1}y^{-1}_{j+1}x^{-1}_{j+1}x_{j}$ for $1\leq j\leq\ell-1$. Note that, if we call $\psi_{\vec{t}}:G^{\ell}\rightarrow G^{\ell}$ the map given by $\psi_{\vec{t}}(\vec{x})=\psi(\vec{x},\vec{t})$, then ${\psi_{\vec{t}}}^{-1}=\psi_{\vec{t}^{-1}}$.

Finally, define the map
\begin{align*}
\rho_{(\vec{g},\vec{t})} & :TT^{\ell+1}|_{\vec{y}=\vec{e}}\rightarrow TG|_{z=e}, & \rho_{(\vec{g},\vec{t})}(v) & =\left.\left(\frac{\partial}{\partial\vec{y}}\phi(\psi(\vec{g},\vec{t}),\vec{y})\right)\right|_{\vec{y}=\vec{e}}(v).
\end{align*}
By the definition of $\psi$, in practice we are moving $(D\phi_{\vec{g}})|_{\vec{y}=\vec{t}}$ to the tangent spaces at the identity, so that we can work with the derivative map while having domain and codomain independent from $\vec{g},\vec{t}$.

The following proposition is closely related to~\cite[Prop.~4.9]{Hel11}.

\begin{proposition}\label{pr:ybound}
Let $G<\mathrm{GL}_{N}$ be an untwisted classical group of rank $r\geq 2$ over $\mathbb{F}_{q}$, with $\ell=\frac{\dim(G)}{r}$, and let $T$ be a non-maximal torus of $G$. Assume that $q\geq e^{3r\log(2r)}$ and $\mathrm{char}(\mathbb{F}_{q})>N$.

Then, there are varieties $\tilde{X}\subseteq\mathbb{A}^{N^{2}\ell},\tilde{Y}\subseteq\mathbb{A}^{N^{2}(2\ell+1)}$ defined over $\mathbb{F}_{q}$ such that, calling $X=\tilde{X}\cap G^{\ell},Y=\tilde{Y}\cap(G^{\ell}\times T^{\ell+1})$, the following properties are satisfied:
\begin{enumerate}[(i)]
\item $X\subsetneq G^{\ell}$ and $Y\subsetneq G^{\ell}\times T^{\ell+1}$,
\item\label{pr:yboundnonsing} $Y\supseteq Y_{G^{\ell}\times T^{\ell+1}}$, where the latter is as in~\eqref{eq:ygt},
\item $Y|_{\vec{y}=\vec{t}}\subsetneq G^{\ell}$ for any $\vec{t}\in T^{\ell+1}$,
\item\label{pr:yboundprop} $(Y|_{\vec{x}=\vec{g}})(\mathbb{F}_{q})\subsetneq T^{\ell+1}(\mathbb{F}_{q})$ for any $\vec{g}\in G^{\ell}\setminus X$, and in particular $\vec{e}\not\in(Y|_{\vec{x}=\vec{g}})(\mathbb{F}_{q})$,
\item\label{pr:yboundlw} $\tilde{X}=\{x\in\mathbb{A}^{N^{2}\ell}:F(x)=0\}$ for some polynomial $F$ with $\deg(F)\leq 2(r-1)$, and
\item\label{pr:ybounddeg} $\deg(\tilde{Y}|_{\vec{x}=\vec{g}})\leq\frac{1}{2}\ell(\ell-1)(r-1)$ for any $\vec{g}\in G^{\ell}$.
\end{enumerate}
\end{proposition}

\begin{proof}
Observe that $(D\phi_{\bar{g}})|_{\vec{y}=\bar{t}}$ is non-singular if and only if $\rho_{(\bar{g},\bar{t})}$ is non-singular. We can rewrite the map $\rho_{(\bar{g},\bar{t})}$ as a linear map between Lie algebras; since
\begin{equation*}
\phi(\psi(\bar{g},\bar{t}),\vec{t})=t\cdot\bar{g}'_{1}t_{1}\bar{g}'^{-1}_{1}\cdot\bar{g}'_{2}t_{2}\bar{g}'^{-1}_{2}\cdot\ldots\cdot\bar{g}'_{\ell}t_{\ell}\bar{g}'^{-1}_{\ell}
\end{equation*}
and $(\ell+1)\dim(\mathfrak{t})<\dim(\mathfrak{g})$, the requirement of a non-singular $\rho_{(\bar{g},\bar{t})}$ translates into asking for the union of the bases of $\mathfrak{t},\bar{g}'_{1}\mathfrak{t}\bar{g}'^{-1}_{1},\ldots,\bar{g}'_{\ell}\mathfrak{t}\bar{g}'^{-1}_{\ell}$ to be a linearly independent set in $\mathfrak{g}$, which means that the spaces have all dimension $\dim(\mathfrak{t})$ and they are transverse to each other (i.e.\ linearly independent). By Corollary~\ref{co:torusoneg} we know that for the choice $\bar{t}=\vec{e}=(e,e,\ldots,e)$ there are $\bar{g}_{i}\in G(\mathbb{F}_{q})$ for which the spaces are of dimension $\dim(\mathfrak{t})$ and transverse: in fact, by definition $\bar{g}_{i}=\bar{g}'_{i}$ when $\bar{t}$ is the identity tuple. Being transverse means that, fixing a basis of $\mathfrak{t}$, there is at least one nonzero minor in the $((\ell+1)\dim(\mathfrak{t}),N^{2})$-matrix defined by the bases of $\mathfrak{t},\bar{g}_{1}\mathfrak{t}\bar{g}^{-1}_{1},\ldots,\bar{g}_{\ell}\mathfrak{t}\bar{g}^{-1}_{\ell}$, where $N^{2}$ is the dimension of the ambient space for the corresponding $\mathfrak{g}$.

Let $\bar{t}=\vec{e}$, and let $\bar{g}$ be a fixed tuple for which we have transverse spaces of maximal dimension. Fix one specific $((\ell+1)\dim(\mathfrak{t}),(\ell+1)\dim(\mathfrak{t}))$-submatrix $M$ as above whose corresponding minor $F$ is nonzero. We define the variety
\begin{equation*}
\tilde{X}_{(1)}=\{g_{1}\in\mathbb{A}^{N^{2}}:F(g_{1},\bar{g}_{2},\ldots,\bar{g}_{\ell})=0\},
\end{equation*}
where the $\bar{g}_{2},\ldots,\bar{g}_{\ell}$ are treated like fixed constants; as we have already observed, $X_{(1)}=\tilde{X}_{(1)}\cap G$ is proper in $G$, so we set $\tilde{X}=\tilde{X}_{(1)}\times \mathbb{A}^{N^{2}(\ell-1)}$ and $X=X_{(1)}\times G^{\ell-1}\subsetneq G^{\ell}$.

Since any $x^{-1}\in G$ has entries that are polynomial in the entries of $x$, we can extend $\psi$ to a map $\tilde{\psi}:\mathbb{A}^{N^{2}(2\ell+1)}\rightarrow\mathbb{A}^{N^{2}\ell}$. The pullback $\tilde{Y}=\tilde{\psi}^{-1}(\tilde{X})$ is defined by the same $F$ in the corresponding spaces $\bar{g}'_{i}\mathfrak{t}\bar{g}'^{-1}_{i}$, and $Y=\tilde{Y}\cap(G^{\ell}\times T^{\ell+1})=\psi^{-1}(X)$ is proper in $G^{\ell}\times T^{\ell+1}$. First of all $Y\supseteq Y_{G^{\ell}\times T^{\ell+1}}$, since for the latter by definition all minors in the collection from which we picked $M$ are zero. Moreover, since $Y|_{\vec{y}=\vec{t}}={\psi_{\vec{t}}}^{-1}(X)=\psi_{\vec{t}^{-1}}(X)$ for every $\vec{t}$, the properness of $X$ implies that $Y|_{\vec{y}=\vec{t}}$ is proper in $G^{\ell}$ as well. By definition, if $\vec{g}\not\in X$ then $(D\phi_{\vec{g}})|_{\vec{y}=\vec{e}}$ is non-singular, so $(Y|_{\vec{x}=\vec{g}})(\mathbb{F}_{q})$ is proper in $T^{\ell+1}(\mathbb{F}_{q})$ because it does not contain $\vec{e}$.

Finally, let us deal with degrees. $M$ has $\dim(\mathfrak{t})$ rows that are subrows of the space $g_{0}\mathfrak{t}g_{0}^{-1}$: thus, their entries are polynomials of degree $2$ in the entries of $g_{0}$, since the inverse map has degree $1$; the other rows of $M$ depend only on the constants $\bar{g}_{i}$, so $\deg(F)\leq 2\dim(\mathfrak{t})\leq 2(r-1)$. Consider then $\tilde{Y}|_{\vec{x}=\vec{g}}$, where $\vec{g}=(g_{1},g_{2},\ldots,g_{\ell})$: it is defined by $F(g'_{1},\bar{g}'_{2},\ldots,\bar{g}'_{\ell})=0$, where by the definition of $\psi$ we have $\bar{g}'_{\ell}=\bar{g}_{\ell}$, $\bar{g}'_{j}=\bar{g}'_{j+1}y_{j+1}^{-1}\bar{g}_{j+1}^{-1}\bar{g}_{j}$ for $1<j\leq\ell-1$, and $g'_{1}=\bar{g}'_{2}y_{2}^{-1}\bar{g}_{2}^{-1}g_{1}$ (thus, the $g_{i}$ for $i>1$ do not matter). The $\dim(\mathfrak{t})$ rows that come from the space $\bar{g}'_{i}\mathfrak{t}\bar{g}'^{-1}_{i}$ have entries that are polynomial of degree $\ell-i$ in the entries of $\vec{y}$; similarly, the entries of the rows of $g'_{1}\mathfrak{t}g'^{-1}_{1}$ are polynomial of degree $\ell$ in the entries of $\vec{y}$, and finally the entries of the rows of $\mathfrak{t}$ are constant, so
\begin{equation*}
\deg(\tilde{Y}|_{\vec{x}=\vec{g}})\leq\sum_{i=1}^{\ell}(\ell-i)\dim(\mathfrak{t})\leq\frac{1}{2}\ell(\ell-1)(r-1),
\end{equation*}
and we are done.
\end{proof}

Thanks to the properties of $\tilde{X},\tilde{Y},X,Y$ proved in Proposition~\ref{pr:ybound}, we show our main result of the section. For technical reasons, we work with a set $B$ generating a conjugate of $G(\mathbb{F}_{q})$ instead of generating $G(\mathbb{F}_{q})$ itself: this is because we can more easily characterize elements that are not regular semisimple inside the canonical diagonal torus, and then shift to any other torus by conjugation.

\begin{theorem}\label{th:torusbound}
Let $G<\mathrm{GL}_{N}$ be an untwisted classical group of rank $r\geq 2$ over $K=\mathbb{F}_{q}$, with $\ell=\frac{\dim(G)}{r}$. Assume that $q\geq e^{3r\log(2r)}$ and $\mathrm{char}(\mathbb{F}_{q})>N$. Fix some $h\in G(\overline{K})$, and let $B\subseteq G(\overline{K})$ be a set of generators of $hG(K)h^{-1}$ with $e\in B=B^{-1}$.

Take a non-maximal torus $T$ of $G$ of degree $\leq r2^{r}$, contained in the canonical maximal torus of Definition~\ref{de:canontorus}. Then, for any $t\geq 1$ we have $|B^{t}\cap T(K)|\leq C_{1}|B^{C_{2}}|^{\frac{1}{\ell+1}}$, with $C_{1}\leq\frac{1}{r(r+1)}(2r)^{10r}$ and $C_{2}\leq(2r)^{14r^{2}}+4rt$.
\end{theorem}

\begin{proof}
For any finite set $S$, we write $S^{\times k}$ for the set $S\times\ldots\times S$ ($k$ times), so as to distinguish it from $S^{k}$. We have $\langle B^{\times\ell}\rangle=\langle B\rangle^{\times\ell}=(hG(K)h^{-1})^{\times\ell}=(hGh^{-1})^{\ell}(K)$. As $T(K)$ is also a group, call $\vec{e}_{i}$ the $i$-tuple $(e,\ldots,e)$ and define
\begin{align*}
B_{0} & =\bigcup_{i=0}^{\ell}\{\vec{e}_{i}\}\times(B\cap T(K))\times\{\vec{e}_{\ell-i}\}, \\
B_{1} & =(B^{t}\cap T(K))^{\times(\ell+1)}, & & (B_{0})^{t}\subseteq B_{1}\subseteq(B_{0})^{(\ell+1)t}, \\
T_{0} & =\langle B_{0}\rangle=\langle B_{1}\rangle.
\end{align*}
We have $T_{0}\leq T^{\ell+1}(K)$ and $\vec{e}=\vec{e}_{\ell+1}\in B_{i}=B_{i}^{-1}$ for $i=0,1$.

First, we select an appropriate tuple of elements of $G$. Let $\tilde{X},\tilde{Y},X,Y$ be as in Proposition~\ref{pr:ybound}: all of them are defined over $K$, and in particular we have $\tilde{X}\cap G^{\ell}=X\subsetneq G^{\ell}$. Call $\bar{h}=(h,h,\ldots,h)\in G^{\ell}(\overline{K})$. Set $V=\bar{h}^{-1}\tilde{X}\bar{h}$, a variety defined over $\overline{K}$ with $V\cap G^{\ell}=\bar{h}^{-1}X\bar{h}\subsetneq G^{\ell}$: by Proposition~\ref{pr:ybound}\eqref{pr:yboundlw}, $V$ satisfies the hypotheses of Corollary~\ref{co:langweil}\eqref{co:langweilstick}. We obtain that
\begin{align*}
 & \ G^{\ell}(K)\supsetneq(\bar{h}^{-1}X\bar{h})(K) \\
\Rightarrow & \ \bar{h}G^{\ell}(K)\bar{h}^{-1}\supsetneq\bar{h}(\bar{h}^{-1}X\bar{h})(K)\bar{h}^{-1}=\bar{h}\bar{h}^{-1}X(\overline{K})\bar{h}\bar{h}^{-1}\cap\bar{h}G^{\ell}(K)\bar{h}^{-1} \\
\Rightarrow & \ \bar{h}G^{\ell}(K)\bar{h}^{-1}\not\subseteq X(\overline{K}).
\end{align*}
Then we escape using Shitov, which is more convenient than either Corollary~\ref{co:escape} or Ariadne's cookbook since the degree of $\tilde{X}$ is very small. Up to a degree-preserving isomorphism, we may think of $G^{\ell}$ as sitting diagonally inside $\mathrm{Mat}_{N\ell}$. Apply then Proposition~\ref{pr:shitov} with $(K,A,N,V)=(\overline{K},B^{\times\ell},N\ell,\tilde{X})$, and with $D=2(r-1)$ by Proposition~\ref{pr:ybound}\eqref{pr:yboundlw}. Thus, there exists $\vec{g}\in(B^{k})^{\times\ell}$ such that $\vec{g}\not\in X(K)$, where
\begin{equation*}
k<22(r-1)(N\ell+1)^{2(r-1)}\log(N\ell).
\end{equation*}

For this choice of $\vec{g}$, define $E=\tilde{Y}|_{\vec{x}=\vec{g}}\cap L^{\ell+1}$, where $L$ is as in \eqref{eq:canoncontain}. Hence
\begin{equation}\label{eq:enonsing}
E\supseteq\tilde{Y}|_{\vec{x}=\vec{g}}\cap T^{\ell+1}=Y|_{\vec{x}=\vec{g}},
\end{equation}
which implies that $(D\phi_{\vec{g}})|_{\vec{y}=\vec{t}}$ is non-singular for any $\vec{t}\not\in E$ by Proposition~\ref{pr:ybound}\eqref{pr:yboundnonsing}. Moreover $E(K)\cap T_{0}\subsetneq T_{0}$, because of Proposition~\ref{pr:ybound}\eqref{pr:yboundprop} and $\vec{e}\in B_{0}\subseteq B_{1}\subseteq T_{0}$. Finally, \eqref{eq:canoncontain} and Proposition~\ref{pr:ybound}\eqref{pr:ybounddeg} gives
\begin{align*}
\dim(E) & \leq(\ell+1)\dim(L)\leq N(\ell+1), & \deg(E) & \leq\deg(\tilde{Y}|_{\vec{x}=\vec{g}})\leq\frac{1}{2}\ell(\ell-1)(r-1).
\end{align*}

By Proposition~\ref{pr:strong-escape}, there are elements $\vec{e}=\vec{b}_{0},\vec{b}_{1},\ldots,\vec{b}_{m_{1}}$ of $B_{0}^{m_{2}}$ such that $\bigcap_{0\leq i\leq m_{1}}\vec{b}_{i}E$ is invariant under the action of $T_{0}$ by left multiplication and
\begin{equation*}
m_{1},m_{2}\leq\left(1+\frac{1}{\deg(E)-1}\right)\deg(E)^{\dim(E)+1}\leq\frac{3}{2}\left(\frac{1}{2}\ell(\ell-1)(r-1)\right)^{N(\ell+1)+1}.
\end{equation*}
This implies that $T_{0}\cap\bigcap_{0\leq i\leq m_{1}}\vec{b}_{i}E(K)$ is also $T_{0}$-invariant, and since $E(K)\cap T_{0}\subsetneq T_{0}$ it must be that
\begin{equation*}
T_{0}\cap\bigcap_{0\leq i\leq m_{1}}\vec{b}_{i}E(K)=\emptyset.
\end{equation*}
Hence, each element of $B_{1}$ is outside $\vec{b}_{i}E(K)$ for some $i$, implying
\begin{align}
|B^{t}\cap T(K)|^{\ell+1} & =|B_{1}|\leq\sum_{i=0}^{m_{1}}|B_{1}\cap(T_{0}\setminus\vec{b}_{i}E(K))| \nonumber \\
 & \leq (m_{1}+1)|B_{1}\cap(T_{0}\setminus\vec{b}_{i_{0}}E(K))|, \label{eq:torusaway}
\end{align}
where $i_{0}$ is the index for which the intersection above has maximal size.

Say $\vec{g}=(g_{1},\ldots,g_{\ell})$ and $\vec{b}_{i_{0}}=(b_{0},b_{1},\ldots,b_{\ell})$. Consider now the map $\xi:T^{\ell+1}\rightarrow G$ defined as follows:
\begin{equation}\label{eq:suitmaptorus}
\xi(t_{0},t_{1},t_{2},\ldots,t_{\ell})=b_{0}^{-1}t_{0}\cdot g_{1}b_{1}^{-1}t_{1}g_{1}^{-1}\cdot g_{2}b_{2}^{-1}t_{2}g_{2}^{-1}\cdot\ldots\cdot g_{\ell}b_{\ell}^{-1}t_{\ell}g_{\ell}^{-1}.
\end{equation}
By definition, $\xi$ is the composition of the left multiplication by $\vec{b}_{i_{0}}^{-1}$ and of $\phi_{\vec{g}}$. By \eqref{eq:enonsing}, $(D\phi_{\vec{g}})|_{\vec{y}=\vec{t}}$ is non-singular for any $\vec{t}\not\in E$, so $\xi^{-1}(\xi(\vec{t}))\setminus\vec{b}_{i_{0}}E$ is $0$-dimensional for all $\vec{t}\in T^{\ell+1}\setminus\vec{b}_{i_{0}}E$ by Lemma~\ref{le:carvefibre}. Furthermore, $\xi$ has been chosen so that $\xi(B_{1})\subseteq B^{m_{2}+(\ell+1)t+2k\ell}$. Set
\begin{equation*}
C_{2}=m_{2}+(\ell+1)t+2k\ell\leq(2r)^{14r^{2}}+4rt,
\end{equation*}
where the inequality can be verified directly for $r\leq 20$, say, and deduced from $k\leq m_{2}\leq(2r)^{13r^{2}}$ otherwise. We obtain
\begin{align}
|B_{1}\cap(T_{0}\setminus\vec{b}_{i_{0}}E(K))| & =\sum_{z\in B^{C_{2}}}|\xi^{-1}(z)\cap B_{1}\cap (T_{0}\setminus\vec{b}_{i_{0}}E(K))| \nonumber \\
 & \leq|B^{C_{2}}|\cdot\max_{\vec{t}\in T^{\ell+1}\setminus\vec{b}_{i_{0}}E}|B_{1}\cap(\xi^{-1}(\xi(\vec{t}))\setminus\vec{b}_{i_{0}}E)(K)| \nonumber \\
 & \leq|B^{C_{2}}|\cdot\max_{\vec{t}\in T^{\ell+1}\setminus\vec{b}_{i_{0}}E}\deg(\overline{\xi^{-1}(\xi(\vec{t}))\setminus\vec{b}_{i_{0}}E}), \nonumber \\
 & \leq|B^{C_{2}}|\cdot\max_{\vec{t}\in T^{\ell+1}\setminus\vec{b}_{i_{0}}E}\deg(T)^{\ell+1}\mdeg(\xi)^{N^{2}}, \label{eq:torusimfi}
\end{align}
where the last inequality of \eqref{eq:torusimfi} holds by Lemma~\ref{le:carvefibre}.

Finally, note that $\deg(T)\leq r2^{r}$ by definition of $T$, and that $\mdeg(\xi)=\mdeg(\phi_{\vec{g}})=\ell+1$. Putting these estimates together with \eqref{eq:torusaway}--\eqref{eq:torusimfi}, we obtain
\begin{equation*}
|B^{t}\cap T(K)|\leq(m_{1}+1)^{\frac{1}{\ell+1}}\cdot r2^{r}\cdot(\ell+1)^{\frac{N^{2}}{\ell+1}}\cdot|B^{C_{2}}|^{\frac{1}{\ell+1}}<\frac{(2r)^{10r}}{r(r+1)}|B^{C_{2}}|^{\frac{1}{\ell+1}},
\end{equation*}
with the last inequality true by Table~\ref{ta:basicg}.
\end{proof}

The result above, as we mentioned before, allows for an easy generalization to any torus, so that we get the following consequence.

\begin{corollary}\label{co:torusnonrs}
Let $G=G_{n}<\mathrm{GL}_{N}$ be an untwisted classical group of rank $r$ over $K=\mathbb{F}_{q}$. Assume that $q\geq e^{3r\log(2r)}$ and that $\mathrm{char}(\mathbb{F}_{q})>N$. Let $A$ be a set of generators of $G(K)$ with $e\in A=A^{-1}$. Let $T$ be a maximal torus of $G$.

Call $W$ the set of elements of $G(\overline{K})$ that are not regular semisimple. Then, for any $t\geq 1$, we have $|A^{t}\cap T(K)\cap W|\leq r(r+1)C_{1}|A^{C_{2}}|^{\frac{1}{\ell+1}}$ with $C_{1},C_{2}$ as in Theorem~\ref{th:torusbound}.
\end{corollary}

\begin{proof}
Write $T_{\mathrm{can}}$ for the canonical maximal torus of $G$ as in Definition~\ref{de:canontorus}. All maximal tori of $G$ are conjugate via some element over $\overline{K}$, so there is some $h\in G(\overline{K})$ such that $T_{\mathrm{can}}=hTh^{-1}$. Thus, we can move our problem back to the diagonal case via conjugation by $h$: the set $B=hAh^{-1}$ is a symmetric set generating $hG(K)h^{-1}$, and $W$ is invariant under conjugation.

By definition, $T_{\mathrm{can}}(K)\cap W$ collects the elements whose eigenvalues (in $\overline{K}$) are not all distinct: we claim that, for all $G$, each $x\in T_{\mathrm{can}}(K)\cap W$ falls into a subtorus $T'\subsetneq T_{\mathrm{can}}$ of degree $\leq r2^{r}$. First, we see from Definition~\ref{de:canontorus} that $\deg(T_{\mathrm{can}})\leq r2^{r}$ for $G=\mathrm{SL}_{n}$ and $\deg(T_{\mathrm{can}})\leq 2^{r}$ in all the other cases; $T'$ is then defined by one additional equation, whose degree we can easily bound. For $G=\mathrm{SL}_{n}$, calling $x_{i}$ the diagonal entries, the equation is of the form $x_{i}=x_{j}$. For $G=\mathrm{SO}_{2n}^{+}$ and $G=\mathrm{Sp}_{2n}$, calling $x_{i}$ the first $n$ diagonal entries, we have either $x_{i}=x_{j}$ or $x_{i}x_{j}=1$. Finally, $G=\mathrm{SO}_{2n+1}$ is the same as for $\mathrm{SO}_{2n}^{+}$ with the addition of $x_{i}=1$. Hence, in all cases we get $\deg(T')\leq r2^{r}$.

In each $G$, there are at most $r(r+1)$ subtori $T'$ defined as above, and all of them are contained in $T_{\mathrm{can}}$. Let $\mathcal{T}$ be the collection of such $T'$. When $r\geq 2$, for each $T'\in\mathcal{T}$ we apply Theorem~\ref{th:torusbound} and get
\begin{align*}
|A^{t}\cap T(K)\cap W| & =|B^{t}\cap T_{\mathrm{can}}(K)\cap W|=\sum_{T'\in\mathcal{T}}|B^{t}\cap T'(K)| \\
 & \leq r(r+1)C_{1}|B^{C_{2}}|^{\frac{1}{\ell+1}}=r(r+1)C_{1}|A^{C_{2}}|^{\frac{1}{\ell+1}}.
\end{align*}
Suppose finally that $r=1$. By Definition~\ref{de:ambient}, the only case is $G=\mathrm{SL}_{2}$, for which we have $W=\{\pm\mathrm{Id}_{2}\}$. Then the bound holds by direct verification.
\end{proof}


\section{Diameter bounds}\label{se:diam}

In the present section, we finally prove the Main Theorem, using the dimensional estimates contained in Theorem~\ref{th:cl(g)-bound} and Corollary~\ref{co:torusnonrs}.

From now on, the group $G=\mathrm{SL}_{n}$ is defined using the more common embedding inside $\mathrm{Mat}_{n}$, instead of the embedding inside $\mathrm{Mat}_{2n}$ that was presented in \S\ref{se:chev}. We can change definition since the two dimensional estimates referenced above still hold, and $\mathrm{diam}(G(\mathbb{F}_{q}))$ does not depend on the embedding of $G$. Some computations improve as a result of this change.

We start with a couple of facts on some specific subvarieties of untwisted classical groups.

\begin{lemma}\label{le:hth}
Let $G$ be an untwisted classical group of rank $r$ over $\mathbb{F}_{q}$, and let $T$ be a maximal torus in $G$. Assume that $\mathrm{char}(\mathbb{F}_{q})>2$. Then, the number of distinct conjugates of $T$ by elements of $G(\mathbb{F}_{q})$ is
\begin{equation*}
|\{hTh^{-1}:h\in G(\mathbb{F}_{q})\}|\geq\frac{(q-1)^{\dim(G)}}{r!2^{r}(q+1)^{r}}\geq\frac{1}{r!4^{r}}(q-1)^{\dim(G)-r}.
\end{equation*}
\end{lemma}

\begin{proof}
We may suppose that $T$ itself is defined over $\mathbb{F}_{q}$, up to taking a conjugate by some element of $G(\overline{\mathbb{F}_{q}})$. Then, the normalizer $N(T)$ of $T$ in $G$ is also defined over $\mathbb{F}_{q}$, and so is the action of $N(T)$ on $T$. Let $\mathcal{W}(G)\simeq N(T)/T$ denote the {\em Weyl group} of $G$, as in \cite[\S 11.19]{Bor91}: by what we said above, $\mathcal{W}(G)$ itself can be defined as a variety over $\mathbb{F}_{q}$ (see \cite[\S 6.3]{Bor91} for quotients of varieties).

By \cite[\S 2.8.4]{Wil09}, we know that
\begin{align*}
\mathcal{W}(\mathrm{SL}_{r+1}) & \simeq\mathrm{Sym}(r+1), & \mathcal{W}(\mathrm{SO}_{2r}^{+}) & \simeq H_{r}\rtimes\mathrm{Sym}(r), \\
\mathcal{W}(\mathrm{SO}_{2r+1}) & \simeq\{\pm 1\}^{r}\rtimes\mathrm{Sym}(r), & \mathcal{W}(\mathrm{Sp}_{2r}) & \simeq\{\pm 1\}^{r}\rtimes\mathrm{Sym}(r),
\end{align*}
where $H_{r}$ is a subgroup of $\{\pm 1\}^{r}$ of index $2$, so in all cases $|\mathcal{W}(G)|\leq r!2^{r}$. In particular, $N(T)$ is the disjoint union of at most $r!2^{r}$ varieties of the form $n_{i}T$. Some of the $n_{i}T$ may not have $\mathbb{F}_{q}$-points, but those that do must have the same number of them: in fact, if $x_{1}\in(n_{1}T)(\mathbb{F}_{q})$ and $x_{2}\in(n_{2}T)(\mathbb{F}_{q})$, the multiplication maps by $x_{1}^{-1}x_{2}$ and by $x_{2}^{-1}x_{1}$ are morphisms defined over $\mathbb{F}_{q}$. Hence, $|N(T)(\mathbb{F}_{q})|/|T(\mathbb{F}_{q})|\leq|\mathcal{W}(G)|\leq r!2^{r}$.

On the other hand $|\{hTh^{-1}:h\in G(\mathbb{F}_{q})\}|=|G(\mathbb{F}_{q})|/|N(T)(\mathbb{F}_{q})|$: $h_{1},h_{2}$ in fact yield the same torus if and only if $h_{1}^{-1}h_{2}\in N(T)$. By \eqref{eq:numelgr} we have $|G(\mathbb{F}_{q})|\geq(q-1)^{\dim(G)}$, and by \cite[I.1.5]{Oes84} we have $|T(\mathbb{F}_{q})|\leq(q+1)^{r}$. The desired inequality follows.
\end{proof}

\begin{proposition}\label{pr:nonrs}
Let $G<\mathrm{GL}_{N}$ be an untwisted classical group of rank $r$ over $\mathbb{F}_{q}$, and let
\begin{equation*}
\mathfrak{B}=\{g\in G(\mathbb{F}_{q}):g\text{ not regular semisimple}\}.
\end{equation*}
Then, for any symmetric set $A$ of generators of $G(\mathbb{F}_{q})$ with $e\in A$, there exists some $g\in A^{k}\setminus(A^{k}\cap\mathfrak{B})$ with $k=(2r)^{3r^{2}+3r}$.
\end{proposition}

\begin{proof}
By Definition~\ref{de:ambient}, $G\subseteq\mathrm{SL}_{N}$. An element $g\in\mathrm{SL}_{N}(\mathbb{F}_{q})$ is regular semisimple if and only if the discriminant of its characteristic polynomial is nonzero. Hence, if we call this discriminant $P$, the variety $W=\{P=0\}$ is such that $\mathfrak{B}=(G\cap W)(\mathbb{F}_{q})$. There is a regular semisimple element in $G(\mathbb{F}_{q})$ by Proposition~\ref{pr:rsq}, so $G\cap W$ is proper inside $G$, and in particular $\dim(G\cap W)<\dim(G)$ since $G$ is irreducible. Thus, we only need to bound $\deg(W)$ and then use the escape results of Section~\ref{se:escape}. If we call $D$ the degree of $P$ in the entries of $g$, we have by definition $\deg(W)\leq D$, so the goal is to bound $D$.

Let $G=\mathrm{SL}_{n}$; recall that in this section we adopt the embedding of $G$ inside $\mathrm{Mat}_{n}$. By the quasi-homogeneity of the resultant (with weight $i$ for the coefficient $a_{i}$ in the polynomial $\sum_{i=0}^{n}a_{i}x^{n-i}$), we have $D\leq n(n-1)=r(r+1)$. Then, using Corollary~\ref{co:escape} on $W$, the number of steps in which we find $g$ regular semisimple is bounded by
\begin{equation*}
k\leq\left(1+\frac{1}{D-1}\right)D^{n^{2}}\leq(2r)^{2r^{2}+3r}.
\end{equation*}

We argue similarly for $G\neq\mathrm{SL}_{n}$. Since the characteristic polynomial is symmetric, the degrees of the coefficients of $x^{i}$ and of $x^{N-i}$ for $i\leq\frac{N}{2}$ are both equal to $i$: hence, from the definition of resultant it is easy to see that
\begin{equation*}
\deg(P)\leq N\cdot\left\lfloor\frac{N}{2}\right\rfloor+2\cdot\sum_{i=1}^{\lfloor N/2\rfloor-1}i+1\cdot 0=\left\lfloor\frac{N}{2}\right\rfloor\left(N+\left\lfloor\frac{N}{2}\right\rfloor-1\right)
\end{equation*}
(taking the maximum degree row by row). By Proposition~\ref{pr:shitov} we have
\begin{equation*}
k\leq 11D(N+1)^{D}\log N\leq(2r)^{3r^{2}+3r}
\end{equation*}
using Table~\ref{ta:basicg}.
\end{proof}

Let us also introduce a result that deals separately with the case of $|A|$ very large. See~\cite[Prop.~64]{PS16}, which is essentially the same.

\begin{proposition}\label{pr:np11}
Let $G$ be an untwisted classical group over $\mathbb{F}_{q}$. Assume that $q>9$ and $\mathrm{char}(\mathbb{F}_{q})>2$.

If $r$ is the rank of $G$ and $A\subseteq G(\mathbb{F}_{q})$ is such that $|A|\geq\frac{4}{3}q^{\dim(G)-\frac{r}{3}}$, then $A^{3}=G(\mathbb{F}_{q})$.
\end{proposition}

\begin{proof}
For any finite group $H$ with minimal degree of a complex representation equal to $\delta$, if $|A|>\delta^{-\frac{1}{3}}|H|$ then $A^{3}=H$ (see~\cite[Cor.~1]{NP11}). In our case, the minimal degree is at least $\frac{q^{r}-1}{2}$: see the entries of~\cite[Table II]{TZ96}, which refer to the corresponding simple groups, and note that the same values hold for the groups in our statement (i.e.\ up to quotient by the centre $Z(G)$, see~\cite[p.~419]{LS74}).

By Section~\ref{se:chev}, we have $|G(\mathbb{F}_{q})|\leq q^{\dim(G)-\frac{r}{3}}(q^{r}-1)^{\frac{1}{3}}$ for $G\neq\mathrm{SL}_{2}$, whereas $|\mathrm{SL}_{2}(\mathbb{F}_{q})|\leq 1.04q^{\frac{8}{3}}(q-1)^{\frac{1}{3}}$ for $q\geq 5$; thus
\begin{equation*}
\left(\frac{2}{q^{r}-1}\right)^{\frac{1}{3}}|G(\mathbb{F}_{q})|\leq\frac{4}{3}q^{\dim(G)-\frac{r}{3}}
\end{equation*}
in all cases, and we obtain the result.
\end{proof}

In order to provide an estimate for $C$ in Conjecture~\ref{conj:Babai}, we first establish a growth result that is as explicit as possible, assuming that we already know some specific dimensional estimates. To sum up, by Theorem~\ref{th:cl(g)-bound} and Corollary~\ref{co:torusnonrs}, we know that for any $l\geq 1$, any regular semisimple element $g$ of $G$, and any maximal torus $T$ of $G$, we have
\begin{align}
|A^{l}\cap\Cl(g)(\mathbb{F}_{q})| & \leq C_{1}|A^{C_{2}(l)}|^{1-\frac{1}{\ell}}, & C_{1} & =(2r)^{18r^{2}}, & C_{2}(l) & =(2r)^{17r^{2}}+2l, \label{eq:de-clg} \\
|A^{l}\cap T(\mathbb{F}_{q})\cap\mathfrak{B}| & \leq C'_{1}|A^{C'_{2}(l)}|^{\frac{1}{\ell+1}}, & C'_{1} & =(2r)^{10r}, & C'_{2}(l) & =(2r)^{14r^{2}}+4rl, \label{eq:de-tnonrs}
\end{align}
where $\mathfrak{B}$ is the set of elements of $G(\mathbb{F}_{q})$ that are not regular semisimple. Note that $1-\frac{1}{\ell}=\frac{\dim(\Cl(g))}{\dim(G)}$, and that $\frac{1}{\ell+1}\geq\frac{\dim(T')}{\dim(G)}$ for any non-maximal torus $T'$.

\begin{theorem}\label{thm:growth}
Let $G<\mathrm{GL}_{N}$ be an untwisted classical group of rank $r$ over $K=\mathbb{F}_{q}$. Assume that $q\geq e^{8r\log(2r)}$ and that $\mathrm{char}(K)>N$. Let $A$ be a symmetric set of generators of $G(K)$.

Then, for any $l\geq 1$, either $A^{3l}=G(K)$ or $|A^{m(l)}|\geq\frac{1}{c}|A^{l}|^{1+\varepsilon}$ for at least one of the following triples:
\begin{align}
(c,m(l),\varepsilon) & =\left((2r)^{32r^{2}},(2r)^{14r^{2}}+8rl,\frac{1}{12r}\right), \label{eq:growth-a} \\
(c,m(l),\varepsilon) & =\left((2r)^{31r^{2}},(2r)^{18r^{2}}+4l,\frac{1}{15r^{2}}\right), \label{eq:growth-b} \\
(c,m(l),\varepsilon) & =\left((2r)^{14r^{2}},(2r)^{18r^{2}}+4l,\frac{1}{20r}\right). \label{eq:growth-c}
\end{align}
\end{theorem}

\begin{proof}
If $|A^{l}|\geq\frac{4}{3}q^{\dim(G)-\frac{r}{3}}$, then by Proposition~\ref{pr:np11} we have $A^{3l}=G(K)$. From now on, we assume the contrary.

Define a maximal torus $T$ of $G$ to be \textit{involved} if $A^{k'}\cap T(K)$ contains a regular semisimple element, where $k'=\max\{k,2l\}$ and $k$ is as in Proposition~\ref{pr:nonrs}; this term goes back to~\cite[Def.~5.2]{BGT11}. By definition of $k$, there is a regular semisimple $g\in A^{k}$, which then lies in a (unique) maximal torus $T$, i.e.\ $g\in A^{k}\cap T(K)$: in other words, there exists an involved torus $T$. The proof will be articulated into two cases:
\begin{enumerate}[{Case} 1:]
\item\label{it:invnoinv} There are a maximal torus $T$ and an element $h\in A$ such that $T$ is involved and $hTh^{-1}$ is not involved.
\item\label{it:allinv} For any $h\in G$, all maximal tori $hTh^{-1}$ are involved where $T$ is our initial involved torus.
\end{enumerate}

The principles behind the proof that we are going to present appear in many versions across the literature: see for instance the pivoting argument in~\cite[\S 5]{Hel19b} and~\cite[\S 5]{BGT11}, or the use of full CCC-subgroups (read: involved maximal tori) in~\cite[\S 8]{PS16}. We follow in particular the proof of~\cite[Thm.~2]{RS18}, which we generalize to untwisted classical groups.

\subsubsection*{Case~\ref{it:invnoinv}:} By hypothesis there is a regular semisimple $g\in A^{k'}\cap T(K)$ for $k'\geq 2$ as in the definition of involved. This implies that $g'=hgh^{-1}\in A^{k'+2}\cap T'(K)$, which defines a map
\begin{align*}
\psi & :A^{l}\rightarrow\Cl(g')(K), & a & \mapsto ag'a^{-1}.
\end{align*}
If $a_{1},a_{2}\in A^{l}$ have the same image under $\psi$, then $a_{1}^{-1}a_{2}\in C(g')=T'$, so $a_{1}^{-1}a_{2}$ is an element of $A^{2l}\cap T'(K)$. By hypothesis $T'$ is not involved, so in particular $a_{1}^{-1}a_{2}\in\mathfrak{B}$, where $\mathfrak{B}$ as in Proposition~\ref{pr:nonrs} denotes the set of elements of $G(K)$ that are not regular semisimple: therefore, if $x$ is any element in the image of $\psi$ and $\psi^{-1}(x)=\{a_{i}\}_{i}$, we can fix $a_{1}$ and note that the $a_{1}^{-1}a_{i}$ are all distinct and in $A^{2l}\cap T'(K)\cap\mathfrak{B}$. Combining this with~\eqref{eq:de-clg}--\eqref{eq:de-tnonrs}, we get
\begin{equation*}
\frac{|A^{l}|}{C'_{1}|A^{C'_{2}(2l)}|^{\frac{1}{\ell+1}}}\leq\frac{|A^{l}|}{|A^{2l}\cap T'(K)\cap\mathfrak{B}|}\leq|A^{k'+2l+2}\cap\Cl(g')(K)|\leq C_{1}|A^{C_{2}(k'+2l+2)}|^{1-\frac{1}{\ell}},
\end{equation*}
or in other words
\begin{equation}\label{eq:growthdouble}
|A^{C_{2}(k'+2l+2)}|^{1-\frac{1}{\ell}}|A^{C'_{2}(2l)}|^{\frac{1}{\ell+1}}\geq\frac{|A^{l}|}{C_{1}C'_{1}}.
\end{equation}
Then we refine~\eqref{eq:growthdouble}, by giving upper bounds for $x^{1-\frac{1}{\ell}}y^{\frac{1}{\ell+1}}$ involving only one variable.
\begin{enumerate}[(1)]
\item\label{it:xyy} If $x\leq y^{1-\frac{1}{6\ell}}$, then we have $\left(1-\frac{1}{\ell}\right)\left(1-\frac{1}{6\ell}\right)+\frac{1}{\ell+1}=1-\frac{\ell^{2}+6\ell-1}{6\ell^{2}(\ell+1)}$, so $x^{1-\frac{1}{\ell}}y^{\frac{1}{\ell+1}}\leq y^{1-\frac{\ell^{2}+6\ell-1}{6\ell^{2}(\ell+1)}}$.
\item\label{it:xyx} If $x>y^{1-\frac{1}{6\ell}}$, then we have $1-\frac{1}{\ell}+\frac{1}{\ell+1}\left(1-\frac{1}{6\ell}\right)^{-1}=1-\frac{5\ell-1}{\ell(\ell+1)(6\ell-1)}$, so $x^{1-\frac{1}{\ell}}y^{\frac{1}{\ell+1}}<x^{1-\frac{5\ell-1}{\ell(\ell+1)(6\ell-1)}}$.
\end{enumerate}
Using the values in Table~\ref{ta:basicg}, we also have
\begin{align}
\left(1-\frac{\ell^{2}+6\ell-1}{6\ell^{2}(\ell+1)}\right)^{-1} & >1+\frac{\ell^{2}+6\ell-1}{6\ell^{2}(\ell+1)}\geq 1+\frac{1}{12r}, \label{eq:simpler1} \\
\left(1-\frac{5\ell-1}{\ell(\ell+1)(6\ell-1)}\right)^{-1} & >1+\frac{5\ell-1}{\ell(\ell+1)(6\ell-1)}\geq 1+\frac{1}{15r^{2}}, \label{eq:simpler2}
\end{align}
where the worst cases are respectively $G\neq\mathrm{SL}_{r+1}$ for $r\rightarrow\infty$, and $G=\mathrm{SL}_{2}$. Thus,~\eqref{eq:growthdouble} implies one of the two possibilities below: if \eqref{it:xyy} applies, then by \eqref{eq:de-clg}--\eqref{eq:de-tnonrs}--\eqref{eq:simpler1}
\begin{equation}\label{eq:growth11}
|A^{C'_{2}(2l)}|\geq\left(\frac{|A^{l}|}{C_{1}C'_{1}}\right)^{1+\frac{1}{12r}}\geq\frac{|A^{l}|^{1+\frac{1}{12r}}}{(2r)^{31r^{2}}},
\end{equation}
whereas if \eqref{it:xyx} applies, then by \eqref{eq:de-clg}--\eqref{eq:de-tnonrs}--\eqref{eq:simpler2}
\begin{equation}\label{eq:growth12}
|A^{C_{2}(k'+2l+2)}|\geq\left(\frac{|A^{l}|}{C_{1}C'_{1}}\right)^{1+\frac{1}{15r^{2}}}\geq\frac{|A^{l}|^{1+\frac{1}{15r^{2}}}}{(2r)^{30r^{2}}}.
\end{equation}

\subsubsection*{Case~\ref{it:allinv}:} We are in the case when for all maximal tori $T'$ of the form $hTh^{-1}$ there exists a regular semisimple element $g'\in A^{k'}\cap T'(K)$, where $T$ is our initial maximal torus and $k'=\max\{k,2l\}$ with $k$ as in Proposition~\ref{pr:nonrs}. We fix one $g'$ arbitrarily for each $T'$; for any one of these $T'$, consider the map
\begin{align*}
\psi_{T'} & :A^{l}\rightarrow\Cl(g')(K), & \psi_{T'}(a) & =ag'a^{-1}.
\end{align*}
Observe that $\psi_{T'}(A^{l})\subseteq A^{k'+2l}\cap\Cl(g')(K)$, and that in fact $\Cl(g')=\Cl(g)$ for all the $g'$ we are considering, as the corresponding tori are all conjugate. By definition, if $x$ is the element of $\Cl(g)$ with fibre $\psi_{T'}^{-1}(x)$ of maximal size, we must have
\begin{equation}\label{eq:fibre}
|A^{k'+2l}\cap\Cl(g)(K)|\geq|\psi_{T'}(A^{l})|\geq\frac{|A^{l}|}{|\psi_{T'}^{-1}(x)|}.
\end{equation}
As in Case~\ref{it:invnoinv}, for any two $a_{1},a_{2}\in A^{l}$ lying in the fibre $\psi_{T'}^{-1}(x)$ we have $a_{1}^{-1}a_{2}\in A^{2l}\cap T'(K)$, which implies that $|\psi_{T'}^{-1}(x)|\leq|A^{2l}\cap T'(K)|$. Combining this with~\eqref{eq:fibre}, we get
\begin{equation}\label{eq:fibre2}
|A^{2l}\cap T'(K)|\geq\frac{|A^{l}|}{|A^{k'+2l}\cap\Cl(g)(K)|}
\end{equation}
for each of the tori $T'$. Now, every regular semisimple element of $G$ sits inside a unique maximal torus, so that $G(K)\setminus\mathfrak{B}$ (where $\mathfrak{B}$ is as in Proposition~\ref{pr:nonrs}) is partitioned between the maximal tori of $G$. Moreover, we have an upper bound on the size of the set $\mathcal{T}$ of conjugates of $T$ by Lemma~\ref{le:hth}. Combining these facts with \eqref{eq:de-clg}--\eqref{eq:de-tnonrs}--\eqref{eq:fibre2},
\begin{align}
|A^{2l}| & \geq|A^{2l}\setminus(A^{2l}\cap\mathfrak{B})|\geq\sum_{T'\in\mathcal{T}}(|A^{2l}\cap T'(\mathbb{F}_{q})|-|A^{2l}\cap T'(\mathbb{F}_{q})\cap\mathfrak{B}|) \nonumber \\
 & \geq\frac{(q-1)^{\dim(G)-r}}{r!2^{r}}\left(\frac{|A^{l}|}{C_{1}|A^{C_{2}(k'+2l)}|^{1-\frac{1}{\ell}}}-C'_{1}|A^{C'_{2}(2l)}|^{\frac{1}{\ell+1}}\right). \label{eq:oofdiff}
\end{align}

Suppose that the second term in \eqref{eq:oofdiff} is large, i.e.,
\begin{equation}\label{eq:oofhyp}
C'_{1}|A^{C'_{2}(2l)}|^{\frac{1}{\ell+1}}\geq\frac{|A^{l}|}{2C_{1}|A^{C_{2}(k'+2l)}|^{1-\frac{1}{\ell}}}.
\end{equation}
We are in the same situation as with \eqref{eq:growthdouble}, so we can perform the same refinement. If \eqref{it:xyy} applies, then by \eqref{eq:de-clg}--\eqref{eq:de-tnonrs}--\eqref{eq:simpler1}
\begin{equation}\label{eq:growth211}
|A^{C'_{2}(2l)}|\geq\left(\frac{|A^{l}|}{2C_{1}C'_{1}}\right)^{1+\frac{1}{12r}}\geq\frac{|A^{l}|^{1+\frac{1}{12r}}}{(2r)^{32r^{2}}},
\end{equation}
whereas if \eqref{it:xyx} applies, then by \eqref{eq:de-clg}--\eqref{eq:de-tnonrs}--\eqref{eq:simpler2}
\begin{equation}\label{eq:growth212}
|A^{C_{2}(k'+2l)}|\geq\left(\frac{|A^{l}|}{2C_{1}C'_{1}}\right)^{1+\frac{1}{15r^{2}}}\geq\frac{|A^{l}|^{1+\frac{1}{15r^{2}}}}{(2r)^{31r^{2}}}.
\end{equation}
Suppose instead that \eqref{eq:oofhyp} does not hold. Then, by \eqref{eq:de-clg}--\eqref{eq:oofdiff},
\begin{align}\label{eq:growthq}
C_{1}|A^{2l}||A^{C_{2}(k'+2l)}|^{1-\frac{1}{\ell}} & \geq\frac{(q-1)^{\dim(G)-r}}{r!2^{2r+1}}|A^{l}|\geq\left(1-\frac{1}{e^{8r\log(2r)}}\right)^{2r^{2}}\frac{q^{\left(1-\frac{1}{\ell}\right)\dim(G)}}{r!2^{2r+1}}|A^{l}| \nonumber \\
 & \geq\frac{q^{\left(1-\frac{1}{\ell}\right)\dim(G)}}{r!2^{2r+2}}|A^{l}|.
\end{align}
Call $R=\frac{|A^{C_{2}(k'+2l)}|}{|A^{l}|}\geq\frac{|A^{2l}|}{|A^{l}|}$. Substituting inside~\eqref{eq:growthq}, we obtain
\begin{align*}
C_{1}R|A^{l}|R^{1-\frac{1}{\ell}}|A^{l}|^{1-\frac{1}{\ell}} & \geq\frac{q^{\left(1-\frac{1}{\ell}\right)\dim(G)}}{r!2^{2r+2}}|A^{l}| & & \Longrightarrow & R^{\frac{2\ell-1}{\ell-1}} & \geq\frac{1}{(r!2^{2r+2}C_{1})^{\frac{\ell}{\ell-1}}}\cdot\frac{q^{\dim(G)}}{|A^{l}|}.
\end{align*}
Since we have assumed from the start that $|A^{l}|<\frac{4}{3}q^{\dim(G)-\frac{r}{3}}$, we have
\begin{align*}
q^{\dim(G)} & \geq\left(\frac{3|A^{l}|}{4}\right)^{\frac{1}{1-\frac{1}{3\ell}}}\geq\frac{2}{3}|A^{l}|^{1+\frac{1}{8r}} & & \Longrightarrow & R^{\frac{2\ell-1}{\ell-1}} & \geq\frac{\frac{2}{3}|A^{l}|^{\frac{1}{8r}}}{(r!2^{2r+2}C_{1})^{\frac{\ell}{\ell-1}}},
\end{align*}
implying by \eqref{eq:de-clg}
\begin{equation}\label{eq:growth22}
|A^{C_{2}(k'+2l)}|\geq\frac{|A^{l}|^{1+\frac{\ell-1}{8r(2\ell-1)}}}{(r!2^{2r+3}C_{1})^{\frac{\ell}{2\ell-1}}}\geq\frac{|A^{l}|^{1+\frac{1}{20r}}}{(r!2^{2r+3}C_{1})^{\frac{\ell}{2\ell-1}}}\geq\frac{|A^{l}|^{1+\frac{1}{20r}}}{(2r)^{14r^{2}}}.
\end{equation}
\medskip

Finally we put all the estimates together, and use the value of $k$ from Proposition~\ref{pr:nonrs} and the values of $C_{2},C'_{2}$ from \eqref{eq:de-clg}--\eqref{eq:de-tnonrs}: \eqref{eq:growth11} and \eqref{eq:growth211} yield \eqref{eq:growth-a}, \eqref{eq:growth12} and \eqref{eq:growth212} yield \eqref{eq:growth-b}, and \eqref{eq:growth22} yields~\eqref{eq:growth-c}.
\end{proof}

\subsection{Proof of Main Theorem}

From the conclusion of Theorem~\ref{thm:growth}, we move to the proof of our main result. Let $K=\mathbb{F}_{q}$. For $q<e^{8r\log(2r)}$, from~\cite{HMPQ19} it follows directly that
\begin{equation*}
\mathrm{diam}(G(K))=e^{O(r^{2}(\log r)^{3})},
\end{equation*}
yielding the bound. Assume $q\geq e^{8r\log(2r)}$ instead.

For any $A$ symmetric in $G(K)$ and any $l\geq 1$, by Theorem~\ref{thm:growth} we have either $A^{3l}=G(K)$ or $|A^{m(l)}|\geq\frac{1}{c}|A^{l}|^{1+\varepsilon}$ for $(c,m(l),\varepsilon)$ equal to at least one of the triples \eqref{eq:growth-a}--\eqref{eq:growth-b}--\eqref{eq:growth-c}. Write $m(l)=m_{1}l+m_{0}$.

Assume first that $|A|\geq c^{2/\varepsilon}$. Thus, either $A^{3l}=G(K)$ or $|A^{m_{1}l+m_{0}}|\geq|A^{l}|^{1+\varepsilon/2}$ for every $l\geq 1$. Define $l_{0}=1$ and
\begin{equation*}
l_{i}=m_{1}^{i}+(m_{1}^{i-1}+\ldots+m_{1}+1)m_{0}
\end{equation*}
for all $i\geq 1$. If $j$ is the smallest non-negative integer such that $A^{3l_{j}}=G(K)$, then $|A|^{(1+\varepsilon/2)^{j-1}}\leq|A^{l_{j-1}}|<|G(K)|$. Therefore
\begin{equation*}
j<1+\frac{1}{\log(1+\frac{\varepsilon}{2})}\log\left(\frac{\log|G(K)|}{\log|A|}\right),
\end{equation*}
which implies
\begin{align*}
\mathrm{diam}(G(K),A) & \leq 3l_{j}\leq 3m_{1}^{j}(1+m_{0})<3m_{1}(1+m_{0})\left(\frac{\log|G(K)|}{\log|A|}\right)^{\frac{\log m_{1}}{\log(1+\varepsilon/2)}}.
\end{align*}
For any choice of $c,m_{0},m_{1},\varepsilon$ we have
\begin{align*}
3m_{1}(1+m_{0}) & \leq(2r)^{22r^{2}}, & \frac{\log m_{1}}{\log(1+\frac{\varepsilon}{2})} & \leq\max\left\{\frac{r\log(8r)}{\log\left(\frac{25}{24}\right)},\frac{r^{2}\log 4}{\log\left(\frac{31}{30}\right)},\frac{r\log 4}{\log\left(\frac{41}{40}\right)}\right\}\leq 57r^{2},
\end{align*}
thus giving
\begin{equation}\label{eq:diamm}
\mathrm{diam}(G(K),A)\leq(2r)^{22r^{2}}\left(\frac{\log|G(K)|}{\log|A|}\right)^{57r^{2}}\leq(\log|G(K)|)^{57r^{2}}
\end{equation}
because $|A|$ is large enough.

Now assume that $|A|<c^{2/\varepsilon}$. Then, trivially, $|A^{\lfloor c^{2/\varepsilon}\rfloor}|\geq c^{2/\varepsilon}$. Applying \eqref{eq:diamm} to $A^{\lfloor c^{2/\varepsilon}\rfloor}$ we obtain
\begin{equation*}
\mathrm{diam}(G(K),A)\leq c^{2/\varepsilon}\mathrm{diam}(G(K),A^{\lfloor c^{2/\varepsilon}\rfloor})\leq(2r)^{930r^{4}}(\log|G(K)|)^{57r^{2}}
\end{equation*}
yielding the result.

As $q\geq e^{8r\log(2r)}$ by assumption, using \eqref{eq:numelgr} we have also the two bounds
\begin{align*}
\log|G(K)| & \geq r^{2}\log q\geq(2r)^{3}\log 2, \\
\log|G(K)| & \geq\log|\mathrm{SL}_{2}(\mathbb{F}_{q})|\geq 16.63\geq(\log 2)^{-7.67}.
\end{align*}
Hence, we can simplify the result as
\begin{align*}
\mathrm{diam}(G(K),A) & \leq\left(\frac{\log|G(K)|}{\log 2}\right)^{\frac{930}{3}r^{4}}(\log|G(K)|)^{57r^{2}} \\
 & \leq(\log|G(K)|)^{\frac{930}{3}\left(1+\frac{1}{7.67}\right)r^{4}+57r^{2}}\leq(\log|G(K)|)^{408r^{4}}.
\end{align*}
The proof of the Main Theorem is concluded.

\subsection{Asymptotics for $r\rightarrow\infty$}\label{se:asymp}

We end the section by giving a quick overview of how to retrieve an explicit asymptotic estimate for $\mathrm{diam}(G(\mathbb{F}_{q}))$ when the rank of $G$ goes to infinity. Since there is no qualitative difference in the proofs, we describe how the intermediate results change without reproving them.

Let $r\rightarrow\infty$; hereafter, $o(1)$ denotes a function $f$ with $\lim_{r\rightarrow\infty}f(r)=0$. The conditions on $q$ in Corollary~\ref{co:langweil}\eqref{co:langweilariadne}-\eqref{co:langweilstick} become $q\geq e^{(2+o(1))r\log r}$. For $g\in G$ regular semisimple, the conjugacy class $\Cl(g)$ is a variety of degree $2^{(2+o(1))r^{2}}$ (Proposition~\ref{pr:clvar}). Consequently, for the map $f$ defined in \eqref{eq:f}, the number of points in a generic fibre as in \eqref{eq:fibresize} can be bounded by $r^{(1+o(1))r\ell}$. Since $m,k$ in \eqref{eq:clmk}-\eqref{eq:clmk2} are at most $r^{(12+o(1))r^{2}}$, we obtain in \eqref{eq:combinedbound} that Theorem~\ref{th:cl(g)-bound} holds with
\begin{align*}
C_{1} & =r^{(12+o(1))r^{2}}, & C_{2} & =r^{(12+o(1))r^{2}}+2t.
\end{align*}
Now, using the definitions and degree bounds contained in Proposition~\ref{pr:ybound}, the values of $m,k$ in the proof of Theorem~\ref{th:torusbound} are at most
\begin{align*}
m & =r^{(6+o(1))r(\ell+1)}, & k & =r^{(4+o(1))r}.
\end{align*}
Thus, putting together \eqref{eq:torusaway} and \eqref{eq:torusimfi}, Theorem~\ref{th:torusbound} holds with
\begin{align*}
C_{1} & =r^{(10+o(1))r}, & C_{2} & =r^{(12+o(1))r^{2}}+(\ell+1)t.
\end{align*}

We proceed to estimating growth, where the values computed above replace the two pairs $C_{1},C_{2}$ and $C'_{1},C'_{2}$ inside \eqref{eq:de-clg} and \eqref{eq:de-tnonrs}. During the proof of Theorem~\ref{thm:growth}, unless $A^{3l}=G(K)$ already, we reach as before one among \eqref{eq:growthdouble}--\eqref{eq:oofhyp}--\eqref{eq:growthq}. The terms $\frac{1}{12r},\frac{1}{15r^{2}}$ of \eqref{eq:simpler1}--\eqref{eq:simpler2} are replaced by $\frac{1}{12r},\frac{5-o(1)}{24r^{2}}$, and the exponent $\frac{1}{20r}$ of \eqref{eq:growth22} is replaced by $\frac{1-o(1)}{12r}$.

Summing up the above, the final pairs $(c,m(l),\varepsilon)$ of Theorem~\ref{thm:growth} become
\begin{align*}
& \left(r^{(12+o(1))r^{2}},r^{(12+o(1))r^{2}}+2(\ell+1)l,\frac{1}{12r}\right), \\
& \left(r^{(12+o(1))r^{2}},r^{(12+o(1))r^{2}}+4l,\frac{5-o(1)}{24r^{2}}\right), \\
& \left(r^{(6+o(1))r^{2}},r^{(12+o(1))r^{2}}+4l,\frac{1-o(1)}{12r}\right).
\end{align*}
Hence, the proof of the Main Theorem works with the exponent $57r^{2}$ in \eqref{eq:diamm} replaced by $\left(\frac{48\log 4}{5}+o(1)\right)r^{2}$ and the quantity $c^{2/\varepsilon}$ bounded by $r^{(576/5+o(1))r^{4}}$. As for the simplified result, the condition on $q$ implies $\log|G(K)|\geq r^{3}$, and furthermore we can use $c^{(1+o(1))/\varepsilon}$ as a threshold for $|A|$; therefore, the final exponent becomes $\left(\frac{96}{5}+o(1)\right)r^{4}$.


\section*{Acknowledgements}
The authors would like to thank the Georg-August-Universit\"at G\"ottingen, Technische Universit\"at Dresden, Max Planck Institute f\"ur Mathematics, Bonn, and the Hebrew University of Jerusalem where much of the work on this article was accomplished, for their hospitality. The authors also thank V.~Finkelshtein, L.~Guan and the participants of the Mariaspring retreat for their suggestions and insights, and M.~Kassabov for his careful reading and valuable remarks.

JB is supported by ERC Consolidator grants 648329 (codename GRANT, with H. Helfgott as PI) and 681207 (codename GrDyAP, with A. Thom as PI). DD is supported by ERC Consolidator grant 648329, the Israel Science Foundation Grants No. 686/17 and 700/21 of A.~Shalev, and the Emily Erskine Endowment Fund; he has been a postdoc at the Hebrew University of Jerusalem under A.~Shalev in 2020/21 and 2021/22. HH is supported by ERC Consolidator grant 648329 and by his Humboldt professorship.

In addition, JB and HH would like to thank the organizers of the HIM, Bonn trimester program on ``Harmonic Analysis and Analytic Number Theory" during which an important part of the article was finalized. 
\nocite{}
\bibliographystyle{abbrv}
\bibliography{BDH}

\begin{thebibliography}{10}

\bibitem{Bab06}
L.~Babai.
\newblock On the diameter of {E}ulerian orientations of graphs.
\newblock In {\em Proceedings of the {S}eventeenth {A}nnual {ACM}-{SIAM}
  {S}ymposium on {D}iscrete {A}lgorithms}, pages 822--831. ACM, New York, 2006.

\bibitem{BS88}
L.~Babai and A.~Seress.
\newblock On the diameter of {C}ayley graphs of the symmetric group.
\newblock {\em J. Combin. Theory Ser. A}, 49(1):175--179, 1988.

\bibitem{BS92}
L.~Babai and A.~Seress.
\newblock On the diameter of permutation groups.
\newblock {\em European J. Combin.}, 13(4):231--243, 1992.

\bibitem{BDH24}
J.~Bajpai, D.~Dona, and H.~A. Helfgott.
\newblock New dimensional estimates for subvarieties of linear algebraic
  groups.
\newblock {\em Vietnam J. Math.}, 52(2):479--518, 2024.

\bibitem{BY17}
A.~Biswas and Y.~Yang.
\newblock A diameter bound for finite simple groups of large rank.
\newblock {\em J. Lond. Math. Soc. (2)}, 95(2):455--474, 2017.

\bibitem{Bor91}
A.~Borel.
\newblock {\em Linear Algebraic Groups}, volume 126 of {\em Graduate Texts in
  Mathematics}.
\newblock Springer, New York (USA), second enlarged edition, 1991.

\bibitem{BKT04}
J.~Bourgain, N.~Katz, and T.~Tao.
\newblock A sum-product estimate in finite fields, and applications.
\newblock {\em Geom. Funct. Anal.}, 14(1):27--57, 2004.

\bibitem{BBBKR}
M.~Brandt, J.~Bruce, T.~Brysiewicz, R.~Krone, and E.~Robeva.
\newblock The degree of {${\rm SO}(n,\mathbb{C})$}.
\newblock In {\em Combinatorial algebraic geometry}, volume~80 of {\em Fields
  Inst. Commun.}, pages 229--246. Fields Inst. Res. Math. Sci., Toronto, ON,
  2017.

\bibitem{BGT11}
E.~Breuillard, B.~Green, and T.~Tao.
\newblock Approximate subgroups of linear groups.
\newblock {\em Geom. Funct. Anal.}, 21(4):774--819, 2011.

\bibitem{DS98}
V.~I. Danilov and V.~N. Shokurov.
\newblock {\em Algebraic Geometry I: Algebraic Curves, Algebraic Manifolds and
  Schemes}, volume~23 of {\em Encyclopaedia of mathematical sciences}.
\newblock Springer-Verlag, Berlin, 1998.

\bibitem{DOR10}
J.-F. Dat, S.~Orlik, and M.~Rapoport.
\newblock {\em Period domains over finite and $p$-adic fields}, volume 183 of
  {\em Cambridge Tracts in Mathematics}.
\newblock Cambridge University Press, Cambridge (UK), 2010.

\bibitem{EMO05}
A.~Eskin, S.~Mozes, and H.~Oh.
\newblock On uniform exponential growth for linear groups.
\newblock {\em Invent. Math.}, 160(1):1--30, 2005.

\bibitem{Fre52}
H.~Freudenthal.
\newblock Elementarteilertheorie der komplexen orthogonalen und symplektischen
  {G}ruppen.
\newblock {\em Indag. Math. (Proc.)}, 55:199--201, 1952.

\bibitem{Ful84}
W.~Fulton.
\newblock {\em Intersection theory}, volume~2 of {\em Ergebnisse der Mathematik
  und ihrer Grenzgebiete : a series of modern surveys in mathematics. Folge 3}.
\newblock Springer, Berlin, 1984.

\bibitem{GL02}
S.~R. Ghorpade and G.~Lachaud.
\newblock Number of solutions of equations over finite fields and a conjecture
  of {L}ang and {W}eil.
\newblock In {\em Number theory and discrete mathematics ({C}handigarh, 2000)},
  Trends Math., pages 269--291. Birkh\"{a}user, Basel, 2002.

\bibitem{HMPQ19}
Z.~Halasi, A.~Mar\'{o}ti, L.~Pyber, and Y.~Qiao.
\newblock An improved diameter bound for finite simple groups of {L}ie type.
\newblock {\em Bull. Lond. Math. Soc.}, 51(4):645--657, 2019.

\bibitem{Hartshorne}
R.~Hartshorne.
\newblock {\em Algebraic geometry}.
\newblock Springer-Verlag, New York, 1977.
\newblock Graduate Texts in Mathematics, No. 52.

\bibitem{Hei83}
J.~Heintz.
\newblock Definability and fast quantifier elimination in algebraically closed
  fields.
\newblock {\em Theoret. Comput. Sci.}, 24(3):239--277, 1983.

\bibitem{Hel08}
H.~A. Helfgott.
\newblock Growth and generation in {$\mathrm{SL}_2(\mathbb{Z}/p\mathbb{Z})$}.
\newblock {\em Ann. of Math. (2)}, 167(2):601--623, 2008.

\bibitem{Hel11}
H.~A. Helfgott.
\newblock Growth in {$\mathrm{SL}_3(\mathbb{Z}/p\mathbb{Z})$}.
\newblock {\em J. Eur. Math. Soc. (JEMS)}, 13(3):761--851, 2011.

\bibitem{Helfgott-BAMS2015}
H.~A. Helfgott.
\newblock Growth in groups: ideas and perspectives.
\newblock {\em Bull. Amer. Math. Soc. (N.S.)}, 52(3):357--413, 2015.

\bibitem{Hel19b}
H.~A. Helfgott.
\newblock Growth and expansion in algebraic groups over finite fields.
\newblock In {\em Analytic methods in arithmetic geometry}, volume 740 of {\em
  Contemp. Math.}, pages 71--111. Amer. Math. Soc., [Providence], RI, 2019.

\bibitem{HS14}
H.~A. Helfgott and A.~Seress.
\newblock On the diameter of permutation groups.
\newblock {\em Ann. of Math. (2)}, 179(2):611--658, 2014.

\bibitem{Hum95a}
J.~E. Humphreys.
\newblock {\em Conjugacy classes in semisimple algebraic groups}, volume~43 of
  {\em Mathematical Surveys and Monographs}.
\newblock American Mathematical Society, Providence, RI, 1995.

\bibitem{Kir08}
A.~Kirillov.
\newblock {\em An introduction to {L}ie groups and {L}ie algebras}, volume 113
  of {\em Cambridge Studies in Advanced Mathematics}.
\newblock Cambridge University Press, Cambridge (UK), 2008.

\bibitem{KL90}
P.~Kleidman and M.~Liebeck.
\newblock {\em The subgroup structure of the finite classical groups}, volume
  129 of {\em London Mathematical Society Lecture Note Series}.
\newblock Cambridge University Press, Cambridge, 1990.

\bibitem{LR15}
G.~Lachaud and R.~Rolland.
\newblock On the number of points of algebraic sets over finite fields.
\newblock {\em J. Pure Appl. Algebra}, 219(11):5117--5136, 2015.

\bibitem{LS74}
V.~Landazuri and G.~M. Seitz.
\newblock On the minimal degrees of projective representations of the finite
  {C}hevalley groups.
\newblock {\em J. Algebra}, 32:418--443, 1974.

\bibitem{LW54}
S.~Lang and A.~Weil.
\newblock Number of points of varieties in finite fields.
\newblock {\em Amer. J. Math.}, 76:819--827, 1954.

\bibitem{LP11}
M.~J. Larsen and R.~Pink.
\newblock Finite subgroups of algebraic groups.
\newblock {\em J. Amer. Math. Soc.}, 24(4):1105--1158, 2011.

\bibitem{Leh92}
G.~I. Lehrer.
\newblock Rational tori, semisimple orbits and the topology of hyperplane
  complements.
\newblock {\em Comment. Math. Helv.}, 67:226--251, 1992.

\bibitem{LPR06}
N.~Lemire, V.~L. Popov, and Z.~Reichstein.
\newblock Cayley groups.
\newblock {\em J. Amer. Math. Soc.}, 19(4):921--967, 2006.

\bibitem{Mil17}
J.~S. Milne.
\newblock {\em Algebraic groups}, volume 170 of {\em Cambridge Studies in
  Advanced Mathematics}.
\newblock Cambridge University Press, Cambridge, 2017.

\bibitem{Mumford}
D.~Mumford.
\newblock {\em The red book of varieties and schemes}, volume 1358 of {\em
  Lecture Notes in Mathematics}.
\newblock Springer-Verlag, Berlin, expanded edition, 1999.
\newblock Includes the Michigan lectures (1974) on curves and their Jacobians,
  With contributions by Enrico Arbarello.

\bibitem{NP11}
N.~Nikolov and L.~Pyber.
\newblock Product decompositions of quasirandom groups and a {J}ordan type
  theorem.
\newblock {\em J. Eur. Math. Soc. (JEMS)}, 13(4):1063--1077, 2011.

\bibitem{Oes84}
J.~Oesterl\'e.
\newblock Nombres de {T}amagawa et groupes unipotentes en caract\'eristique
  $p$.
\newblock {\em Invent. Math.}, 78(1):13--88, 1984.
\newblock In French.

\bibitem{Ols84}
J.~E. Olson.
\newblock On the sum of two sets in a group.
\newblock {\em J. Number Theory}, 18(1):110--120, 1984.

\bibitem{PS16}
L.~Pyber and E.~Szab\'{o}.
\newblock Growth in finite simple groups of {L}ie type.
\newblock {\em J. Amer. Math. Soc.}, 29(1):95--146, 2016.

\bibitem{RS18}
M.~Rudnev and I.~D. Shkredov.
\newblock On growth rate in $\text{SL}_{2}(\mathbb{F}_{p})$, the affine group
  and sum-product type implications.
\newblock \texttt{arXiv:1812.01671}, 2018.

\bibitem{Sel67}
G.~B. Seligman.
\newblock {\em Modular {L}ie algebras}.
\newblock Ergebnisse der Mathematik und ihrer Grenzgebiete, Band 40.
  Springer-Verlag New York, Inc., New York, 1967.

\bibitem{Shafv1}
I.~R. Shafarevich.
\newblock {\em Basic algebraic geometry 1: Varieties in projective space}.
\newblock Springer, Heidelberg, third {R}ussian edition, 2013.

\bibitem{Shi19}
Y.~Shitov.
\newblock An improved bound for the lengths of matrix algebras.
\newblock {\em Algebra Number Theory}, 13(6):1501--1507, 2019.

\bibitem{SS70}
T.~A. Springer and R.~Steinberg.
\newblock Conjugacy classes.
\newblock In A.~Borel, R.~Carter, C.~W. Curtis, N.~Iwahori, T.~A. Springer, and
  R.~Steinberg, editors, {\em Seminar on algebraic groups and related finite
  groups}, volume 131 of {\em Lecture Notes in Mathematics}, pages 167--266.
  Springer-Verlag, Berlin (Germany), 1970.

\bibitem{TZ96}
P.~H. Tiep and A.~E. Zalesskii.
\newblock Minimal characters of the finite classical groups.
\newblock {\em Comm. Algebra}, 24(6):2093--2167, 1996.

\bibitem{zbMATH03880868}
W.~Vogel.
\newblock {\em Lectures on results on {Bezout}'s theorem. {Notes} by {D}. {P}.
  {Patil}}, volume~74 of {\em Lect. Math. Phys., Math., Tata Inst. Fundam.
  Res.}
\newblock Springer, Berlin; Tata Inst. of Fundamental Research, Bombay, 1984.

\bibitem{Wil09}
R.~A. Wilson.
\newblock {\em The Finite Simple Groups}, volume 251 of {\em Graduate Texts in
  Mathematics}.
\newblock Springer, London (UK), 2009.

\end{thebibliography}
\end{document}